\documentclass[11pt]{amsart}
\usepackage[margin=1.2in]{geometry}
\usepackage{xcolor}
\usepackage{graphicx} 
\usepackage{amsmath,amsthm,amsfonts,hyperref} 
\usepackage{amssymb}
\usepackage{enumerate}
\usepackage{amsthm,parskip}
\usepackage{comment}
\usepackage{mathtools} 
\usepackage{bm}
\usepackage{mathrsfs}
\DeclareMathOperator{\Arg}{Arg}
\theoremstyle{plain}
\newtheorem{theorem}{Theorem}[section]
\newtheorem{lemma}[theorem]{Lemma}
\newtheorem{corollary}[theorem]{Corollary}
\newtheorem{proposition}[theorem]{Proposition}
\theoremstyle{definition}
\newtheorem{definition}[theorem]{Definition}
\newtheorem{example}[theorem]{Example}

\newtheorem{remark}[theorem]{Remark}
\newcommand{\RE}{\text{Ritt}_{\text{E}}}

\newcommand{\calM}{{\mathcal{M}}}
\newcommand{\calN}{{\mathcal{N}}}

\newcommand{\abs}[1]{\Big\lvert #1 \Big\rvert}

\renewcommand{\le}{\leqslant}
\renewcommand{\ge}{\geqslant}

\title[Maximal ergodic theorems]%
{Maximal Ergodic Theorems for Operators with Finite Peripheral Spectrum }

 \author{Subhajit Palai}
\address{School of Mathematics, Indian Institute of Science Education and Research Thiruvananthapuram, Kerala - 695551}
\email{subhajit22@iisertvm.ac.in}

\author{Samya Kumar Ray}
\address{
The Institute of Mathematical Sciences, 4th Cross Street, CIT Campus, Tharamani, Chennai, Tamil Nadu 600113, India, and \newline
Homi Bhabha National Institute, Training School Complex, Anushakti Nagar, Mumbai 400094, India} \email{samya@imsc.res.in}

\begin{document}

 \begin{abstract}
Let $\mathcal M$ be a semifinite von Neumann algebra and $T : \mathcal{M} \to \mathcal{M}$ be a positive $L_\infty-L_1$ contraction in the sense of Junge-Xu, of which the numerical range, when viewed as an operator on $L_2(\mathcal M),$ is contained in a closed polygon with vertices on the unit circle. In this article, we prove that there exists a positive constant $C_p(T)$ such that
\begin{equation}\label{abstract1stin}
    \Big\|\sup_{n \ge 0}\!^{+} T^n x \Big\|_p \le C_p(T)\, \|x\|_p
\end{equation}
for all \( x \in L_p(\mathcal{M}) \), $1<p<\infty$ extending some noncommutative maximal ergodic inequalities proved by Junge-Xu \cite{junge-Xu} and later generalized by Bekjan \cite{Bekjan2008}. In the commutative setting, similar inequalities as in \eqref{abstract1stin} hold for arbitrary $L_\infty-L_1$ contractions with the same condition in the numerical range, yielding a vast generalization of a classical maximal ergodic theorem of Stein \cite{Stein-ergodic-theorem} proved in 1960s. Moreover, we establish a noncommutative weak-type maximal inequality for convolution powers which was proved by Calder\'on and Bellow \cite{Bellow-Calderon} in the classical setting, complementing our strong type noncommutative maximal ergodic inequalities. Our method relies on several new polynomial identities, suitable square function estimates tailored to fit our setting and generalization of Stein's method of embedding maximal function into analytic family of operators. However, we show that even in the classical setting, the variational inequality extending \eqref{abstract1stin} holds for arbitrary operators described above, precisely when the spectrum meets the unit circle only at $1.$
 \end{abstract}
\maketitle
\tableofcontents
\section{Introduction}
 The foundations of classical ergodic theory were established in the early 1930s by Birkhoff \cite{birkhoff1931proof} and von Neumann \cite{neumann1932proof} motivated by statistical properties of  dynamical systems. In particular Birkhoff \cite{birkhoff1931proof} proved the so-called `point-wise ergodic theorem' which says that 
 the `time averages' $M_nf:={\displaystyle \frac{1}{n+1}}\sum\limits^{n}_{i=0}T^if$ converge point-wise almost everywhere for all $f\in L_p(\Omega), 1\leq p<\infty$ to the conditional expectation of $f$ with respect to the invariant $\sigma$-algebra, often referred to as the “space average”. Here $T$ denotes the Koopman operator associated with a measure-preserving transformation. A well-known approach to proving point-wise ergodic theorems is to establish a corresponding maximal ergodic inequality for the ergodic averages, via the so-called Banach principle. This method goes back to Wiener \cite{Wiener-ergodic-theorem}, who showed that the associated maximal operator is of weak type $(1,1).$ In \cite{DunfordSchwartz} Dunford and Schwartz substantially generalized the preceding results by proving a weak type $(1,1)$ inequality for the maximal operator $M_{*}f:=\sup_{n\geq 1}|M_nf|$ for ergodic averages of positive $L_\infty-L_1$ contractions. As a consequence one can deduce a strong type $(p,p)$ maximal inequalities \begin{equation}\label{normal-type}
     \|M_{*}f\|_p\leq C_p\|f\|_p
 \end{equation} for all $1<p<\infty,$ where $C_p$ is a positive constant. In this direction the most general ergodic theorem was established by Akcoglu \cite{Akcogluergodicthm}, who proved a maximal ergodic inequality for general positive contraction on $L_p-$spaces for a fixed $1<p<\infty$. We refer to \cite{Ulrich} for a systematic study on measure theoretic ergodic theory.
 
Let $1<p<\infty.$ A natural problem is to characterise the class of operators for which the stronger maximal inequality
\begin{equation}\label{Stein-type}
\big\| \sup_{n\geq 0}|P^nf| \big\|_p \lesssim \|f\|_p,\quad \text{for all},\ f\in L_p(\Omega)
\end{equation}
holds. A fundamental advancement in answering the above question was made by Stein \cite{Stein-ergodic-theorem}, who obtained a maximal ergodic inequality as in \eqref{Stein-type} for the class of $L_\infty-L_{1}$ contractive operators which are also self-adjoint operators on $L_2(\Omega).$ His proof relied on certain Littlewood-Paley square function estimates, Abel summation methods and then embedding the time averages into a complex-analytic family of operators followed by a clever trick of linearization of the maximal operator, an approach which has been widely used in harmonic analysis and ergodic theory. For similar maximal inequalities in the vector-valued setting for both discrete and continuous semigroups and connection to Littlewood-Paley theory we refer to \cite{Xu2015}, \cite{Xu2020}, \cite{Xu2025} and references therein.

 Motivated by quantum mechanics  and operator algebra theory, non-commutative mathematics grew rapidly during the mid-20th century. The connection between ergodic theory and von Neumann algebras goes back as early as the theory of  ``rings of operators" studied by Murray and von Neumann. In a remarkable work Lance \cite{Lance} pioneered the study of noncommutative point-wise ergodic theorems, which was followed by a series of significant contributions due to Conze and Dang-Ngoc \cite{Conze},
K{\"u}mmerer \cite{Noncommutative-ergodic-theorem-kummmerer}, Yeadon \cite{Yeadon-Noncommutative-ergodic-theorem} and others. Extending the classical Dunford-Schwartz maximal ergodic inequality to the noncommutative setting proved to be an elusive, long-standing open problem, a major stumbling block being no notion of maximal function in the noncommutative setting. This difficulty was resolved by Pisier's theory of noncommutative vector-valued $L_p$-spaces developed for injective von Neumann algebras by Pisier \cite{Non-commutative-vector-valued-L-p-space-pisier} and later generalised by Junge \cite{Doob-inequality-for-non-commutative-martingales-Junge-Marius} for arbitrary von Neumann algebra, who proved a noncommutative analogue of Doob's martingale maximal inequalities. Shortly thereafter in a groundbreaking work, Junge and Xu \cite{junge-Xu} extended the Dunford–Schwartz maximal ergodic theorem to the noncommutative setting. In the same article, they also established an analogue of the maximal ergodic theorem \eqref{Stein-type} in the noncommutative setting generalizing Stein's \cite{Stein-ergodic-theorem} maximal ergodic theorem mentioned earlier in the introduction. Following this breakthrough, Bekjan \cite{Bekjan2008} established a noncommutative maximal ergodic theorem for $L_\infty-L_1$ contractions whose numerical range when viewed as an operator on $L_2$ is contained in a Stolz region with vertex at $1.$ This allows one to relax the so-called symmetry condition. We refer to \cite{Hong20201}, \cite{Hong2023}, \cite{ray2025noncommutative}, \cite{JungeLeMerdyXu06}, \cite{Hong2024}, \cite{Operatorvaluedhardyspacemei}, \cite{hong2025noncommutative}, \cite{ZeqianXuZhi}, \cite{HongWangWang},\cite{hong2026lampertioperatorsdilationtheory} for recent developments on noncommutative ergodic theory and harmonic analysis.

 In this article, we prove a vast generalization of Bekjan's \cite{Bekjan2008} result, which automatically generalises the maximal ergodic theorem \cite[Part (ii), Theorem, 0.1]{junge-Xu} proved by Junge-Xu. Let $\mathcal{M}$ be a semifinite von Neumann algebra equipped with a normal, semifinite, faithful trace $\tau$, and let $ L_p(\mathcal{M})$ be the associated noncommutative $ L_p$-space. Consider a linear map $T:\mathcal{M}\to\mathcal{M}$ which satisfies the following conditions:
\phantomsection
\label{thosefourconditions}
\begin{itemize}    \item[(i)] $T$ is a contraction on $\mathcal{M}$, that is
    $ \|T x\|_\infty \leq \|x\|_\infty \quad \text{for all } x \in \mathcal{M}.$
    \item[(ii)] $T$ is positive, that is $Tx \geq 0$ if $x \ge 0.$
    \item[(iii)] $\tau(T(x)) \leq \tau(x) \quad \text{for all } x \in L_1(\mathcal{M}) \cap \mathcal{M}_+.$
    \item[(iv)] There is a closed polygon $\Delta\equiv\Delta_T\subseteq\overline{\mathbb{D}}$ such that $\Delta\cap \mathbb T$ is  a finite set and numerical range of $T$ as an operator on $L_2(\calM)$ is contained in ${\Delta}.$
\end{itemize}
If $T$ satisfies the properties (i)-(iii) of \eqref{thosefourconditions}, then $T$ naturally extends to a contraction on \( L_p(\mathcal{M}) \) for all \( 1 \le p\le\infty \) \cite{junge-Xu}. The extension will still be denoted by \( T \). We prove the following theorem.
\begin{theorem}\label{strong-ergodic-theorem-for-ritt-E}
   Let $1<p<\infty$ and $T:\mathcal M\to\mathcal M$ be such that it satisfies (i)-(iv) as above in \eqref{thosefourconditions}. Then there exists a positive constant $C_p(T)$ such that 
   \begin{equation}
       \|\sup_{n\geq 0}\!^{+}T^nx\|_p\leq C_p(T)\|x\|_p
   \end{equation}
   for all $x\in L_p(\mathcal M).$
\end{theorem}
In the classical setting, the proof of the above theorem works for general $L_\infty-L_1$ contractions satisfying (iv) in \eqref{thosefourconditions}. Consequently, our result is new even in the commutative setting and significantly generalizes Stein's maximal ergodic theorem \cite{Stein-ergodic-theorem}. We refer Section \eqref{preliminaries} for any unexplained notation or terminologies used in the introduction. The case when $\Delta\cap\mathbb{T}$ has cardinality one, can be also recovered from the works we have mentioned earleir. Hence we always assume that cardinality of $\Delta\cap\mathbb{T}$ is at least two.

 The absence of notion of maximal functions, together with the inherent noncommutativity of the setting, renders several steps of our argument considerably more involved. To overcome these difficulties, we establish a number of preparatory lemmas and propositions. One useful technique is of course Junge-Xu's \cite{junge-Xu} method of leveraging duality between noncommutative vector-valued $L_1$ and $L_\infty$-space. This is crucial at several places of our work. Apart from these essential obstacles, the proof of Theorem \eqref{strong-ergodic-theorem-for-ritt-E} also relies on several key ideas and intermediate steps, which we briefly summarize below.

\begin{itemize}
\item[(i)] The first key step is to select the appropriate square function. If $E\equiv E_T:=\Delta_T \cap \mathbb{T}=\{\xi_1,\dots,\xi_N\}$, then the first order square function one exploits is \[x\mapsto\Big(\sum_{k=0}^\infty(k+1)\|T^k\prod_{j=1}^N(T-\xi_jI)x\|_2^2\Big)^{\frac{1}{2}}.\] If $E=\{1\}$ we can get back to the usual square function exploited by Stein \cite{Stein-ergodic-theorem}, Junge-Xu \cite{junge-Xu} and Bekjan \cite{Bekjan2008}. When $E=\{1\}$ one proves the required $L_2$-maximal theorem by an Abel's summation method combining well-known maximal ergodic theorem for ergodic averages for example noncommutative Dunford-Schwartz ergodic theorem \eqref{JungeXu} or Akcoglu's \cite{Akcogluergodicthm} (in the classical setting). However, in our general setting, the situation becomes more delicate. By using square function estimate and maximal ergodic theorem for ergodic averages we first manage to prove a maximal ergodic theorem of the form \[\|\sup_{n\geq 0}\!^{+}T^n{Q}_{\hat{E_i}}(T)x\|_2\lesssim\|x\|_2\] for all $x\in L_2(\mathcal M)$ where ${Q}_{\hat{E_i}}(T):=\prod\limits_{j=1,j\neq i}^N(T-\xi_jI)$ for $1\leq i\leq N.$ Then we get the desired maximal $L_2$-inequality by Lagrange interpolation.
    \item[(ii)] The second step is based on proving the maximal inequality of the form \[\|\sup_{k\geq 0}\!^{+}(k+1)^mT^k\prod_{j=1}^N(T-\xi_j I)^mx\|_2\lesssim_m\|x\|_2\] for all $x\in L_2(\mathcal M).$ The heart of the argument is a clever use of mathematical induction, in the spirit of \cite{Stein-ergodic-theorem}, to dominate the quantity
\[
\left\|\frac{1}{n+1}\sum_{k=0}^n (k+1)^m T^k
\prod_{j=1}^N (T-\xi_j I)^m x\right\|_2
\]
in terms of higher-order square functions together with the estimates obtained at the preceding inductive steps. The implementation of this induction, however, relies on nontrivial algebraic computations. 
    \item[(iii)] In the third major step of the proof one needs to interpolate between $L_2$-maximal inequality in step (i) and maximal ergodic theorem for ergodic averages. To this end one requires to identify suitable interpolating family $S_n^\lambda$ as in subsection \eqref{section3.2}. After this one needs to have good estimates of the coefficients of $T^k$ in $S_n^\lambda$. We do this in Lemma \eqref{final-estimates-of-A-k-lambda} with some techniques used in analytic combinatorics \cite{Flajolet1990SingularityAO}. 
    \item[(iv)] In the final stage, we need to implement Abel's summation method as in Stein \cite{Stein-ergodic-theorem}, Bekjan \cite{Bekjan2008} and \cite{junge-Xu}. The crucial insight is to view $T^{k-N}\prod_{j=1}^N(T-\xi_j)$ as a discrete differential operator, which is actually the discrete difference operator itself when $E=\{1\}$ for $k\geq N.$ We can then use Abel's summation method repeatedly to prove Theorem \eqref{firsttimerealiseddiofferentialop}. The proof of the main theorem can then be completed using complex interpolation.
\end{itemize}
It is well-known that the class of contractively regular analytic operators already covers all $L_\infty-L_1$ contractions which are positive as an operator on $L_2$ by a remarkable extrapolation result of \cite[Theorem 1.1]{Blunk}. Moreover, for contractively regular analytic operators, Le Merdy and Xu proved a maximal ergodic theorem as in \eqref{Stein-type}, thus extending Stein's maximal ergodic theorem \cite{Stein-ergodic-theorem}. In Theorem \eqref{maximal-inequality-classical-L-p}, we extend their result to contractively regular operators on classical $L_p$-spaces which are also $Ritt_E$ (see Section \eqref{section5} for definition) for $1<p<\infty.$ Moreover, we prove an extrapolation principle for $\RE$
operators on interpolation of Banach spaces (see Proposition \eqref{Blunck-RE-interpolation}), which indeed shows that Theorem \eqref{maximal-inequality-classical-L-p} is a generalization of Theorem \eqref{strong-ergodic-theorem-for-ritt-E} in the setting of classical $L_p$-spaces, for $1<p<\infty.$

To complement our results for $p=1$, we turn our attention to weak type $(1,1)$ analogue of maximal ergodic theorem of \eqref{Stein-type}. In classical setting Calder\'on and Bellow \cite{Bellow-Calderon} obtained such weak type $(1,1)$-inequalities for convolution powers. It is well-known that such operators satisfy conditions as in \eqref{thosefourconditions} (see \cite{CunnyRttWeaktype}) .
\begin{theorem}\label{thmonconvpower}
    Let $\mu$ be a probability measure on $\mathbb{Z}$ which is strictly aperiodic, has expectation zero, that is, $\mathbb{E}(\mu):= \sum_{k\in\mathbb{Z}}k\mu(k)=0$, and finite second moment, that is, $\sum_{k\in\mathbb{Z}}k^2\mu(k)<\infty$. Let $Tf=\mu*f$ for $f\in L_1(\ell_\infty(\mathbb{Z})\overline{\otimes}\mathcal M).$ Then $(T^n)_{n\geq 0}$ is weak type $(1,1),$ that is, there exists a positive constant $C>0$ such that for any $f\in L_1(\ell_\infty(\mathbb{Z})\overline{\otimes}\mathcal M)$ and $\lambda>0$, there exists a projection $e\in\ell_\infty(\mathbb{Z})\overline{\otimes}\mathcal M$ such that
\[
\sup_{n\geq 0} \, \|\, e \, T^{n}(f) \, e \,\|_{\infty} \le \lambda
\quad \text{and} \quad
\varphi(e^\perp) \le C \, \lambda^{-1} \|f\|_{1},
\] where $\varphi:=\int d\mu\otimes \tau.$
\end{theorem}
The above theorem is a consequence of Theorem \ref{thm5.4} and \cite[Corollary 3.4, 3.5]{Bellow-Calderon}. Our proof of Theorem \eqref{thm5.4} relies on noncommutative Calder\'on-Zygmund decomposition developed recently in \cite{Cadilhac-Conde-Alonso-Parcet-2022}. 
In the classical setting, one can apply the Calder\'on--Zygmund decomposition directly to the associated maximal function. However, in the noncommutative setting there is no noncommutative analogue of the maximal function, so this approach cannot be used directly. To overcome this difficulty, we employ the method of Hong-Lai-Xu \cite{Bang-Xu}, who developed a framework to handle noncommutative weak type $(1,1)$ maximal inequalities for truncated Calder\'on-Zygmund operators with again a crucial use of Abel's summation method.

In \cite{le-Merdy-Xu-q-variational-inequality}, the authors established a variational inequality extending \eqref{Stein-type} for contractively regular analytic operators on classical $L_p$-space for $1<p<\infty$, further strengthening Stein's \cite{Stein-ergodic-theorem} classical result. Bourgain \cite{varitaionalarthimeticsetBourgain,Bourgain} was the first to consider variational inequality in ergodic theory. Since then a substantial body of work has demonstrated strong variational inequalities for a wide range of classical operator sequences and semigroups. We refer to \cite{varitaionalarthimeticsetBourgain}, \cite{TaoTorreaXuvariation}, \cite{MirekSteinTrojanvariation}, \cite{pisierXuMartinglevariation}, \cite{Hong-Tao-Ma}, \cite{Hong2016}, \cite{Hong2017}, \cite{Hong2020} and the references therein for more on this. In the present work, we investigate the corresponding variational problem for contractively regular $\RE$ operators. We observe that the variational inequality holds for arbitrary contractively regular $\RE$ operator on classical $L_p$-spaces if and only if $E=\{1\}$.

We conclude the introduction with a brief outline of the organization of the paper. In Section \eqref{preliminaries}, we review the basic notions of noncommutative $L_p$-spaces and noncommutative vector-valued $L_p$-spaces, together with several auxiliary results that will be used throughout the paper. Section \eqref{section3} is devoted to the proof of Theorem \eqref{strong-ergodic-theorem-for-ritt-E}. Section \eqref{kiarboli} is devoted to the spectral obstructions to the variational inequalities and the extrapolation result for $\RE$ operators. In Section \eqref{sectionczo} we prove Theorem \eqref{thmonconvpower}. In Section \eqref{individualergodicthm} we prove individual ergodic theorems.

 \section{Preliminaries}\label{preliminaries}
 Let $X$ be a Banach space. We denote $B(X)$ as the set of all bounded linear transformations from $X$ to $X$. Let us denote $\sigma(T)$ the spectrum of $T.$ For $a,b>0$,
 we write $a\lesssim b$ if $a\leq cb$, where $c>0$ independent of $a$ and $b$. Also we define $a\lesssim_m b$ as $a\le c_mb$ where $c_m$ is a constant depending on $m$. Often, a positive numerical constant $C$ may vary from line to line. Let $(\Omega, \mathcal{F}, \mu)$ be a measure space and let $1 \leq p \leq \infty$. We denote $L_p(\Omega, X)$ to be the usual Bochner space. 
Let us denote $\mathbb{C}[z]$ as the set of one-variable polynomials with complex coefficients. For a complex number $z,$ $\Re z$ and $\Im z$ denote the real and imaginary parts of a complex number $z$ respectively. We denote $\mathbb D:=\{z\in\mathbb{C}:|z|< 1\}$ and $\mathbb T:=\{z\in\mathbb C:|z|=1\}.$
\subsection{Noncommutative $L_p$-space.}
Let $\mathcal{M}$ be a von Neumann algebra equipped with a normal semifinite faithful trace $\tau$ acting on a Hilbert space $\mathcal H.$ Let us define $s(x)$ as the support projection of $x\in\mathcal{M_{+}}$ where $\mathcal{M_{+}}$ is the positive part of $\mathcal{M}.$ 
We define noncommutative $L_p$-space $L_p(\mathcal{M})$ to be completion of $\mathcal{M}\cap\{x\in\calM:\tau(|x|^p)<\infty\},$ with respect to the norm $\lVert x\rVert_{L_p(\mathcal{M})}=(\tau(\lvert x\rvert^p))^{1/p}$ where $\lvert x\rvert=(x^*x)^\frac{1}{2}$ and $1\le p<\infty.$ We set $L_{\infty}(\mathcal{M})=\mathcal{M}$ with respect to operator norm. Let $L_0(\mathcal{M})$ be the $*$-algebra of all closed densely defined operators on $\mathcal{H}$ measurable with respect to $(\mathcal{M}, \tau)$. For a subspace $A \subseteq L_0(\mathcal{M})$, we denote by $A_+$ the cone of positive elements in $A.$ It is well known that $L_p(\mathcal{M})$ can be viewed as a subspace of $L_0(\mathcal{M})$. A linear map $T : L_p(\mathcal{M}) \to L_p(\mathcal{M})$ is said to be positive if $T\big(L_p(\mathcal{M})_+\big) \subseteq L_p(\mathcal{M})_+.$ We refer to \cite{PisierXuNclp} for a comprehensive study on noncommutative $L_p$-spaces. For a projection $e\in\calM$ define $e^{\perp}:=1-e.$
\subsection{Noncommutative vector-valued $L_p$-spaces} Let $1\leq p\leq\infty.$ Let us denote $L_{p}(\mathcal{M},\ell_{\infty})$ as the space of  all sequences $x=(x_n)_n$ which admit the following factorization condition: there exist $a,b \in L_{2p}(\mathcal{M})$ and $(y_n)_{n\geq1}\subseteq\mathcal{M}$ with $\sup_{n\geq 1}\|y_n\|_\infty<\infty$ such that $x_n=ay_nb$ for $n\geq1.$ One defines 
 \begin{align*}
     \lVert (x_n)_{n\geq1}\rVert_{L_p(\mathcal{M};\ell_{\infty})}:=\inf\{\lVert a\rVert_{2p}\sup_{n\geq1}\lVert y_n\rVert_{\infty}\lVert b\rVert_{2p}\}
 \end{align*}
 where the infimum is taken over all possible factorization. Equipped with the above norm $L_p(\mathcal M;\ell_\infty)$ becomes a Banach space \cite{junge-Xu}. Let us denote \[\lVert x\rVert_{L_p(\mathcal M;\ell_{\infty})}:=\| \sup\limits_{n\geq 1}\!^{+}x_n\|_p.\] The following proposition will be used several times in this article. We denote $\mathbb{N}_0:=\mathbb{N}\cup\{0\}.$ 
\begin{proposition}\label{supnormfiniting}\cite{junge-Xu}
Let $1 \leq p \leq \infty$. A sequence $(x_n)_{n \geq 0} \subseteq L_p(\mathcal M)$ belongs to
$L_p(\mathcal M;\ell_\infty)$ if and only if
\[
\sup_{\substack{J \subseteq \mathbb N_0\\ J \text{ finite}}}
\|(x_i)_{i\in J}\|_{L_p(\mathcal M;\ell_\infty)}
<\infty .
\]
Moreover, we have $\|(x_n)_{n\geq 0}\|_{L_p(\mathcal M;\ell_\infty)}
=\sup\limits_{\substack{J \subseteq \mathbb N_0\\ J \text{ finite}}}
\|(x_i)_{i\in J}\|_{L_p(\mathcal M;\ell_\infty)}.$
\end{proposition}
A simpler description of the above noncommutative vector-valued $L_p$-space is available as the following, which can be found in \cite{ZeqianXuZhi}. \begin{remark}\label{selfadjoindomination}
Let $1\leq p\leq \infty.$ Let $(x_n)_{n\geq 1}\subseteq L_p(\mathcal M)$ be such that $x_n=x_n^*$ for all $n\geq 0$. Then
$
(x_n)_{n\geq 1} \in L_p(\mathcal M;\ell_\infty)$
if and only if there exists $a\in L_p(\mathcal M)_{+}$ such that
$
-a \leq x_n \leq a,$ for all $n\geq 1.$
Moreover, we have
\[
\|\sup_{n\geq 1}\!^{+} x_n\|_p
=
\inf \{
\|a\|_p :
a\in L_p(\mathcal M),\,
-a \leq x_n \leq a,\ \forall\, n\geq 1\}.
\]
\end{remark} 
We need following elementary lemmas.
\begin{lemma}\cite{junge-Xu}\label{maxcomingoutjuxulem}
Let $1\leq p\leq\infty.$ Let $x=(x_n)\in L_p(\mathcal{M};\ell_\infty)$ and let $(z_{n,k})\subset \mathbb{C}$. Then
\[
\Big\|\sup_{n}\!^{+}\sum_{k} z_{n,k}x_k\Big\|_{p}
\leq
\sup_{n}\Big(\sum_{k}|z_{n,k}|\Big)
\Bigl\|\sup_{k}\!^{+} x_k\Big\|_{p}.
\]
\end{lemma}
A simple corollary of the above lemma is the following. We skip the proofs.
\begin{corollary}\label{mod1andmaximum}
  Let $x=(x_n)\in L_p(\mathcal{M};\ell_\infty)$ and $(\alpha_n)\subseteq \mathbb C.$ Then 
  \begin{itemize}
      \item[(i)] $\|\sup\limits_{n}\!^{+}\alpha_nx_n\|_p\leq \Big(\sup\limits_{n}|\alpha_n|\Big)\|\sup\limits_{n}\!^{+}x_n\|_p.$
      \item[(ii)] If $|\alpha_n|=1$ for all $n$, then  $\|\sup\limits_{n}\!^{+}\alpha_nx_n\|_p=\|\sup\limits_{n}\!^{+}x_n\|_p.$
  \end{itemize}
\end{corollary}
We define $L_p(\mathcal{M}; \ell_1)$ to be the space of all sequences 
$x:=(x_n)_{n \ge 1} \subset L_p(\mathcal{M})$ which admits a decomposition
\[
x_n = \sum_{k \ge 1} u_{kn}^* v_{kn}
\]
for all $n \ge 1$, where $(u_{kn})_{k,n \ge 1}$ and $(v_{kn})_{k,n \ge 1}$ are two families in 
$L_{2p}(\mathcal{M})$ such that
\[
\sum_{k,n \ge 1} u_{kn}^* u_{kn} \in L_p(\mathcal{M}), 
\qquad \text{and}\qquad
\sum_{k,n \ge 1} v_{kn}^* v_{kn} \in L_p(\mathcal{M}).
\]

In above all the series are required to converge in $L_p$-norm for $1\leq p<\infty$ and in $\text{w}^*$ for $p=\infty.$ We equip the space 
$L_p(\mathcal{M}; \ell_1)$ with the norm
\[
\|x\|_{L_p(\mathcal{M}; \ell_1)} 
=: \inf \Big\{ 
\Big\| \Big( \sum_{k,n \geq 1} u_{kn}^* u_{kn} \Big)^{1/2} \Big\|_{2p}
\Big\|\Big( \sum_{k,n \ge 1} v_{kn}^* v_{kn} \Big)^{1/2}\Big\|_{2p}
\Big\},
\]
where infimum runs over all possible decompositions of $x$ described as above. 
For any positive sequence $x = (x_n)_{n \ge 1} \in L_p(\mathcal{M}; \ell_1)$ we have a simpler 
description of the norm as follows.
\[
\|x\|_{L_p(\mathcal{M}; \ell_1)} 
= \Big\| \sum_{n \ge 1} x_n \Big\|_p.
\]

It is known that like $L_p(\mathcal{M}; \ell_\infty)$, the space $L_p(\mathcal{M}; \ell_1)$ is also a Banach space \cite{junge-Xu}. Moreover, we have the following duality relationship.
\begin{proposition}{\cite{Doob-inequality-for-non-commutative-martingales-Junge-Marius}}\label{dualityofvectorvaluednorm}
Let $1\leq p < \infty$ and let $p'$ be such that
$
\frac{1}{p} + \frac{1}{p'} = 1.$
Then, $L_p(\mathcal{M}; \ell_1)^*$ is isometrically isomorphic to $L_{p'}(\mathcal{M}; \ell_\infty),$
with the duality relation given by
\[
\langle x, y \rangle = \sum_{n \ge 1} \tau(x_n y_n),
\]
for all $x \in L_p(\mathcal{M}; \ell_1)$ and $y\in L_{p'}(\mathcal{M};\ell_{\infty}).$
\end{proposition}
We need the following lemma which can be found in \cite{junge-Xu} and \cite{Doob-inequality-for-non-commutative-martingales-Junge-Marius}.
\begin{lemma}\label{sumofsixoreight}
Let $1<p<\infty.$ Then for any $(x_n)\subseteq L_p(\mathcal{M};\ell_\infty)$,
\[
\Big\|\sup_n\!^{+} x_n\Big\|_{p}
\simeq
\sup \Big\{
\Big|
\sum_n \tau(x_n y_n)
\Big|
:
\, y_n\in L_{p^\prime}(\mathcal{M}),\ y_n\ge 0,\ 
\Bigl\|\sum_n y_n\Bigr\|_{p^\prime}\le 1
\Big\},
\]
with universal equivalence constants.
\end{lemma}
Now we are going to state the following interpolation theorem for vector-valued noncommutative $L_p$-space.
\begin{proposition}\cite{junge-Xu}\label{Interpolation}
Let $1\le p_0<p_1\le \infty$ and $0<\theta<1$. Then we have isometrically
\[
L_p(\mathcal{M};\ell_1)
=
\bigl(L_{p_0}(\mathcal{M};\ell_1),
L_{p_1}(\mathcal{M};\ell_1)\bigr)_\theta,
\qquad
L_p(\mathcal{M};\ell_\infty)
=
\bigl(L_{p_0}(\mathcal{M};\ell_\infty),
L_{p_1}(\mathcal{M};\ell_\infty)\bigr)_\theta,
\]
where $\displaystyle{\frac{1}{p}
=
\frac{1-\theta}{p_0}
+
\frac{\theta}{p_1}}.$
\end{proposition}

\begin{lemma}\cite{hong2025noncommutative}\label{opHol}
Let $(\Omega,\mu)$ be a measure space and let $1<p<\infty$. Then
\begin{equation}\label{eq:holder}
\int_{\Omega} f(x)g(x)\, d\mu(x)
\;\leq\;
\Big( \int_{\Omega} f(x)^p \, d\mu(x) \Big)^{\frac{1}{p}}
\Big( \int_{\Omega} g(x)^{p^\prime} \, d\mu(x) \Big)^{\frac{1}{p^{\prime}}},
\end{equation}
where $f : \Omega \to L_1(\mathcal{M}) + L_\infty(\mathcal{M})$ and 
$g : \Omega \to \mathbb{C}$ are positive functions such that all terms in \eqref{eq:holder} are well-defined.
\end{lemma}
\begin{proposition}\label{sqtomx}
Let $(x_k)_{k\geq 1}\subseteq L_2(\mathcal M)$ and $(b_{nk})_{n,k\geq 1}\subseteq\mathbb C.$ Then we have
    \[
\left\|
\left(\sum_{k\geq1}b_{n,k}x_k\right)_{n\geq1}
\right\|_{L_2(\mathcal M;\ell_\infty)}
\leq 16
\sup_{n\geq1}
\left(\sum_{k\geq1}|b_{n,k}|^2\right)^{1/2}
\left(\sum_{k\geq1}\|x_k\|_2^2\right)^{1/2}.
\]
\end{proposition}
\begin{proof}
Note that it is enough to prove for any nonnegative family $(b_{nk})_{n,k\geq 1}$ and $(x_k)_{k\geq 1}\subseteq L_2(\calM)_{+}$ we have 
 \[
\left\|
\left(\sum_{k\geq1}b_{n,k}x_k\right)_{n\geq1}
\right\|_{L_2(\mathcal M;\ell_\infty)}
\leq
\sup_{n\geq1}
\left(\sum_{k\geq1}|b_{n,k}|^2\right)^{1/2}
\left(\sum_{k\geq1}\|x_k\|_2^2\right)^{1/2}.
\] which follows easily by Lemma \eqref{opHol} and Remark \eqref{selfadjoindomination}.
\end{proof}
\subsection{Junge-Xu ergodic theorem}Let $T:\mathcal{M}\rightarrow \mathcal{M}$ be a linear map which satisfies the following properties
\begin{itemize}\label{junge-xuoperator}
    \item[(i)] $T$ is contraction on $\mathcal{M}: $$\lVert Tx \rVert_{\infty}\leq \lVert x\rVert_{\infty}$ for all $x\in\mathcal{M}$
    \item[(ii)] $T$ is positive: $Tx\geq0$ if $x\geq0.$
    \item[(iii)] $\tau\circ T\leq\tau:\tau(T(x))\leq\tau(x)$ for all $x\in L_1(\mathcal{M})\cap \mathcal{M_{+}}.$
\end{itemize}
\begin{theorem}\label{JungeXu}\cite{junge-Xu}
Let $T$ be a linear map with $(i)$--$(iii)$. Let
\[
M_n\equiv M_n(T) = \frac{1}{n+1} \sum_{k=0}^{n} T^k.
\]
Then for every $1 < p \leq \infty$ we have
\begin{equation}
\| \sup_{n}\!^{+} M_n(x)\|_p \leq C_p\|x\|_p,
\quad \forall\, x \in L_p(\mathcal{M}).
\end{equation}
\end{theorem}
We record the following elementary lemma.
\begin{lemma}\label{saclarmultijungexu}
    Let $\alpha\in\mathbb T$ and $T:\mathcal M\to\mathcal M$ be a linear map satisfying $(i)$--$(iii)$ in \eqref{junge-xuoperator}. Then for all $1<p<\infty$ we have 
    \begin{equation}\label{alphaT}
\Big\| \sup_{n}\!^{+} M_n(\alpha T)(x) \Big\|_p \leq C_p\|x\|_p,
\quad \forall\, x \in L_p(\mathcal{M}).
\end{equation}
\end{lemma}
\begin{proof}
    Note that by \cite[Part (i), Proposition 2.1]{junge-Xu} it is enough to prove \eqref{alphaT} for $x\in L_p(\mathcal M)_{+}.$ Then by Lemma \eqref{sumofsixoreight} we obtain that 
\[
\begin{aligned}
\Bigl\|\sup_n\!^{+} M_n(\alpha T)x\Bigr\|_{p}
&\lesssim
\sup_{\substack{(y_n)\subseteq L_{p'}(\mathcal M),\, y_n\ge0\\
\|\sum_n y_n\|_{L_{p'}}\le1}}
\Big|
\sum_n \frac1{n+1}\sum_{k=0}^n
\alpha^k\tau(T^kxy_n)
\Big|
\\
&\le
\sup_{\substack{(y_n)\subseteq L_{p'}(\mathcal M),\, y_n\ge0\\
\|\sum_n y_n\|_{L_{p'}}\le1}}
\sum_n \frac1{n+1}\sum_{k=0}^n
\bigl|\tau(T^kxy_n)\bigr|\lesssim\|x\|_p.
\end{aligned}
\]
where we used the triangle inequality and the fact that $|\alpha|=1$ from second to third inequality and Junge-Xu ergodic theorem Theorem \eqref{JungeXu} and Lemma \eqref{sumofsixoreight} again in the final inequality. This completes the proof of the lemma.
\end{proof}
\section{Maximal ergodic theorem for $L_\infty-L_1$ contractions with finite peripheral spectrum on noncommutative $L_p$-spaces}\label{section3}
In this section we prove Theorem \eqref{strong-ergodic-theorem-for-ritt-E}. We start with the following essential square function estimate.
\subsection{Square function estimates and $L_2$-maximal ergodic theorem}
\begin{lemma}\label{lemma-for-T-to-use-square-functio-estimates}[\cite[Theorem 3]{Delyon-Bernard}]
Let $\mathcal H$ be a Hilbert space and $T\in B(\mathcal{H})$. Suppose $\Omega\subseteq\mathbb{C}$ is a bounded convex set  which contains the numerical range of $T$. Then there exists a constant $D_{\Omega}$ such that for any finite sequence of rational functions $v_1,\dots,v_n$ we have 
\begin{align*}
\Big\lVert \sum^n_{i=1} v_i(T)^*v_i(T)\Big\rVert\leq D^2_{\Omega}\sup_{z\in\Omega}\sum^n_{i=1}\abs{v_i(z)}^2
\end{align*}
\end{lemma}
\begin{definition}(Stolz domain)\label{definition-Stolz-domain}
For any $\gamma\in(0,\frac{\pi}{2})$, let $B_{\gamma}\subseteq\mathbb{C}$ be the open Stolz domain defined as interior of convex hull of $1$ and the disc $\mathbb{D}(0,\sin\gamma).$
\end{definition}
Let $\{\xi_1,\dots,\xi_N\}\subseteq\mathbb{T}$, consider $E=\Delta \cap \mathbb{T}=\{\xi_1,\dots,\xi_N\}$. We denote the interior of convex hull of $D(0,r)\cup E$ by $E_r$, $0<r<1$.
We now state the following lemma, which can be found in \cite[Lemma 2.4]{Bekjan2008}.
\begin{lemma}\label{lemma-for-finiteness-of-series}
Let $B_{\gamma}$ be a Stolz domain defined as in Definition \eqref{definition-Stolz-domain}. Then there exists $C_{m,\gamma}>0$ such that
\begin{align}
 \sup_{\lambda\in \overline{B_{\gamma}}}\sum_{n\geq 0}(n+1)^{2m-1}|\lambda^n(\lambda-1)^{m}|^2<C_{m,\gamma}.  
\end{align}
\end{lemma}
 \begin{lemma}\label{L2sqrfnestimatewithnumercond}
    Let us consider $T\in B(H)$ where $H$ is a Hilbert space such that the numerical range of $T$ contained in $\overline{E}_r$ for some $0<r<1$. Then for all $m\geq 1$ there exists $C_{m,r}>0$ such that 
    \begin{equation}
         \sum^{\infty}_{k=0}(k+1)^{2m-1}\Big\lVert \Theta^m_k(x)\Big\rVert^2_{H} \leq C_{m,r}\|x\|^2_{H}, \quad \text{where}\quad \Theta^m_k\equiv\Theta^m_k(T):=T^k\prod\limits^N_{j=1}(T-{\xi_j}I)^m.
    \end{equation}
\end{lemma}
\begin{proof}
    By Lemma \eqref{lemma-for-T-to-use-square-functio-estimates}, for all $n\geq0$ one can see that 
    \begin{align*}
    \Big\lVert \sum^n_{k=0}(k+1)^{2m-1}(\Theta^m_k)^*(\Theta^m_k)\Big\rVert\leq C_r\sup_{\lambda\in \overline{E_r}}\sum\limits^{\infty}_{k=0}(k+1)^{2m-1}|\lambda^k\prod_{j=1}^N(\lambda-\xi_j)^{m}|^2.
    \end{align*}
Let us consider neighbourhoods of $\xi_j$ denoted  by $U_j$ so that $\overline{U_j}$ are mutually disjoint for $1\leq j\leq N.$ Also choose $\epsilon>0$ such that  we have $0<s<1-\epsilon$ and $\overline{E_r}\subseteq(\cup_{j=1}^N\overline{U_j})\cup \overline{B(0,s)}.$ Note that 
\begin{equation}\label{some-series-of-epsilon}
    \sup_{\lambda\in \overline{B(0,s)}}\sum_{k\geq 0}(k+1)^{2m-1}|\lambda^k\prod_{j=1}^N(\lambda-\xi_j)^{m}|^2\leq 2^{2mN}\sum_{k\geq 0}(k+1)^{2m-1}(1-\epsilon)^{2k}<\infty.
\end{equation}
The series in the right hand side of \eqref{some-series-of-epsilon} converges because of ratio test. Notice that $\overline{\xi_j}(E_r\cap U_j)$ is contained in a Stolz domain with vertex at $1$, call it $S_j$ for all $1\leq j\leq N.$ Now using Lemma \eqref{lemma-for-finiteness-of-series}, we have
\begin{equation}
    \sup_{\lambda\in \overline{E_r}\cap \overline{U_j}}\sum_{k\geq 0}(k+1)^{2m-1}|\lambda^k\prod_{j=1}^N(\lambda-\xi_j)^{m}|^2\lesssim_{m,r}\sup_{\lambda\in \overline{S_j}}\sum_{k\geq 0}(k+1)^{2m-1}|\lambda^k(\lambda-1)^{m}|^2<\infty.
\end{equation}
  Therefore, we obtain \begin{align}\label{equation-for-delta-k-m}
    \Big\lVert \sum^n_{k=0}(k+1)^{2m-1}(\Theta^m_k)^*(\Theta^m_k)\Big\rVert\leq C_{m,r}
    \end{align}
    for all $n\in\mathbb{N}.$
    By the inequality in \eqref{equation-for-delta-k-m}, we have for all $n\in\mathbb N$
    \begin{align*}
        \sum^{n}_{k=0}(k+1)^{2m-1}\lVert \Theta^m_k(x)\rVert^2_{H}&=\sum^{n}_{k=0}(k+1)^{2m-1}\langle \Theta^m_k(x),\Theta^m_k(x)\rangle\\&=\sum^{n}_{k=0}(k+1)^{2m-1}\langle (\Theta^m_k)^*\Theta^m_k(x),x\rangle.\\&\le\Big\lVert \sum^{n}_{k=0}(k+1)^{2m-1}(\Theta^m_k)^*\Theta^m_k(x)\Big\rVert_{H} \lVert x \rVert_{H}\leq C_{m,r}\lVert x \rVert^2_H.
    \end{align*}
    This completes the proof of the lemma.  
\end{proof}
Let us begin with the following identity which can be found in \cite{I.G.Macdonald}
\begin{align}\label{symmetric-function-equation}
\prod_{j=1}^N (z-\xi_j)=\sum_{k=0}^N (-1)^k\, e_k(\xi_1,\dots,\xi_N)\, z^{N-k}
\end{align}
where $e_k$'s are the elementary symmetric polynomials given by
\begin{align}\label{ccoffecient-of-e}
e_0(\xi_1,\dots,\xi_N)=1,\ \text{and}\quad e_k(\xi_1,\dots,\xi_N) = \sum_{1\le i_1<\dots<i_k\le N} \xi_{i_1}\dots\xi_{i_k},\ 1\leq k\leq N.
\end{align}

For a finite set $E=\{\xi_1,\dots,\xi_N\}\subseteq\mathbb{T}$, let us define 
\begin{align}\label{Q-E}
Q_{E}(T):=\prod\limits^{N}_{j=1}(T-\xi_jI).
\end{align}Let us also define $\hat{E}_i=\{\xi_1,\dots,\xi_{i-1},\xi_{i+1},\dots,\xi_N\}$, for $1\leq i\leq N.$ Note that
\begin{align}\label{Q-E-i}
{Q}_{\hat{E}_i}(T):=\prod\limits^N_{j=1,j\neq i}(T-\xi_jI), \text{where}\ i=1,\dots,N.
\end{align}

\begin{lemma}\label{lemma-for-lagrange-polynomial-type}
Suppose $T\in B(X)$ where $X$ is a Banach space.
Let $N\geq 2$ and $E=\{\xi_1,\dots,\xi_N\}\subseteq\mathbb{T}.$ Then we have
    \begin{align}
        \sum^N_{i=1}\lambda_i{Q}_{\hat{E_i}}(T)=T^{N-1}\qquad
        \end{align}where $\lambda_i=\frac{\xi^{N-1}_i}{\prod\limits^N_{j=1,j\neq i}(\xi_i-\xi_j)}$, $i=1,\dots,N$ and ${Q}_{\hat{E_i}}$ are defined in equation \eqref{Q-E-i}.

\end{lemma}
\begin{proof} Let us define  the vector space $\mathbb{C}_{N-1}[z]:=\{p\in\mathbb{C}[z]:\deg p\le N-1\}.$
Let us consider  the Lagrange polynomials $\ell_j$  defined by
\begin{equation}\label{basis}
\ell_j(z)
=
\prod_{\substack{i=1\\ i\neq j}}^{N}
\frac{z-\xi_i}{\xi_j-\xi_i},
\qquad j=1,\dots,N .
\end{equation}
One can easily check that $\{l_1,\dots, l_N\}$  forms a basis for $\mathbb{C}_{N-1}[z]$. Note that $l_j(\xi_k)=\delta_{jk}$,$1\le j,k\le N$. Hence for any polynomial $p\in\mathbb{C}_{N-1}[z]$, we have
 \begin{align}\label{lagrange-from-of-polynomial}
  p=\sum\limits ^N_{j=1}p(\xi_j)\ell_j.\end{align} Therefore, we get
    $p(T)=\sum^N_{j=1}p(\xi_j)\ell_j(T)$.
This completes the proof of the lemma.
\end{proof} 
\phantomsection
\label{equation-D-n}
For any sequence $a:=(a_n)_n$ let us define the backward difference operator $\bm{\Delta}_n(a):=a_n-a_{n-1}\label{operator-Delta-definition}$ and $\bm{\Delta}^m_n:=\bm{\Delta}^{m-1}_n(\bm{\Delta}_n).$ Also consider $\bm{\Delta}^0_n(a):=a_n$ and $\bm{\Delta}^1_n(a):=\bm{\Delta}_n(a)$. Let us define the differential operator 
\begin{equation}\label{definition-of-D-n-alpha}
\mathscr{D}_n:=\sum\limits^N_{i=0}\alpha^{(N)}_i{\bf{\Delta^i_n}},\text{and}\ \mathscr{D}^m_n:=\mathscr{D}^{m-1}_n(\mathscr{D}_n),\ \text{where}\ \alpha^{(N)}_i\equiv \alpha_i:=(-1)^i\sum\limits^N_{r=i}(-1)^r\binom{r}{i}e_r
\end{equation} and $e_r$s are as in equation \eqref{ccoffecient-of-e}.
 Now we are going to state the following lemma which can be found in \cite[Page 188]{GrahamKnuthPatashnik1989}.
 \begin{lemma}\label{lem:backward-difference}
 Let $\bm{\Delta}_k$ be the backward difference operator defined above. Then, for every integer $m\geq 0$ and $k\geq m$, for {$\mathsf{T}:=(T^k)_{k\geq 0}$} we have 
 \[
 \bm{\Delta}^{m}_{k}(\mathsf{T})
 =
 \sum_{j=0}^{m}
 (-1)^j\binom{m}{j}T^{k-j}.
 \]
 \end{lemma}

\begin{lemma}\label{operator-identity-used-in-l-2-estimates}
Let $n\geq 1$ and $P\in\mathbb{C}[z].$ For $Q(z):=(z-1)P(z),$ we have
for any $T\in B(X)$ 
\begin{equation}\label{eq:operator1}
\sum_{k=1}^{n}kT^{k-1}Q(T)
=
P(T)
\left(
nT^{n}
-
\sum_{k=0}^{n-1}T^{k}
\right).
\end{equation}
Equivalently,
\begin{equation}\label{eq:operator2}
T^{n}P(T)
=
\frac1n
\sum_{k=1}^{n}kT^{k-1}Q(T)
+
\frac1n
\sum_{k=0}^{n-1}T^{k}P(T).
\end{equation}
\end{lemma}
\begin{theorem}\label{strong-maximal-ergodic-theorem-for-l2}
 Let $T:\mathcal M\to\mathcal M$ be such that it satisfies (i)-(iv) in \eqref{thosefourconditions}. Then there exists a positive constant $C_2(T)$ such that 
    \begin{equation}\label{equation-for-strong-maximal-inequality}
       \|\sup_{n\geq 0}\!^{+}T^nx\|_2\leq C_2(T)\|x\|_2
   \end{equation}
   for all $x\in L_2(\mathcal M).$
 \end{theorem}
\begin{proof}
     We first prove \eqref{equation-for-strong-maximal-inequality} for $T_i=\overline{\xi_i}T$.
    Applying Lemma \eqref{operator-identity-used-in-l-2-estimates} for the operator $T_i$, we have \begin{align}\label{the-main-square-function-identity}
        T_i^{n}\prod^N_{j=1,j\neq i}(T_i-\overline{\xi_i}\xi_j)=\frac1n
\sum_{k=1}^{n}kT_i^{k-1}(T_i-I)\prod^N_{j=1,j\neq i}(T_i-\overline{\xi_i}\xi_j)
+
\frac1n
\sum_{k=0}^{n-1}T_i^{k}\prod^N_{j=1,j\neq i}(T_i-\overline{\xi_i}\xi_j).\end{align}
        Applying  Proposition \eqref{sqtomx}, Lemma \eqref{L2sqrfnestimatewithnumercond} and Lemma \eqref{saclarmultijungexu} to the above equation, we have for $\overline\xi_iT$, $i=1,\dots,N$,
 \begin{equation*}
\left\| \sup_{n\geq 1}\!^{+} T^n {Q}_{\hat{E_i}}(T)x\right\|_2\leq C\|x\|_2.
\end{equation*}
where ${Q}_{\hat{E_i}}(T)$ is defined in \eqref{Q-E-i}.
Now we observe
\begin{align*}
\left\|\sup_{n \geq N}\!^{+} T^{n+N-1}x\right\|_2\underset{\eqref{lemma-for-lagrange-polynomial-type}}{\leq} \left\|\sup_{n \geq N}\!^{+} \sum_{i=1}^N T^n \lambda_i {Q}_{\hat{E_i}}(T)x \right\|_2 \leq \sum_{i=1}^N |\lambda_i| \left\|\sup_{ n \geq N}\!^{+} T^n {Q}_{\hat{E_i}}(T)x \right\|_2 \leq C\|x\|_2\end{align*}
Now observe that 
\begin{align*}
\left \lVert \sup\!^{+}_{n\geq 0} T^n x\right\rVert_2\leq\sum^{2N-2}_{i=0}\lVert T^ix\rVert_2+ \left\|\sup_{n \geq N}\!^{+} T^{n+N-1}x\right\|_2\lesssim \lVert x \rVert_2.
 \end{align*}
 This completes the proof of the theorem.
\end{proof}

Let $E=\Delta\cap\mathbb{T}=\{\xi_1,\dots,\xi_N\}$  define the following.  
\begin{align}\label{B-k-m}
   B^m_k\equiv B^m_k(T):=(k+1)^m\Theta^m_k(T)
    \end{align}
where $\Theta^m_k$ is defined as in Lemma \eqref{L2sqrfnestimatewithnumercond}. In what follows we take $\xi_1=1$. Throughout this section  $E$ will remain fixed unless explicitly stated otherwise. Denote $B^m\equiv B^m(T):=(B_k^m(T))_{k\geq 0}.$

We can rewrite equation \eqref{symmetric-function-equation} in the following form
\begin{equation}\label{equation-of-Q-E}
Q_E(z)=\prod^N_{j=1}(z-\xi_j)=\sum^N_{r=0}c_rz^{N-r}, z\in \mathbb{C}
\end{equation}
where $c_r=(-1)^{r}e_{r}(\xi_1,\dots,\xi_N).$
As $\xi_1=1$ then $\sum\limits ^N_{r=0}c_r=0\label{equation-zero-of-c-r}$.
For any sequence $a:=(a_k)_k$ and $k\geq N$, we define the following
$(\delta^a_{Q_{E}})_k:=\sum\limits^N_{r=0}c_ra_{k-r}$.
Also we have $\prod\limits^N_{j=2}(1-\xi_jz)=\sum\limits^{N-1}_{r=0}d_rz^r$ where 
$d_r=(-1)^r e_r(\xi_2,\ldots,\xi_N),
\ r=0,\ldots,N-1.
$
\begin{lemma}\label{lemma-for-s-n} Let $X$ be a Banach space and $T\in B(X)$. For $\{\xi_1,\dots,\xi_N\}$ with $\xi_1=1$, let us define $a:=(a_k)_{k}$,  and \begin{align}\label{equation-delta-Q}
     \mathfrak D_n\equiv\mathfrak D_n(a):=\sum^n_{k=N}k\Big(\delta^{a}_{Q_E}\Big)_k,\qquad \text{for}\hspace{0.2cm}n\geq N.
\end{align}
Then, for every $n\geq N$, we have
\begin{align}
\sum^{N-1}_{r=0}d_ra_{n-r}
=\frac{1}{n}\sum_{k=N-1}^{n-1}
\sum_{r=0}^{N-1}
d_ra_{k-r}
\notag
+\frac{1}{n}\,\mathfrak D_n(a)
+\frac{N-1}{n}
\left(
\sum_{r=0}^{N-1}d_ra_{N-1-r}
\right).
\label{backward-Bkm-identity}
\end{align}

\end{lemma}
    \begin{proof}
For $k\geq N-1$, we set
$C_k
:=
\sum\limits_{r=0}^{N-1}d_ra_{k-r}.$
Since $\xi_1=1$, we have
\[
\prod_{j=1}^{N}(1-\xi_jz)
=
(1-z)\prod_{j=2}^{N}(1-\xi_jz)
=
(1-z)\sum_{r=0}^{N-1}d_rz^r.
\]
Putting $d_{-1}=d_N=0$ and comparing coefficients gives
\[
(-1)^r e_r(\xi_1,\ldots,\xi_N)
=
d_r-d_{r-1},
\qquad 0\leq r\leq N.
\]
Consider the following
\begin{align*}
C_k-C_{k-1}
&=
\sum_{r=0}^{N-1}d_ra_{k-r}
-
\sum_{r=0}^{N-1}d_ra_{k-1-r}\\
&=
\sum_{r=0}^{N}
(d_r-d_{r-1})a_{k-r}\\
&=
\sum_{r=0}^{N}
(-1)^r e_r(\xi_1,\ldots,\xi_N)a_{k-r}\\
&= (\delta^a_{Q_E})_k.
\end{align*}
Hence we have
\begin{align*}
\sum_{k=N}^{n}k(\delta^a_{Q_E})_k
&=
\sum_{k=N}^{n}k(C_k-C_{k-1})\\
&=
\sum_{k=N}^{n}kC_k
-
\sum_{k=N-1}^{n-1}(k+1)C_k\\
&=
nC_n
-
\sum_{k=N-1}^{n-1}C_k
-
(N-1)C_{N-1}.
\end{align*}
Therefore, we obtain
\begin{align}C_n
=
\frac{1}{n}\sum\limits_{k=N-1}^{n-1}C_k
+
\frac{1}{n}
\,\mathfrak D_n(a)
+
\frac{N-1}{n}C_{N-1}.\end{align}
Substituting the values of the coefficients $C_k$, $k=N-1,\dots,n$ yields the desired conclusion.
\end{proof}
\begin{lemma}\label{lemma-for-gamma-r}
   For $m\geq0$  let us define $\gamma_r:=(k-N+r)^{m+1}-k^{m+1}$, $2N\le k\le n$ $r=1,\dots, N+1.$ Then $|{\gamma_r}|\lesssim_m (k+1)^m.$
\end{lemma}\begin{proof} We will prove the lemma for $r=N+1$ for other $r$ the proof follows similarly. Consider the function $f(x)=x^{m+1}$,$m\geq 0$. By mean value theorem for interval $[k,k+1]$, we have 
\begin{align*}
(k+1)^{m+1}-k^{m+1}=h^m(m+1)
\end{align*}
for some $h\in[k,k+1].$
Which implies $\lvert \gamma_{N+1}\rvert\lesssim_m (k+1)^m.$
\end{proof}
\begin{lemma}\label{lemma-for-C-n-m}Let us consider $T\in B(L_2(\mathcal{M}))$  such that the numerical range of $T$ contained in $\overline{E}_r$ for some $0<r<1.$
 For $n\geq N-1$ and $m\geq1$, let us define
$C_n^m\equiv C^m_n(T)
:=
\sum\limits_{r=0}^{N-1}d_rB_{n-r}^m.$ Then we have \begin{align*}
C_n^m 
=
B^m_{n-N+1}Q_{\hat{E_1}}(T)
+ R^m_n(T)
\end{align*}
where\begin{align*}
R^m_n\equiv R^m_n(T):=T^{n-N+1}Q_E(T)^m
\sum_{r=0}^{N-1}
d_r\bigl[(n-r+1)^m-(n-N+2)^m\bigr]
T^{N-1-r}.
\end{align*}
Moreover, $\left\|
\sup_{n\geq N}\!^{+}R_n^m(x)
\right\|_2\lesssim\lVert x\rVert_2$ where $x\in L_2(\mathcal{M}).$
\end{lemma}
\begin{proof}
    For $n\geq N$, we have
\begin{align*}
C_n^m
&=
\sum_{r=0}^{N-1}
d_r(n-r+1)^mT^{n-r}Q_E(T)^m\\
&=
T^{n-N+1}Q_E(T)^m
\sum_{r=0}^{N-1}
d_r(n-r+1)^mT^{N-1-r}.\\
&=
(n-N+2)^mT^{n-N+1}Q_E(T)^m
\sum_{r=0}^{N-1}d_rT^{N-1-r}\\
&+
T^{n-N+1}Q_E(T)^m
\sum_{r=0}^{N-1}
d_r\bigl[(n-r+1)^m-(n-N+2)^m\bigr]
T^{N-1-r}.
\end{align*}
Since
$\sum\limits_{r=0}^{N-1}d_rT^{N-1-r}
=
\prod\limits_{j=2}^{N}(T-\xi_jI)=Q_{\hat{E}_1}(T)$,
the required identity follows.
Let us define $R^{m}_{n,r}:=d_r\bigl[(n-r+1)^m-(n-N+2)^m\bigr]
T^{N-1-r}T^{n-N+1}Q_E(T)^m$,\ $0\leq r\leq N-1.$
For $0\leq r\leq N-1$, from Lemma \eqref{lemma-for-gamma-r}, we have
\[
\left|
(n-r+1)^m-(n-N+2)^m
\right|
\lesssim_{m,N} n^{m-1}.
\]
Hence by Lemma \eqref{L2sqrfnestimatewithnumercond} and Proposition \eqref{sqtomx}, we obtain
\begin{align*}
 \left\|
\sup_{n\geq N}\!^{+}R_{n,r}^m(x)
\right\|_2\lesssim\left(
\sum_{n\geq N}\lVert R_{n,r}^m(x)\rVert^2_2
\right)^\frac{1}{2}&\lesssim_{m,N}
\left(
\sum_{k\geq1}
(k+1)^{2m-1}
\lVert T^kQ_E(T)^m(x)\rVert_2
^2
\right)^{1/2}\\
&\lesssim_{m,N} \|x\|_2,
\end{align*}
Therefore, by triangle inequality we obtain

\[
\left\|
\sup_{n\geq N}\!^{+}R_n^m(x)
\right\|_2
\leq
\sum_{r=0}^{N-1}\left\|
\sup_{n\geq N}\!^{+}R_{n,r}^m(x)
\right\|_2
\lesssim_{m,N}
\|x\|_2.
\]
This completes the prove of the theorem.
\end{proof}

\begin{lemma}\label{lemma-for-B-m-plus-1)}Let $T\in B(X)$ where $X$ is a Banach space.
     Consider $B^m_k$ and $\Theta^m_k$ which are defined in equation \eqref{B-k-m} and Lemma \eqref{L2sqrfnestimatewithnumercond} respectively with $E=\{1,\xi_2,\ldots,\xi_N\}.$ Then for any $n\geq2N$ and $m\geq0$, we have 
     \begin{equation*}
         \sum^n_{k=0}B^{m+1}_k=\sum^{2N-1}_{k=0}\alpha_k\Theta^m_k +\sum^n_{k=2N}\beta_k\Theta ^m_k+\sum^{n+N}_{k=n+1}\mathcal{P}_k(n)\Theta ^m_k
     \end{equation*}
    where \[
\alpha_k:=
\begin{cases}
\displaystyle
\sum_{r=0}^{k} c_{N-r} (k-r+1)^{m+1},
& 0\leq k\leq N-1, \\[3mm]
\displaystyle
\sum_{r=0}^{N} c_{N-r} (k-r+1)^{m+1},
& N\leq k\leq 2N-1,
\end{cases}
\] for each $k=2N,\dots, n$, $\beta_k$ is defined by
    $\beta_k:=c_N\{(k+1)^{m+1}-k^{m+1}\}+\dots+c_0\{(k-N+1)^{m+1}-k^{m+1}\}$ and for each $k=n+1,\dots,n+N$, $\mathcal{P}_k(n):=\sum\limits^N_{r=k-n}c_{N-r}(k-r+1)^{m+1}.$
     \end{lemma}
     \begin{proof}
     From the definition of $B^m_k$, we have 
     \begin{align}\label{equation-of-sum-of-B-m-plus-1}\sum^n_{k=0}B^{m+1}_k&=\sum^n_{k=0}(k+1)^{m+1}T^kQ_{E}(T)^{m+1}\notag\\&=Q_{E}(T)^m\sum^n_{k=0}\sum^N_{r=0}(k+1)^{m+1}c_rT^{k+N-r}\notag\\&=Q_{E}(T)^m\sum^{N-1}_{k=0}\sum^N_{r=0}(k+1)^{m+1}c_{N-r}T^{k+r}+Q_{E}(T)^m\sum^n _{k=N}\sum^N_{r=0}(k+1)^{m+1}c_{N-r}T^{k+r}\notag\\&=Q_{E}(T)^m\sum^{N-1}_{k=0}\sum^N_{r=0}(k+1)^{m+1}c_{N-r}T^{k+r}+Q_{E}(T)^m\sum^n _{k=N}(k+1)^{m+1}c_NT^{k}\notag\\&\qquad+Q_{E}(T)^m\sum^n_{k=N}(k+1)^{m+1}c_{N-1}T^{k+1}+\dots+Q_{E}(T)^m\sum^n_{k=N}(k+1)^{m+1}c_0T^{k+N}\notag\\&=Q_{E}(T)^m\sum^{N-1}_{k=0}\sum^N_{r=0}(k+1)^{m+1}c_{N-r}T^{k+r}+Q_{E}(T)^m\sum^{n}_{k=N}c_N(k+1)^{m+1}T^k\notag \\&\qquad+Q_{E}(T)^m\sum^{n+1} _{k=N+1}c_{N-1}k^{m+1}T^{k}+\dots+Q_{E}(T)^m\sum^{n+N}_{k=2N}c_0(k-N+1)^{m+1}T^k.
     \end{align}
    Finally, we have
     \begin{align}\label{equation-P}
         \sum^n_{k=0}B^{m+1}_k=\sum^{N-1}_{k=0}\sum^N_{r=0}(k+1)^{m+1}c_{N-r}\Theta^m_{k+r}+\sum^{n}_{k=N}c_N(k+1)^{m+1}\Theta^m_k\notag+\sum^{n+1} _{k=N+1}c_{N-1}k^{m+1}\Theta^m_k+\dots\\+\sum^{n+N}_{k=2N}c_0(k-N+1)^{m+1}\Theta^m_k.
     \end{align}
     Now we focus on calculating the coefficients of $\Theta^m_k$ in equation \eqref{equation-P}.
When $2N\leq k\leq n$, the coefficient of $\Theta^m_k$ in equation \eqref{equation-P} becomes
\begin{align*}
    &c_N(k+1)^{m+1}+\dots+c_0(k-N+1)^{m+1}\\&\underset{\eqref{equation-zero-of-c-r}}{=}c_N\{(k+1)^{m+1}-k^{m+1}\}+\dots+c_0\{(k-N+1)^{m+1}-k^{m+1}\}.
\end{align*}
Also, from the last part of equation \eqref{equation-P}, one can observe that the coefficient of $\Theta^m_k$, $k=n+1,\dots,n+N$ becomes $\sum\limits^N_{r=k-n}c_{N-r}(k-r+1)^{m+1}.$ Let us denote  $\mathcal{P}_k(n):=\sum\limits^N_{r=k-n}c_{N-r}(k-r+1)^{m+1}$, $k=n+1,\dots,n+N$. Note that $\lvert\mathcal{P}_{n+1}(n)\rvert\leq\sum\limits^N_{r=1}\abs{c_{N-r}(n-r+2)^{m+1}}\lesssim_m n^{m+1}$. Similarly we can prove that $\lvert \mathcal{P}_k(n)\rvert\lesssim_m n^{m+1}$, $n+2\leq k\leq n+N.$
Finally, from equation \eqref{equation-P} one can notice that the coefficient of $\Theta^m_k$, $k=0,\dots,2N-1$ are independent of $n.$  An elementary computation shows that
\[\alpha_k=
\begin{cases}
\displaystyle
\sum_{r=0}^{k} c_{N-r} (k-r+1)^{m+1},
& 0\leq k\leq N-1, \\[3mm]
\displaystyle
\sum_{r=0}^{N} c_{N-r} (k-r+1)^{m+1},
& N\leq k\leq 2N-1.
\end{cases}
\] The proof is completed.
\end{proof}

\begin{lemma}\label{lemma-for-beta-j-delta-j}Let us consider $T\in B(L_2(\mathcal{M}))$  such that the numerical range of $T$ contained in $\overline{E}_r$ for some $0<r<1.$ Also assume that $E=\Delta\cap\mathbb{T}=\{1,\dots,\xi_N\}.$ For $n\geq 2N$ and $m\geq1$, let us consider $\beta_k$, $k=2N,\dots,n$ which is defined in Lemma \eqref{lemma-for-B-m-plus-1)}. Then
\begin{align}
\Big\lVert \sup_{n\geq 2N}\!^{+}\frac{1}{n+1}\sum^n_{k=2N}\beta_k\Theta^m _k(x)\Big\rVert_2\lesssim_m\lVert x\rVert_2.   
\end{align}\begin{proof}
Note that by Proposition \eqref{sqtomx} and Lemma \eqref{lemma-for-gamma-r} we obtain that
\begin{align*}
\Big\lVert \sup_{n\geq2N}\!^{+}\frac{1}{n+1}\sum^n_{k=2N}\beta_k\Theta^m _k(x)\Big\rVert_2& \lesssim  \sup_{n\geq2N}\!^{+} \frac{1}{n+1}\Big(\sum^n_{k=0}(k+1)\Big)^\frac{1}{2}\Big(\sum^\infty_{k=0}(k+1)^{2m-1}\|\Theta^m_k(x)\|^2\Big)^{\frac{1}{2}}\\
&\underset{\eqref{L2sqrfnestimatewithnumercond}}{\lesssim_m} \|x\|_2.
\end{align*}
This completes the proof of the lemma. \end{proof}
\end{lemma}
\begin{lemma}\label{lemma-for-the-estimates-of-forward-difference} Let us consider $T\in B(L_2(\mathcal{M}))$  such that the numerical range of $T$ contained in $\overline{E}_r$ for some $0<r<1$ with $E=\Delta\cap\mathbb{T}=\{\xi_1,\dots,\xi_N\}$.
Consider $\mathfrak D_n$ which is defined in \eqref{equation-delta-Q}. Then for $m\geq 1$, we have   \begin{align}\Big\lVert\sup\limits_{n\geq N}\!^{+}\frac {1}{n}\mathfrak D_n(B^m)(x)\Big\rVert_2\lesssim\lVert x\rVert_2,\ \text{for all}\ x\in L_2(\calM).\end{align}
\end{lemma}
\begin{proof}
For $n\geq N$, consider the following computation
\begin{align}\label{equation-of-S-n}    \mathfrak D_n(B^m)=\sum^n_{k=N}k\Big(\sum^N_{r=0}c_rB^m_{k-r}\Big)=\sum^n_{k=N}k\Big(\sum^N_{r=0}c_r(k-r+1)^m\Theta^m_{k-r}\Big).
 \end{align}
Note that \begin{align}\label{equation-for-delta-k-m-plus-1}&\sum^n_{k=N}k^{m+1}\Theta^{m+1}_{k-N}\\&=\sum^n_{k=N}k^{m+1}\Theta^m_{k-N}\Big(\sum^N_{r=0}c_rT^{N-r}\Big)\notag\\&=\sum^n_{k=N}\sum^N_{r=0}k^{m+1}c_r\Theta^m_{k-r}\notag.
 \end{align}
Therefore, adding and subtracting $\sum\limits^n_{k=N}k^{m+1}\Theta^{m+1}_{k-N}$ to the equation \eqref{equation-of-S-n} and using equation \eqref{equation-for-delta-k-m-plus-1} we have
\begin{align}\label{D-n}
    \mathfrak D_n(B^m)=\sum^n_{k=N}k^{m+1}\Theta^{m+1}_{k-N}+\sum^n_{k=N}k\Big(\sum^N_{r=0}c_r\Big((k-r+1)^m-k^m\Big)\Theta^m_{k-r}\Big).
 \end{align}
 Using Proposition \eqref{sqtomx} we have for all $0\leq r\leq N$
 \begin{align}\label{square-function-k-plus-r}
 &\Big\|\sup\limits_{n\geq N}\!^{+}\frac {1}{n}\sum^n_{k=N}k\Big((k-r+1)^m-k^m\Big) \Theta^m_{k-r}(x)\Big\|_2\notag \\&\lesssim\sup_{n\geq N}\frac{1}{n}\Big(\sum^n_{k=N}\frac{[k\Big((k-r+1)^m-k^m\Big)]^2}{(k+1)^{2m-1}}\Big)^{\frac{1}{2}}\Big(\sum^{\infty}_{k=N}(k+1)^{2m-1}\|\Theta^m_{k-r}(x)\|_2^2\Big)^\frac{1}{2}\notag\\&\underset{\eqref{lemma-for-gamma-r}\eqref{L2sqrfnestimatewithnumercond}}{\lesssim_m} \|x\|_2.\end{align}
Again using Proposition \eqref{sqtomx} we have
\begin{align}\label{Square-function-for-k-to-power-m-plus-1} & \Big\|\sup\limits_{n\geq N}\!^{+}\frac{1}{n}\sum^{n}_{k=N}k^{m+1}\Theta^{m+1}_{k-N}(x)\Big\|_2\\&=\Big\|\sup\limits_{n\geq N}\!^{+}\frac{1}{n}\sum^{n-N}_{k=0}(k+N)^{m+1}\Theta^{m+1}_{k}(x)\Big\|_2\notag\\ &\lesssim\sup\limits_{n\geq N}\frac{1}{n}\Big(\sum^{n-N}_{k=0}\frac{(k+N)^{2(m+1)}}{(k+1)^{2m+1}}\Big)^{\frac{1}{2}}\Big(\sum^\infty_{k=0}\Big(k+1\Big)^{2m+1}\|\Theta^{m+1}_k(x)\|^2_2\Big)^\frac{1}{2}\notag\\&\lesssim\sum^\infty_{k=0}\Big(k+1\Big)^{2m+1}\|\Theta^{m+1}_k(x)\|^2_2\Big)^\frac{1}{2}\underset{\eqref{L2sqrfnestimatewithnumercond}}{\lesssim_m} \|x\|_2.\notag
\end{align}
 Hence putting inequalities in \eqref{square-function-k-plus-r} and \eqref{Square-function-for-k-to-power-m-plus-1} in equation \eqref{D-n}, we have the desired result by using triangle inequality. This completes the proof of the lemma.
 \end{proof}
\begin{theorem}\label{strong-maximal-theorem-of-B-k-m}
 Let $T:\mathcal M\to\mathcal M$ be an operator such that it satisfies (i)-(iv) in \eqref{thosefourconditions} such that $E=\Delta\cap\mathbb{T}=\{\xi_1,\dots,\xi_N\}$.  Then for all $m\geq 0$ and $x\in L_2(\mathcal M)$, we have 
   \begin{equation}\label{estimates-B-k-m}
     \|\sup_{n\geq 0}\!^{+}B^m_n(x)\|_2\lesssim_m\|x\|_2.
   \end{equation}
\end{theorem}\begin{proof} Let $T_i=\overline{\xi_i}T$, $i=1,\dots,N.$ Also assume that $E_{T_i}=\{\overline\xi_i\xi_1,\dots,\overline\xi_i\xi_{i-1},1,\overline\xi_i\xi_{i+1},\dots,\overline\xi_i\xi_N\}.$ 
   We proceed by induction on $m\geq0$ to prove 
   \begin{align*}
       \Big\lVert \sup\limits_{n\geq 0}\!^{+}B^m_n(x)\Big\lVert_2\lesssim \lVert x\rVert_2.
   \end{align*}
   For any $m\geq1$, let us consider the estimate which again we will prove via induction for $m\geq 1$
   \begin{align}\label{estimates-sum-B-k-m}
       \Big\lVert \sup\limits_{n\geq 0}\!^{+}\frac{1}{n+1}\sum^n_{k=0}B^m_k(x)\Big\rVert_2\lesssim_m \lVert x\rVert_2.
   \end{align}
  Note that estimate in \eqref{estimates-B-k-m} implies \eqref{estimates-sum-B-k-m}. For clarity we will write $(\eqref{estimates-B-k-m})_m$ and $(\eqref{estimates-sum-B-k-m})_m$ instead of $\eqref{estimates-B-k-m}$ and $\eqref{estimates-sum-B-k-m}.$
   For $m=0$, using Theorem \eqref{strong-maximal-ergodic-theorem-for-l2} we have
   \begin{align*}
       \Big\lVert \sup\limits_{n\geq 0}\!^{+}B^0_n(x)\Big\rVert_2=\Big\lVert \sup\limits_{n\geq 0}\!^{+}T^nx\Big\rVert_2 \lesssim \lVert x\rVert_2.
   \end{align*}
   Applying Theorem \eqref{strong-maximal-ergodic-theorem-for-l2} to equation \eqref{the-main-square-function-identity} for the operator $T_i$, we obtain $\eqref{estimates-sum-B-k-m}_1.$
For all $n\geq 2N$ from Lemma \eqref{lemma-for-B-m-plus-1)}, we have \begin{align*}\frac{1}{n+1} \sum^n_{k=0}B^{m+1}_k(T_i)&=\sum^{2N-1}_{k=0}\frac{\alpha_k}{n+1}\Theta_k ^m(T_i)+\frac{1}{n+1}\sum^n_{k=2N}\beta_k\Theta ^m_k(T_i)+\frac{1}{n+1}\sum^{n+N}_{k=n+1}\mathcal{P}_k(n)\Theta^m_k(T_i)\\&=\sum^{2N-1}_{k=0}\frac{\alpha_k}{n+1}\Theta_k ^m(T_i)+\frac{1}{n+1}\sum^n_{k=2N}\beta_k\Theta ^m_k(T_i)+\frac{1}{n+1}\sum^{n+N}_{k=n+1}\frac{\mathcal{P}_k(n)}{(k+1)^{m}} B^m_k(T_i).
     \end{align*}
     Let $(\eqref{estimates-B-k-m})_m$ hold. Then this implies that $(\eqref{estimates-sum-B-k-m})_m$ holds. Now let us focus to estimates the boundary term.
     Putting $r=k-n$, we have 
     \begin{align*}
         \frac{1}{n+1}\sum^{n+N}_{k=n+1}\frac{\mathcal{P}_k(n)}{(k+1)^{m}} B^m_k=\frac{1}{n+1}\sum^{N}_{r=1}\frac{\mathcal{P}_{n+r}(n)}{(n+r+1)^{m}} B^m_{n+r}.
     \end{align*}
    Using the estimates of $\mathcal{P}_{n+r}$ from Lemma \eqref{lemma-for-B-m-plus-1)}, $(\eqref{estimates-B-k-m})_m$ and Corollary \eqref{mod1andmaximum}, for $r=1,\dots,N$ and $x\in L_2({\mathcal{M}})$ we obtain
     \begin{align}\label{boundary-term-of-B-k-m-plus-1}
         \left\lVert\sup_{n\geq1}\!^{+}\frac{\mathcal{P}_{n+r}(n)}{(n+1)(n+r+1)^{m}} B^m_{n+r}(T_i)(x)\right\rVert_2&\lesssim \sup_{n\geq1}\frac{\lvert\mathcal{P}_{n+r}(n)\rvert}{(n+1)(n+r+1)^{m}}\left\lVert\sup_{n\geq1}\!^{+}B^m_{n+r}(T_i)(x)\right\rVert_2\\&\lesssim \lVert x\rVert_2\notag.
     \end{align}
     Using Lemma \eqref{lemma-for-beta-j-delta-j} and inequality in \eqref{boundary-term-of-B-k-m-plus-1}, we have $(\eqref{estimates-sum-B-k-m})_{m+1}$ for all $T_i,$ $1\leq i\leq N.$
     From Lemma \eqref{lemma-for-s-n} and Lemma \eqref{lemma-for-C-n-m}, we have the following
     \begin{align*}
   &B^{m+1}_{n-N+1}Q_{\hat{E}_{T_i}}(T_i)\\&
=\frac{1}{n}\sum_{k=N-1}^{n-1}
\sum_{r=0}^{N-1}
d_rB^{m+1}_{k-r}(T_i)
\notag
+\frac{1}{n}\,\mathfrak D_n(B^{m+1})
+\frac{N-1}{n}
\left(
\sum_{r=0}^{N-1}d_rB^{m+1}_{N-1-r}(T_i)
\right)-R^{m+1}_n(T_i).  
     \end{align*}
     where $n\geq 2N.$
    By using Lemma \eqref{lemma-for-C-n-m}, $\eqref{estimates-sum-B-k-m}_{m+1}$ for $T_i$ and  Lemma \eqref{lemma-for-the-estimates-of-forward-difference}  one can obtain that $$\Big\lVert\sup\limits_{n\geq 2N}\!^{+}B^{m+1}_{n-N+1}Q_{\hat{E}_{T_i}}(T_i)(x)\Big\rVert_2\lesssim_m \lVert x\rVert_2.$$ Also from Corollary \eqref{mod1andmaximum} for all $i=1,\dots,N$, observe that \begin{align*}
        \Big\lVert\sup\limits_{n\geq 2N}\!^{+}B^{m+1}_{n-N+1}Q_{\hat{E}_{i}}(T)(x)\Big\rVert_2=\Big\lVert\sup\limits_{n\geq 2N}\!^{+}B^{m+1}_{n-N+1}Q_{\hat{E}_{T_i}}(T_i)(x)\Big\rVert_2\lesssim_{m}\lVert x\rVert_2.
    \end{align*}
    Now using Lemma \eqref{lemma-for-lagrange-polynomial-type}, we have
    \begin{align*}
        \Big\lVert\sup\limits_{n\geq 2N}\!^{+}B^{m+1}_{n}(T)(x)\Big\rVert_2\leq\sup_{n\geq2N}\left(\frac{n+1}{n-N+2}\right)^{m+1}\sum^N_{i=1}\lvert\lambda_i\rvert\Big\lVert\sup\limits_{n\geq 2N}\!^{+}B^{m+1}_{n-N+1}Q_{\hat{E}_{i}}(T)(x)\Big\rVert_2\lesssim_{m}\lVert x\rVert_2.
    \end{align*}
    Note that \begin{align*}    
    \Big\lVert \sup\limits_{n\geq 0}\!^{+}B^{m+1}_n(T)(x)\Big\rVert_2\lesssim_m \Big\lVert \sup\limits_{n\geq 2N}\!^{+}B^{m+1}_{n}(T)(x)\Big\rVert_2+\sum^{2N-1}_{r=0}\Big\lVert B^{m+1}_r(T)(x)\Big\rVert_2\lesssim_{m}\lVert x\rVert_2.
    \end{align*}
    Therefore, the result follows from induction.
\end{proof}
\subsection{Embedding the maximal operator in analytic family}\label{section3.2}
Let $T\in B(X)$ satisfy \eqref{thosefourconditions} with $E=\Delta\cap\mathbb{T}=\{\xi_1,\dots,\xi_N\}.$ For a given complex number $\lambda$ and  nonnegative integers $n,k$, let us define the following
\begin{align}\label{all-required-definition-Here}
&a^\lambda_k:=\frac{(\lambda+1)(\lambda+2)\dots(\lambda+k)}{k!}\hspace{0.2cm} \textit{when}\hspace{0.2cm} k\geq0,\ \text{and}\ a^{\lambda}_0:=1.\\
&A^{\lambda}_k:=\sum_{m_1+\dots+m_N=k}\prod^N_{j=1} a^\lambda_{m_j}\xi^{m_j}_j\qquad \textit{where}\hspace{0.2 cm} m_i\geq0, i=1,\dots,N \notag\\& A^{\lambda}_k:=0, k<0 \notag\\&
    S^{\lambda}_n:=\sum^n_{k=0}A^{\lambda-1}_{n-k}T^k\notag
    \\&S^{\lambda}_n:=0,\ n<0\notag\\&
    M^{\lambda}_{n}:=(n+1)^{-\lambda} S^{\lambda}_n\notag.
\end{align}
\begin{lemma}\label{formn1}
Let $N\geq2$ and $\{\xi_1, \dots,\xi_N\}\subseteq \mathbb{T}.$ Then for any $k\geq 0$ we have
 \begin{equation}\label{morecompexity}
\sum_{\substack{m_1+\cdots+m_N=k\\ m_i\ge 0}}\prod^N_{j=1}\xi^{m_j}_j=\sum^N_{j=1}\frac{\xi^{k+N-1}_j}{\prod\limits^N_{l=1,l\neq j}(\xi_j-\xi_l)}.\end{equation}
     \end{lemma}
\begin{proof}
    To prove this lemma we will proceed by induction on N. For $N=2$, we observe that the LHS of \eqref{morecompexity} is equal to\begin{align*}\xi^k_2\sum^k_{r=0}\Big(\frac{\xi_1}{\xi_2}\Big)^r=\xi^k_2\Big(\frac{1-\Big(\frac{\xi_1}{\xi_2}\Big)^{k+1}}{1-\frac{\xi_1}{\xi_2}}\Big)=\frac{\xi^{k+1}_1}{\xi_1-\xi_2}+\frac{\xi^{k+1}_2}{\xi_2-\xi_1}.
    \end{align*}
    Now, let \eqref{morecompexity} be true for some $N\geq 2.$
Then using the induction hypothesis we observe the following.    \begin{align}\label{equation-h-k}&\sum_{m_1+\dots+m_{N+1}=k\ m_i\ge 0}\prod^{N+1}_{j=1}\xi^{m_j}_j\notag\\&=\sum^k_{r=0}\xi^r_{N+1}\sum_{m_1+\dots+m_N=k-r}\prod^N_{j=1}\xi^{m_j}_j\notag\\&=\sum^k_{r=0}\xi^r_{N+1}\sum^N_{j=1}\frac{\xi^{N-1+k-r}_j}{\prod\limits^{N}_{l=1, l\neq j}(\xi_j-\xi_l)} \notag\\&=\sum^N_{j=1}\frac{\xi^{N-1}_j}{\prod\limits^{N}_{l=1, l\neq j}(\xi_j-\xi_l)} \xi^k_j\sum^k_{r=0}\Big(\frac{\xi_{N+1}}{\xi_j}\Big)^r\notag\\&=\sum^N_{j=1}\frac{\xi^{N-1}_j}{\prod\limits^{N}_{l=1, l\neq j}(\xi_j-\xi_l)} \frac{\xi^{k+1}_j-\xi^{k+1}_{N+1}}{\xi_j-\xi_{N+1}}\end{align}
   Let us set \begin{align*}
       A:=\sum^N_{j=1}\frac{\xi^{N-1}_j}{\prod\limits^{N+1}_{l=1,l\neq j}(\xi_j-\xi_l)}.
\end{align*}
Consider the identity from \eqref{lagrange-from-of-polynomial} $z^{N-1}=\sum^{N+1}_{j=1}\xi^{N-1}_jl_j(z)$ where $l_j(z)=\prod\limits^{N+1}_{j=1,l\neq j}\frac{z-\xi_l}{\xi_j-\xi_l}.$ Comparing the coefficient of $z^N$ in the both side of above equation, we have
$A=-\frac{\xi^{N-1}_{N+1}}{\prod^{N}_{l=1}(\xi_{N+1}-\xi_l)}.$ 
Substituting $A$ to the equation \eqref{equation-h-k}, we obtain \begin{align*}
    \sum^{N+1}_{j=1}\frac{\xi^{N+k}_j}{\prod\limits ^{N+1}_{l=1,l\neq j}(\xi_j-\xi_l)}.
\end{align*}
This completes the proof of the lemma. 
\end{proof}
Let us observe the following.
\begin{equation}\label{putting0and1}
M^0_n=T^n,\ \text{and}\ M^1_n\underset{\eqref{formn1}}{=}\frac{1}{n+1}\Big(\sum\limits^n_{k=0}\sum\limits^N_{j=1}\frac{\xi^{n-k+N-1}_j}{\prod\limits^N_{l=1,l\neq j}(\xi_j-\xi_l)}T^k\Big).
\end{equation}
\begin{lemma}\label{lemma-for-alpha-beta}
Let $S^{\lambda}_n$ be defined in equation \eqref{all-required-definition-Here}. Then we have that $S^{\alpha+\beta}_{n}=\sum\limits^n_{i=0}A^{\alpha-1}_{n-i}S^{\beta}_{i}$ for all $\alpha,\beta\in\mathbb C.$
\end{lemma}
\begin{proof} 
   For any complex no $\alpha$, we denote 
   \begin{align*}
     P(z):=(1-\xi_1z)\dots(1-\xi_Nz).  
   \end{align*}
   Let us define $P(z)^{\alpha}:=(1-\xi_1z)^{\alpha}\dots(1-\xi_Nz)^{\alpha}.$
   Then one can observe that $P(z)^{-\alpha}=\sum^{\infty}_{n=0}A^{\alpha-1}_nz^n$ where $|{z}|<1$.
   From definition of $P(z)^{-\alpha}$, we have 
   \begin{align*}
       {P(z)}^{-\alpha-\beta}&=P(z)^{-\alpha}P(z)^{-\beta}\\
&=\Big(\sum^{\infty}_{n=0}A^{\alpha-1}_nz^n\Big)\Big(\sum^{\infty}_{n=0}A^{\beta-1}_nz^n\Big)=\sum^{\infty}_{n=0}\Big(\sum^n_{k=0}A^{\alpha-1}_kA^{\beta-1}_{n-k}\Big)z^n.
   \end{align*}
   Therefore, we obtain that      $A^{\alpha+\beta-1}_n=\sum^n_{k=0}A^{\alpha-1}_kA^{\beta-1}_{n-k}.$
   So, we have 
   \begin{align*}
       S^{\alpha+\beta}_n&=\sum^{n}_{k=0}A^{ \alpha+\beta-1}_{n-k}T^{k}\\&=\sum^n_{k=0}\sum^{n-k}_{i=0}A^{\alpha-1}_iA^{\beta-1}_{n-k-i}T^{k}\\&=\sum^n_{i=0}\sum^{n-i}_{k=0}A^{\alpha-1}_iA^{\beta-1}_{n-k-i}T^k=\sum^n_{i=0}A^{\alpha-1}_iS^{\beta}_{n-i}=\sum^n_{i=0}A^{\alpha-1}_{n-i}S^{\beta}_{i}.
   \end{align*}
    This completes the proof of the lemma.
\end{proof}
\begin{lemma}\label{lem:negative-integer-means}
Let $m\geq 1$ be an integer and consider the polynomial $Q_E$ is defined as in equation \eqref{equation-of-Q-E}. Then, for every $n\geq Nm$,
\[
S_n^{-m}
=
T^{\,n-Nm}Q_E(T)^m.
\]
Consequently, we have
$
M_n^{-m}
=
(n+1)^mT^{\,n-Nm}Q_E(T)^m.$
\end{lemma}
\begin{proof}
Applying Lemma~\ref{lemma-for-alpha-beta} with
$\alpha=-m
\ \text{and}\
\beta=0.$
We obtain
\[
S_n^{-m}
=
\sum_{i=0}^{n}
A_{n-i}^{-m-1}S_i^0.
\]
Since $S_i^0=T^i$, it follows that $
S_n^{-m}=\sum\limits_{i=0}^{n}
A_{n-i}^{-m-1}T^i.
$
Putting $r=n-i$, we get
$$S_n^{-m}
=\sum\limits_{r=0}^{n}
A_r^{-m-1}T^{n-r}.$$
For $0\leq k\leq m$, we have
$a_k^{-m-1}
=(-1)^k\binom{m}{k}$,
whereas
$a_k^{-m-1}=0,
\ k>m.$
Hence, by the definition of $A_r^\lambda$ in
\eqref{all-required-definition-Here}, we have
$A_r^{-m-1}=0,
\ r>Nm.$
Therefore, for $n\geq Nm$, we have
\[
S_n^{-m}
=
\sum_{r=0}^{Nm}
A_r^{-m-1}T^{n-r}.
\]

On the other hand, from the definition of $A_r^{-m-1}$, we have
\[
\prod_{j=1}^{N}(T-\xi_j)^m
=
\sum_{r=0}^{Nm}
A_r^{-m-1}T^{Nm-r}.
\]
Multiplying both sides by $T^{n-Nm}$ gives
$
T^{n-Nm}Q_E(T)^m
=
\sum\limits_{r=0}^{Nm}
A_r^{-m-1}T^{n-r}$.

Consequently,
\[
S_n^{-m}
=
T^{n-Nm}Q_E(T)^m.
\]
This completes the proof.
\end{proof}
\begin{lemma}\label{backward-difference-of-A-k-lambda}
Let $A_k^\lambda$ be defined as in \eqref{all-required-definition-Here}. 
Then, for every $k\geq N$,
\[
A_k^\lambda
=
\sum_{r=0}^{N}(-1)^r
e_r(\xi_1,\dots,\xi_N)
A_{k-r}^{\lambda+1}.
\]
Moreover, we have 
\begin{align*}
        S^{\lambda}_n=\sum^N_{r=0}(-1)^re_rS^{\lambda+1}_{n-r}.
    \end{align*}
\end{lemma}

\begin{proof}
We first observe that
\[
a_0^{-2}=1,\qquad a_1^{-2}=-1,
\qquad\text{and}\qquad
a_k^{-2}=0,\quad k\geq2.
\]
Hence, by the definition of $A_r^\lambda$ in
\eqref{all-required-definition-Here},
\begin{align*}
A_r^{-2}
&=
\sum_{\substack{m_1+\cdots+m_N=r\\ m_j\geq0}}
\prod_{j=1}^{N}a_{m_j}^{-2}\xi_j^{m_j}  \\
&=
(-1)^r
\sum_{1\leq j_1<\cdots<j_r\leq N}
\xi_{j_1}\cdots\xi_{j_r} \\
&=
(-1)^r e_r(\xi_1,\dots,\xi_N),
\qquad 0\leq r\leq N.
\end{align*}
Moreover,
$A_r^{-2}=0,\ r>N$.
In the proof of Lemma~\ref{lemma-for-alpha-beta} with
$\alpha=-1\ \text{and}\
\beta=\lambda+2.$
We observe that
\begin{align}
A_k^{\lambda}
&=
\sum_{i=0}^{k}
A_{i}^{-2}A_{k-i}^{\lambda+1}\notag\\&=\sum^N_{i=0}(-1)^ie_i(\xi_1,\xi_2,\dots,\xi_N)A^{\lambda+1}_{k-i}\label{equation-A-k-lambda}.
\end{align}
By the definition of $S^{\lambda}_n$ and equation \eqref{equation-A-k-lambda}, we have
    \begin{align*}
        S^{\lambda}_n&=\sum^n_{k=0}\sum^N_{r=0}(-1)^re_rA^{\lambda}_{n-k-r}T^k\\&=\sum^N_{r=0}(-1)^re_r\sum^{n-r}_{k=0}A^{\lambda}_{n-r-k}T^k+\sum^N_{r=0}(-1)^re_r\sum^n_{k=n-r+1}A^{\lambda}_{n-r-k}T^k\\&=\sum^N_{r=0}(-1)^re_r\left(\sum^{n-r}_{k=0}A^{\lambda}_{n-r-k}T^k\right)=\sum^N_{r=0}(-1)^re_rS^{\lambda+1}_{n-r}.
    \end{align*}
This completes the proof of the lemma.
\end{proof}

We now turn our attention to estimates for \(A_k^\lambda\). In \cite{Flajolet1990SingularityAO}, the authors established the corresponding result for functions with a single dominant singularity. They also mention in their paper that the result can be extended to functions with finitely many dominant singularities.
\begin{lemma}[\cite{Flajolet1990SingularityAO}, Theorem 1]\label{lemma-Flajolet-paper}
Let $\beta\in\mathbb{R}$. Fix $\eta>0$ and $0<\phi<\frac{\pi}{2}.$ Then 
\begin{align*}
\int_{\gamma_1+\gamma_2+\gamma_3+\gamma_4}|{z-1}|^{\beta}\frac{d|{z}|}{|z|^{k+1}}\lesssim_{\beta} (k+1)^{-\beta-1}
\end{align*}
where
\begin{align*}
\gamma_1
&=
\left\{
z \;\middle|\;
|z-1|=\frac1k,\;
|\Arg(z-1)|\ge \phi
\right\},\\[.3ex]
\gamma_2
&=
\left\{
z \;\middle|\;
\frac1k \le |z-1|,\;
|z|\le 1+\eta,\;
\Arg(z-1)=\phi
\right\},\\[.3ex]
\gamma_3
&=
\left\{
z \;\middle|\;
|z|=1+\eta,\;
|\Arg(z-1)|\ge \phi
\right\},\\[.3ex]
\gamma_4
&=
\left\{
z \;\middle|\;
\frac1k \le |z-1|,\;
|z|\le 1+\eta,\;
\Arg(z-1)=-\phi
\right\} 
\end{align*}
and $k$ is sufficiently large.
\end{lemma}
\begin{lemma}\label{final-estimates-of-A-k-lambda}
     Let $A^{\lambda}_k$, $\lambda=\beta+i\gamma $ be defined in \eqref{all-required-definition-Here} and $E=\{\xi_1,\dots,\xi_N\}\subseteq \mathbb{T}$. Then \[|A^{\lambda}_k|\lesssim_{\beta,E}Ne^{2\pi N|\gamma|}(k+1)^\beta.\]
\end{lemma}
\begin{proof}
As in \cite{Flajolet1990SingularityAO} we concentrate only when $k$ is  sufficiently large. Choose some angles $0<\delta_i,\phi<\frac{\pi}{2},i=1,\dots,N,N+1$ and for a $\eta\geq0$ define
\begin{align*}
\gamma^1_{\delta_i}
&=
\left\{
z \;\middle|\;
|z-\overline{\xi_i}|=\frac1k,\;
|\Arg(\xi_iz-1)|\ge \phi
\right\},\\[.3ex]
\gamma^2_{\delta_i}
&=
\left\{
z \;\middle|\;
\frac1k \le |z-\overline{\xi_i}|,\;
|z|\le 1+\eta,\;
\Arg(\xi_iz-1)=\phi
\right\},\\[.3ex]
\gamma^3_{\delta_i^{+}}
&=
\left\{
z \;\middle|\;
|z|=1+\eta,\;
\phi\le \Arg(\xi_iz-1) \le \phi+\delta_i
\right\},\\[.3ex]
\gamma^3_{\delta_i^{-}}
&=
\left\{
z \;\middle|\;
|z|=1+\eta,\;
-\phi-\delta_i \le \Arg(\xi_iz-1) \le -\phi
\right\},\\[.3ex]
\gamma^3_{\delta_i}
&=
\gamma^3_{\delta_i^{+}}\cup\gamma^3_{\delta_i^{-}},\\[.3ex]
\gamma^4_{\delta_i}
&=
\left\{
z \;\middle|\;
\frac1k \le |z-\overline{\xi_i}|,\;
|z|\le 1+\eta,\;
\Arg(\xi_i z-1)=-\phi
\right\}
\end{align*} and $k$ is sufficiently large.
such that $\gamma^3_{\delta^{+}_i}$ and $\gamma^3_{\delta^{-}_{i+1}}$ intersects in one point and $\delta_{N+1}=\delta_1$(see the figure 1). Also we define $\gamma_{\delta_i}:=\cup^4_{j=1}\gamma^j_{\delta_i}$ and $O'=\sum^N_{i=1}\gamma_{\delta_i}.$
Also assume that $O'$ is positively oriented.
Let us define $P(z)^{\lambda}:=(1-\xi_1z)^{\lambda}\dots(1-\xi_Nz)^{\lambda }.$
Note that $P(z)^{\lambda}=\sum^{\infty} _{k=0}A^{-\lambda-1}_kz^{k}.$ Therefore, by Cauchy integral formula we have 
\begin{align}\label{A-lambda-minus-one}
    A^{-\lambda-1}_k=\frac{1}{2\pi i}\int_{O'}P(z)^{\lambda}\frac {dz}{z^{k+1}}=\sum^N_{i=1}\frac{1}{2\pi i}\int_{\gamma_{\delta_i}}P(z)^{\lambda}\frac{dz}{z^{k+1}}.
\end{align}
\begin{figure}[htbp]
    \centering
    \includegraphics[width=0.8\textwidth]{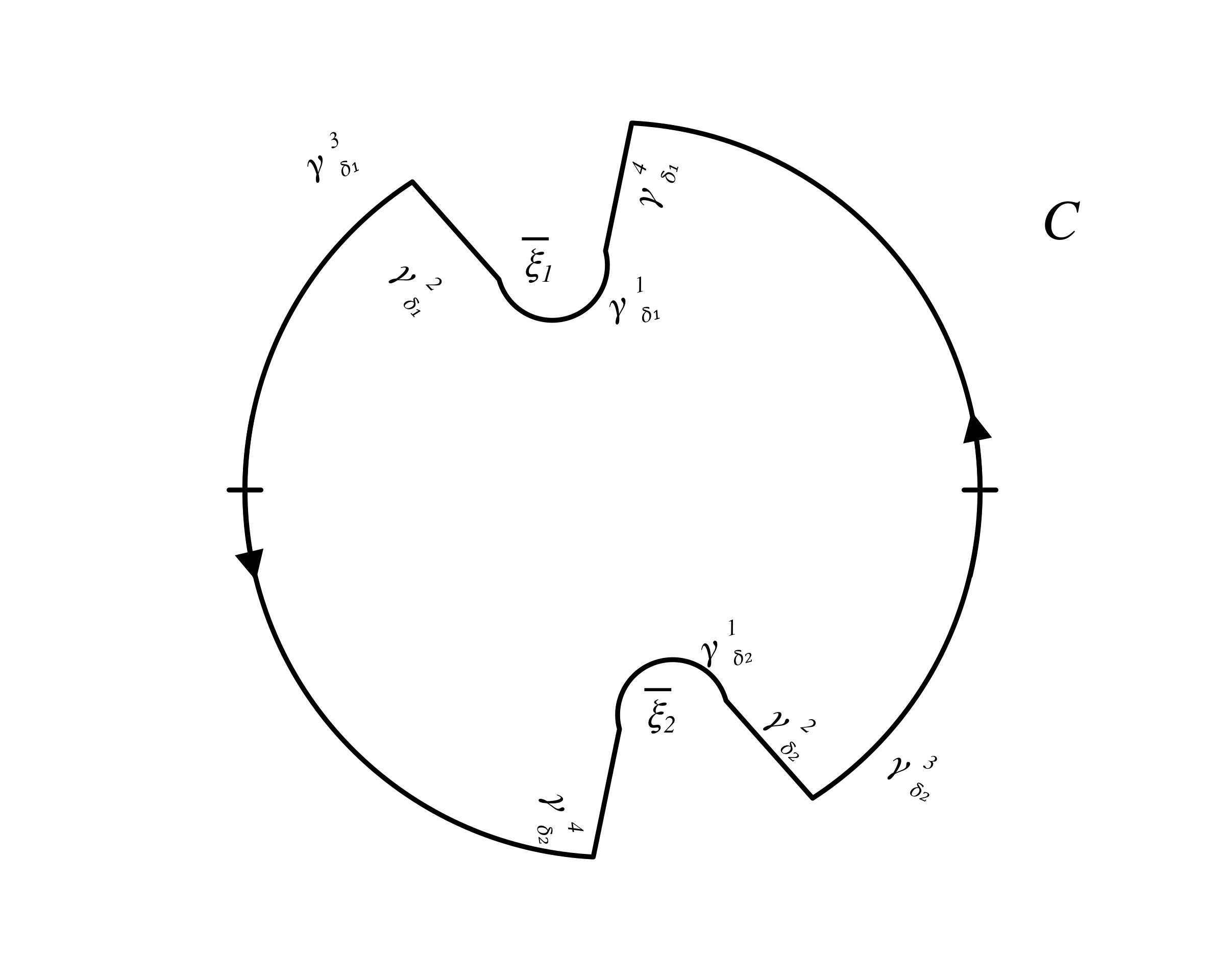}
    \caption{The contour \(C\).}
    \label{fig:contour}
\end{figure}

Now try to estimate the first term of the above equation 
\begin{align*}&\abs{\frac{1}{2\pi i}\int_{\gamma_{\delta_1}}(1-\xi_1z)^{\lambda}\dots(1-\xi_Nz)^{\lambda }\frac{dz}{z^{k+1}}}\\
    &=\abs{\frac{1}{2\pi i}\int_{\gamma_{\delta_1}}(1-{\xi_1}z)^{\beta}\dots(1-\xi_Nz)^{\beta }\prod^N_{j=1}e^{i\gamma\left(ln\abs{(1-\xi_jz)}+iArg((1-\xi_jz)\right)}\frac{dz}{z^{k+1}}}\\&\lesssim e^{2\pi N\lvert{\gamma}\rvert}\int_{\gamma_{\delta_1}}\abs{(1-{\xi_1}z)^{\beta}}\dots\abs{(1-\xi_Nz)^{\beta}}\frac{d\lvert{z}\rvert}{\lvert z\rvert^{k+1}}\\&\lesssim_{E} e^{2\pi N\lvert{\gamma}\rvert} {\int_{{\xi_1}\gamma_{\delta_1}}\abs{(z-1)}^{\beta}\dots\abs{(z-{\xi_1}\overline{\xi_N})}^{\beta }\frac{d\lvert z\rvert}{\lvert z\rvert^{k+1}}}\\&\lesssim_{E} e^{2\pi N\lvert{\gamma}\rvert}\int_{{\xi_1}\gamma_{\delta_1}}\Big\lvert(z-1)\Big\rvert^{\beta}\frac{d\lvert z \rvert}{\lvert{z}\rvert^{k+1}}
    \\&\lesssim_{E} e^{2\pi N|\gamma|}\int_{{\xi_1}\gamma_{\delta_1}}|{z-1}|^{\beta}\frac{d|{z}|}{|{z}|^{k+1}}\lesssim_{E} e^{2\pi N|\gamma|}\int_{\gamma_1\cup\gamma_2\cup\gamma_3\cup\gamma_4}|{z-1}|^{\beta}\frac{d\lvert{z}\rvert}{{\lvert{z}\rvert}^{k+1.}}\\&\underset{\eqref{lemma-Flajolet-paper}}{\lesssim_{\beta,E}} e^{2\pi N|\gamma|}(k+1)^{-\beta-1}.
\end{align*}
Similarly, we have 
$
\displaystyle{\abs{\frac{1}{2\pi i}\int_{\gamma_{\delta_i}}(1-\xi_1z)^{\lambda}\dots(1-\xi_Nz)^{\lambda }\frac{dz}{z^{k+1}}}\lesssim_{\beta} e^{2\pi N|\gamma|}(k+1)^{-\beta-1}},$ where $i=2,\dots,N.$ Now replacing $\lambda$ by $-\lambda-1$ and putting the all estimates in equation \eqref{A-lambda-minus-one}, we get the desired estimate.
\end{proof}
\begin{lemma}\label{D-n-m-to-s-n-m}
Let $m\geq0$, and let $\mathsf T=(T^k)_{k\geq0}$.
We use the convention $\mathsf T_k=0$ for $k<0$. Then, for every $k\geq0$,
\[
\mathscr D_k^m(\mathsf T)=S_k^{-m}.
\]
Moreover, for $k\geq Nm$ we have
\[
S_k^{-m}
=
T^{k-Nm}Q_E(T)^m.
\]
\end{lemma}
\begin{proof}
We proceed by induction on $m$.

For $m=1$, by the definition of $\mathscr D_k$,
\begin{align*}
\mathscr D_k(\mathsf T)
&=
\sum_{j=0}^{N}
\alpha_j\bm{\Delta}_k^j(\mathsf T)\\
&=
\sum_{j=0}^{N}
(-1)^j
\sum_{r=j}^{N}
(-1)^r
\binom{r}{j}e_r
\bm{\Delta}_k^j(\mathsf T)\\
&=
\sum_{r=0}^{N}
(-1)^re_r
\sum_{j=0}^{r}
(-1)^j\binom{r}{j}
\bm{\Delta}_k^j(\mathsf T)\\
&=
\sum_{r=0}^{N}
(-1)^re_r
(\bm{\Delta}_k^0-\bm{\Delta}_k)^r(\mathsf T).
\end{align*}
Since
$(\bm{\Delta}_k^0-\bm{\Delta}_k)(a_k)=a_{k-1}$,
we obtain
$\mathscr D_k(\mathsf T)
=
\sum_{r=0}^{N}
(-1)^re_rT^{k-r},$
where $T^{k-r}=0$ whenever $k-r<0$.

On the other hand, by Lemma~\eqref{lemma-for-alpha-beta} with
$\alpha=-1$ and $\beta=0$,
\[
S_k^{-1}
=
\sum_{i=0}^{k}
A_{k-i}^{-2}S_i^0.
\]
Since $S_i^0=T^i$ and
$A_r^{-2}=(-1)^re_r,
\ 0\leq r\leq N$, while $A_r^{-2}=0$ for $r>N$, we obtain
\[
S_k^{-1}
=
\sum_{r=0}^{N}
(-1)^re_rT^{k-r}.
\]
Hence $\mathscr D_k(\mathsf T)=S_k^{-1}.$ Now suppose that, for some $m\geq1$, $\mathscr D_k^m(\mathsf T)=S_k^{-m},$ $k\geq0.$
Then
\begin{align*}
\mathscr D_k^{m+1}(\mathsf T)
&=
\mathscr D_k
\Big(
(\mathscr D_r^m(\mathsf T))_{r\geq0}
\Big)\\
&=
\mathscr D_k\big((S_r^{-m})_{r\geq0}\big)\\
&=
\sum_{j=0}^{N}
\alpha_j\bm{\Delta}_k^j(S_k^{-m}).
\end{align*}
Again, applying Lemma~\eqref{lemma-for-alpha-beta} with
$\alpha=-1$ and $\beta=-m$, we have
\[
S_k^{-m-1}
=
\sum_{i=0}^{k}
A_{k-i}^{-2}S_i^{-m}.
\]
Putting $r=k-i$ and using
$A_r^{-2}=(-1)^re_r,
\qquad 0\leq r\leq N,$
we get
\[
S_k^{-m-1}
=
\sum_{r=0}^{N}
(-1)^re_rS_{k-r}^{-m},
\]
where $S_j^{-m}=0$ for $j<0$. By the definition of $\mathscr D_k$,
$\sum_{j=0}^{N}
\alpha_j\bm{\Delta}_k^j(S_k^{-m})
=
\sum_{r=0}^{N}
(-1)^re_rS_{k-r}^{-m}.$
Therefore,
$\mathscr D_k^{m+1}(\mathsf T)
=
S_k^{-m-1}.$
Hence the result follows by induction. Finally, for $k\geq Nm$, Lemma~\eqref{lem:negative-integer-means}
gives
\[
S_k^{-m}
=
T^{k-Nm}Q_E(T)^m.
\]
This completes the proof.
\end{proof}

\begin{lemma}\label{p-estimates-for-s-n}
     For two real number $\beta$ and $\gamma$ consider the $M^{\lambda}_n$, $\lambda=\beta+i\gamma$ which is defined in equation \eqref{all-required-definition-Here}.  Then for all $x\in L_p(\mathcal{M})$, where $1<p<\infty$, we have $\Big\lVert \sup\limits_{n\geq 0}\!^{+}M^{\lambda}_nx\Big \rVert_{p}\lesssim_{\beta}e^{2\pi N\gamma^2}\lVert x\rVert_p$,  where $\beta>1.$
\end{lemma}
\begin{proof}
   Let us consider the following computation
    \begin{align*}
        &\Big\lVert \sup\limits _{n\geq 1}\!^{+}M^{\lambda}_nx\Big \rVert_p\\&=\Big\lVert\sup\limits_{n\geq 1}\!^{+}(n+1)^{-\beta-i\gamma}\Big(\sum^{n}_{k=0}A^{\lambda-1}_{n-k}T^kx\Big)\Big\rVert_{p}\\
        &\lesssim\sup_{n\geq1,0\le k\le n}\frac{\lvert A^{\lambda-1}_{n-k}\rvert}{(n+1)^{\beta-1}}\Big\lVert\sup\limits_{n\geq 1}\!^{+}(n+1)^{-1}\Big(\sum^{n}_{k=0}T^{k}x\Big)\Big\rVert_{p}\\&\underset{\eqref{final-estimates-of-A-k-lambda},\eqref{JungeXu}}{\lesssim_{\beta}}e^{2\pi N\gamma ^2}\lVert x\rVert_p.
    \end{align*}
The first inequality in the above follows from Proposition \eqref{dualityofvectorvaluednorm} and positivity of $T.$ This completes the proof of the lemma.
\end{proof}
For a sequence $(a_n)_n$ let us define the operator $\bm\nabla_n(a_n)_n=a_{n+1}-a_n.$ We state the following theorem without proof for later use.
\begin{theorem}\label{thm:abel-summation}[\cite{rudin}, Theorem 3.41]
Let us consider two sequences $(a_n)_{n\geq 0}$ and $(b_n)_{n\geq 0}$ . Then, for $0\leq p\leq q$, we have
\begin{align*}
\sum_{n=p}^{q}\bm{\Delta}_{n}(a_n)b_n
=
-\sum_{n=p}^{q-1}a_n\bm{\nabla}_{n}(b_n)
+a_qb_q-a_{p-1}b_p
\end{align*}
where $\bm{\Delta}$ is defined in \eqref{operator-Delta-definition}.
\end{theorem}
\begin{lemma}\label{lemma-difference-A-eta}
For a complex no $\lambda$, $A_k^\lambda$ be defined as in \eqref{all-required-definition-Here}, and $\alpha_j$, $j=1,\dots,N$ be defined as in equation \eqref{equation-D-n}.
Then, for every $n,k$ such that $n-k\geq N$, we have
\[
\sum_{j=0}^{N}
\alpha_j(-1)^j
\bm{\nabla}_{k}^{\,j}
\bigl(A_{n-k}^{\lambda}\bigr)_k
=
A_{n-k}^{\lambda-1}.
\]
\end{lemma}
\begin{proof}
Set $s=n-k$.
We first observe that
$
\bm{\nabla}_{k}
\bigl(A_{n-k}^{\lambda}\bigr)_k
=
A_{n-k-1}^{\lambda}-A_{n-k}^{\lambda}
=
-\bm{\Delta}_s(A_s^\lambda).
$
Hence, by iteration,
$\bm{\nabla}_{k}^{\,j}
\bigl(A_{n-k}^{\lambda}\bigr)_k
=
(-1)^j
\bm{\Delta}_s^{\,j}(A_s^\lambda),
\ j\geq0.$
Therefore,
$
(-1)^j
\bm{\nabla}_{k}^{\,j}
\bigl(A_{n-k}^{\lambda}\bigr)_k
=
\bm{\Delta}_s^{\,j}(A_s^\lambda).
$
It follows that
\begin{align*}
\sum_{j=0}^{N}
\alpha_j(-1)^j
\bm{\nabla}_{k}^{\,j}
\bigl(A_{n-k}^{\lambda}\bigr)_k
&=
\sum_{j=0}^{N}
\alpha_j
\bm{\Delta}_s^{\,j}(A_s^\lambda)\\
&=
\sum_{j=0}^{N}
\sum_{r=j}^{N}
(-1)^{j+r}
\binom{r}{j}e_r\,
\bm{\Delta}_s^{\,j}(A_s^\lambda)
\\&=
\sum_{r=0}^{N}
(-1)^r e_r
\sum_{j=0}^{r}
(-1)^j
\binom{r}{j}
\bm{\Delta}_s^{\,j}(A_s^\lambda)\\
&=
\sum_{r=0}^{N}
(-1)^r e_r
(I-\bm{\Delta}_s)^r(A_s^\lambda).
=\sum_{r=0}^{N}
(-1)^r e_r A_{s-r}^{\lambda}
\end{align*}
Now, by Lemma \eqref{backward-difference-of-A-k-lambda}, we have
\[
A_s^{\lambda-1}
=
\sum_{r=0}^{N}
(-1)^r e_r A_{s-r}^{\lambda}.
\]

This completes the proof.
\end{proof}
\begin{lemma}\label{lemma-repeated-Abel-N}
Let $N\geq 2$ , $(a_k)_{k\geq0}$ and $(b_k)_{k\geq0}$ be two sequences. Also assume that $b_k=0$, $k<0.$
Then, for every $h\geq N$, there exist constants
$\mu_{j,r}^{(N)}$, $0\leq j\leq N-1$ and $0\leq r\leq N$,
depending only on $E=\{\xi_1,\dots,\xi_N\}$, such that
\begin{align}\label{repeated-Abel-N}
&\sum_{k=0}^{h}
a_k
\sum_{\ell=0}^{N}
\alpha^{(N)}_\ell
\bm{\Delta}_k^\ell(b_k)
\nonumber\\
&\qquad=
\sum_{j=0}^{N-1}\sum_{r=0}^{N}
\mu_{j,r}^{(N)}
a_{h+r}b_{h-j}
+
\sum_{k=0}^{h}
\left(
\sum_{\ell=0}^{N}
(-1)^\ell\alpha^{(N)}_\ell
\bm{\nabla}_k^\ell(a_k)
\right)b_k
\end{align}
where $\alpha^{(N)}_l$, $l=0,\dots,N$ is defined in equation \eqref{definition-of-D-n-alpha}.
\end{lemma}
\begin{proof}
We will proceed by induction on $N.$ Taking $p=0$, $q=h$, from Theorem \eqref{thm:abel-summation} we obtain
\[
\sum_{k=0}^{h}a_k\bm{\Delta}_k(b_k)
=
-\sum_{k=0}^{h-1}b_k\bm{\nabla}_k(a_k)
+a_hb_h,
\]
where we have used $b_{-1}=0$. Since
$a_hb_h
=
a_{h+1}b_h-\bm{\nabla}_h(a_h)b_h$,
it follows that
\begin{equation}\label{abel-convenient-form}
\sum_{k=0}^{h}a_k\bm{\Delta}_k(b_k)
=
a_{h+1}b_h
-
\sum_{k=0}^{h}
\bm{\nabla}_k(a_k)b_k.
\end{equation}

\medskip
\noindent

Applying \eqref{abel-convenient-form} twice, observe that
\begin{align*}
\sum_{k=0}^{h}a_k\bm{\Delta}_k^2(b_k)
&=
a_{h+1}\bm{\Delta}_h(b_h)
-
\sum_{k=0}^{h}
\bm{\nabla}_k(a_k)\bm{\Delta}_k(b_k)\\
&=
a_{h+1}\bm{\Delta}_h(b_h)
-
\bm{\nabla}_{h+1}(a_{h})b_h
+
\sum_{k=0}^{h}
\bm{\nabla}_k^2(a_k)b_k.
\end{align*}
Therefore we obtain
\begin{equation}\label{second-order-abel-N2}
\begin{split}
\sum_{k=0}^{h}a_k\bm{\Delta}_k^2(b_k)
&=
(2a_{h+1}-a_{h+2})b_h
-a_{h+1}b_{h-1}+
\sum_{k=0}^{h}
\bm{\nabla}_k^2(a_k)b_k.
\end{split}
\end{equation}

Combining \eqref{abel-convenient-form} and
\eqref{second-order-abel-N2}, we get
\begin{align*}
&\sum_{k=0}^{h}
a_k
\sum_{\ell=0}^{2}
\alpha^{(2)}_\ell
\bm{\Delta}_k^\ell(b_k)\\
&=
\sum_{k=0}^{h}
\left(
\alpha^{(2)}_0a_k
-\alpha^{(2)}_1\bm{\nabla}_k(a_k)
+\alpha^{(2)}_2\bm{\nabla}_k^2(a_k)
\right)b_k\\
&\quad+
\bigl(\alpha^{(2)}_1+2\alpha^{(2)}_2\bigr)
a_{h+1}b_h
-\alpha^{(2)}_2a_{h+2}b_h
-\alpha^{(2)}_2a_{h+1}b_{h-1}.
\end{align*}
Thus the required formula holds for $N=2$.
Assume that the assertion holds for some $N\geq2$. Thus, for
suitable constants $\mu_{j,r}^{(N)}$ depending only on
$\xi_1,\ldots,\xi_N$, we have
\begin{align}
&\sum_{k=0}^{h}
a_k
\sum_{\ell=0}^{N}
\alpha^{(N)}_\ell
\bm{\Delta}_k^\ell(b_k)
\notag\\
&=
\sum_{j=0}^{N-1}\sum_{r=0}^{N}
\mu^{(N)}_{j,r}
a_{h+r}b_{h-j}
+
\sum_{k=0}^{h}
\left(
\sum_{\ell=0}^{N}
(-1)^\ell\alpha^{(N)}_\ell
\bm{\nabla}_k^\ell(a_k)
\right)b_k.
\label{induction-hypothesis-Abel}
\end{align}

From the definition of the coefficients $\alpha_\ell^{(N)}$, note that
\[
\sum_{\ell=0}^{N}
\alpha_\ell^{(N)}z^\ell
=
\prod_{s=1}^{N}
\bigl((1-\xi_s)+\xi_s z\bigr).
\]
Consequently,
\[
\sum_{\ell=0}^{N+1}
\alpha_\ell^{(N+1)}z^\ell
=
\bigl((1-\xi_{N+1})+\xi_{N+1}z\bigr)
\sum_{\ell=0}^{N}
\alpha_\ell^{(N)}z^\ell,
\]
and hence
\begin{equation}\label{alpha-recursion-N}
\alpha_\ell^{(N+1)}
=
(1-\xi_{N+1})\alpha_\ell^{(N)}
+
\xi_{N+1}\alpha_{\ell-1}^{(N)},
\end{equation}
where
\[
\alpha_{-1}^{(N)}
:=
\alpha_{N+1}^{(N)}
:=
0.
\]
Let us set
$c_k
:=
\sum_{\ell=0}^{N}
\alpha_\ell^{(N)}
\bm{\Delta}_k^\ell(b_k).$
By equation \eqref{alpha-recursion-N},
\[
\sum_{\ell=0}^{N+1}
\alpha_\ell^{(N+1)}
\bm{\Delta}_k^\ell(b_k)
=
(1-\xi_{N+1})c_k+\xi_{N+1}\bm{\Delta}_k(c_k).
\]
Therefore,
\begin{align}
&\sum_{k=0}^{h}
a_k
\sum_{\ell=0}^{N+1}
\alpha_\ell^{(N+1)}
\bm{\Delta}_k^\ell(b_k)\notag\\
&=
(1-\xi_{N+1})\sum_{k=0}^{h}a_kc_k
+
\xi_{N+1}\sum_{k=0}^{h}a_k\bm{\Delta}_k(c_k).
\notag\\&
=\sum_{k=0}^{h}\widetilde a_kc_k
+
\xi_{N+1} a_{h+1}c_h,
\label{Nplusone-reduction}
\end{align}
where
\begin{align}\label{equation-of-tilde-a}
\widetilde a_k
:=
(1-\xi_{N+1})a_k-\xi_{N+1}\bm{\nabla}_k(a_k)
=
a_k-\xi_{N+1} a_{k+1}.
\end{align}

Applying the induction hypothesis
\eqref{induction-hypothesis-Abel} to the sequences
$(\widetilde a_k)$ and $(b_k)$ gives
\begin{align*}
\sum_{k=0}^{h}\widetilde a_kc_k
&=
\sum_{j=0}^{N-1}\sum_{r=0}^{N}
\mu_{j,r}^{(N)}
\widetilde a_{h+r}b_{h-j}\\
&\quad+
\sum_{k=0}^{h}
\left(
\sum_{\ell=0}^{N}
(-1)^\ell\alpha_\ell^{(N)}
\bm{\nabla}_k^\ell(\widetilde a_k)
\right)b_k.
\end{align*}
Since
$\widetilde a_{h+r}
=
a_{h+r}-\xi_{N+1} a_{h+r+1},
$
the first term is a linear combination of
$a_{h+r}b_{h-j}$,
\
$0\leq r\leq N+1$,\
$0\leq j\leq N-1.$
Moreover,
$\xi_{N+1} a_{h+1}c_h
=
\xi_{N+1} a_{h+1}
\sum_{\ell=0}^{N}
\alpha_\ell^{(N)}
\bm{\Delta}_h^\ell(b_h).$
Since
$\bm{\Delta}_h^\ell(b_h)
\underset{\eqref{lem:backward-difference}}{=}
\sum\limits_{j=0}^{\ell}
(-1)^j\binom{\ell}{j}b_{h-j},
$
this is a linear combination of
$a_{h+1}b_{h-j},
\ 0\leq j\leq N.$
Thus all the boundary terms in
\eqref{Nplusone-reduction} can be written as
\[
\sum_{j=0}^{N}\sum_{r=0}^{N+1}
\mu_{j,r}^{(N+1)}
a_{h+r}b_{h-j},
\]
where the constants $\mu_{j,r}^{(N+1)}$ depend only on
$\xi_1,\ldots,\xi_{N+1}$.

Finally, from equation \eqref{equation-of-tilde-a} we have
\begin{align*}
&\sum_{\ell=0}^{N}
(-1)^\ell\alpha_\ell^{(N)}
\bm{\nabla}_k^\ell(\widetilde a_k)\\
&=
(1-\xi_{N+1})
\sum_{\ell=0}^{N}
(-1)^\ell\alpha_\ell^{(N)}
\bm{\nabla}_k^\ell(a_k)
-
\xi_{N+1}
\sum_{\ell=0}^{N}
(-1)^\ell\alpha_\ell^{(N)}
\bm{\nabla}_k^{\ell+1}(a_k)\\
&\underset{\eqref{alpha-recursion-N}}{=}
\sum_{\ell=0}^{N+1}
(-1)^\ell
\alpha_\ell^{(N+1)}
\bm{\nabla}_k^\ell(a_k),
\end{align*}. Hence
\begin{align*}
&\sum_{k=0}^{h}
a_k
\sum_{\ell=0}^{N+1}
\alpha_\ell^{(N+1)}
\bm{\Delta}_k^\ell(b_k)\\
&=
\sum_{j=0}^{N}\sum_{r=0}^{N+1}
\mu_{j,r}^{(N+1)}
a_{h+r}b_{h-j}\\
&\quad+
\sum_{k=0}^{h}
\left(
\sum_{\ell=0}^{N+1}
(-1)^\ell
\alpha_\ell^{(N+1)}
\bm{\nabla}_k^\ell(a_k)
\right)b_k.
\end{align*}
This completes the induction.
\end{proof}

\begin{theorem}\label{firsttimerealiseddiofferentialop}
    For a real number $\gamma$ let $\lambda=-m+i\gamma$, $m\in\mathbb{N}.$ Then for all $x\in L_2(\mathcal{M})$, we have $$\Big\lVert \sup\limits_{n\geq 1}\!^{+}M^{\lambda}_{n}x\Big\rVert_2\lesssim_{m} e^{C\gamma^2}\lVert x\rVert_{2}$$
for some constant $C.$  
\end{theorem}
\begin{proof}
    Given $n\in \mathbb{N}$, set $n_0=[\frac{n}{2}]$.
    So using Lemma \eqref{lemma-for-alpha-beta}, we have 
    \begin{align*}
        S^{-m+i\gamma}_nx&=\sum^{n_0}_{k=0}A^{i\gamma}_{n-k}S^{-m-1}_kx+\sum^{n}_{k=n_0+1}A^{i\gamma}_{n-k}S^{-m-1}_kx
    \end{align*}
 Now, observe that 
    \begin{align*}
        &\Big\lVert\sup_{n\geq1}\!^{+} (n+1)^{m-i\gamma}\sum^n_{k=n_0+1} A^{i\gamma}_{n-k}S^{-m-1}_kx\Big\rVert_2\\&\underset{\eqref{maxcomingoutjuxulem}}
        {\le}\sup_{n\geq1}\Big[{(n+1)^{m}\sum^n_{k=n_0+1}}\frac{\lvert A^{i\gamma}_{n-k} \rvert}{(k+1)^{m+1}}\Big]\Big\lVert\mathop{\sup\nolimits^{+}}\limits_{k\geq N(m+1)}(k+1)^{m+1}S^{-m-1}_kx\Big\rVert_2+\sum_{j=1}^{2N(m+1)-1}\|M_j^{-m+i\gamma}x\|_2\\&\lesssim_{m}e^{2N\pi\gamma^2}\Big\lVert \mathop{\sup\nolimits^{+}}\limits_{k\geq N(m+1)}(k+1)^{m+1}S^{-m-1}_kx\Big\rVert_2+e^{2\pi N\gamma^2}\lVert x\rVert_2\\&\lesssim_{m}e^{2N\pi\gamma^2}\Big\lVert\mathop{\sup\nolimits^{+}}\limits_{k\geq N(m+1)}(k+1)^{m+1}\Theta^{m+1}_{k-(m+1)N}x\Big\rVert_2+e^{2\pi N\gamma^2}\lVert x\rVert_2\underset{\eqref{strong-maximal-theorem-of-B-k-m}}{\lesssim_{m}} e^{2\pi N\gamma^2}\lVert x\rVert_2.
    \end{align*}
    From the  Lemma \eqref{D-n-m-to-s-n-m},   we have 
    \begin{align*}
        \sum^{n_0}_{k=0}A^{i\gamma}_{n-k}S^{-m-1}_k&=\sum^{n_0}_{k=0} A^{i\gamma}_{n-k}\mathscr{D}^{m+1}_{k}(\mathsf{T})\\&=\sum^{n_0}_{k=0}A^{i\gamma}_{n-k}\sum^N_{j=0}\alpha_j\bm{\Delta}^j_k(S^{-m}_k)\\&=\sum^N_{j=0}\sum^{n_0}_{k=0}\alpha_jA^{i\gamma}_{n-k}\bm{\Delta}^j_k(S^{-m}_k).
    \end{align*}
   
From Lemma \eqref{lemma-repeated-Abel-N} , for some constants $\mu^{(N)}_{j,r}$, $j=0,\dots,N-1$ and $r=0,\dots,N$ which are depended on only $\xi_j$, $j=1,
\dots,N$ and $N$,  we have 
\begin{align*}
    &\sum^{n_0}_{k=0}A^{i\gamma}_{n-k}S^{-m-1}_k\\&=\sum^{N-1}_{j=0}\sum^N_{r=0}\mu^{(N)}_{j,r}A^{i\gamma}_{n-n_0-r}S^{-m}_{n_0-j}+\sum^{n_0}_{k=0}{\sum^N_{j=0}(-1)^j\alpha_j\bm{\nabla}^j_{k}(A^{i\gamma}_{n-k})}S^{-m}_k\\&\underset{\eqref{lemma-difference-A-eta}}{=}\sum^{N-1}_{j=0}\sum^N_{r=0}\mu^{(N)}_{j,r}A^{i\gamma}_{n-n_0-r}S^{-m}_{n_0-j}+\sum^{n_0}_{k=0}A^{i\gamma-1}_{n-k}S^{-m}_k.
\end{align*}
    Similarly proceed forward to $m$ step, we have \begin{align}\label{equation-last-abel-summation}
    \sum^{n_0}_{k=0}A^{i\gamma}_{n-k}S^{-m-1}_k=\sum^{m}_{l=0}\sum^{N-1}_{j=0}\sum^N_{r=0}\mu^{(N)}_{j,r}A^{i\gamma-l}_{n-n_0-r}S^{-m+l}_{n_0-j}+\sum^{n_0}_{k=0}A^{i\gamma-m-1}_{n-k}S^{0}_k. 
\end{align}
Therefore, we obtain that
\begin{align}\label{inequality-to-estimates-the-given-sum-zero-to-n-zero}
    &\Big\lVert \sup\limits_{n\geq 1}\!^{+}(n+1)^m\sum^{n_0}_{k=0} {A^{i\gamma}_{n-k}S^{-m-1}_kx}\Big\rVert_2\\& \underset{\eqref{maxcomingoutjuxulem}}{\leq}\sum^{N-1}_{j=0}\sum^N_{r=0}\lvert\mu^{(N)}_{j,r}\rvert\sum^m_{l=0}\sup_{n}(n+1)^m\frac{\abs{A^{i\gamma-l}_{n-n_0-r}}}{(n_0+1)^{m-l}}\left\lVert \sup\!^{+}_{n\geq 2N(m+1)}(n_0+1)^{m-l}S^{-m+l}_{n_0-j}(x)\right\rVert_2\notag\\&\qquad\qquad+\sup_{n}(n+1)^m\sum^{n_0}_{k=0}\abs{A^{i\gamma-m-1}_{n-k}}\lVert \sup\limits\!^{+}_{k\geq 0} S^{0}_k(x)\rVert_2\notag\\&\underset{\eqref{final-estimates-of-A-k-lambda},\eqref{strong-maximal-theorem-of-B-k-m},\eqref{strong-maximal-ergodic-theorem-for-l2}}{\lesssim_{m}} e^{2\pi N\gamma^2} \lVert x\rVert_2\notag.
\end{align}
This completes the proof of the theorem.
\end{proof}
\subsection{Proof of the Theorem \texorpdfstring{\eqref{strong-ergodic-theorem-for-ritt-E}}{Theorem}}
\begin{proof}Let $\alpha =\beta+i\gamma$ where $\beta,\gamma \in \mathbb{R}.$ Choose $\theta \in(0,1),$ $q\in(1,\infty)$, $m\in\mathbb{N}$ and $b>\max\{\beta, 1\}$ such that $\frac{1}{p}=\frac{1-\theta}{2}+\frac{\theta}{q}$ and $\beta=-(1-\theta)m+\theta b.$ Let $x\in L_p(\mathcal{M})$ and $y=(y_n)$ be a finite sequence in $L_{p'}(\mathcal{M})$ such that $\lVert x\rVert_p$ and $\lVert y\rVert_{L_{p'}(\mathcal{M},\ell_1)}<1$.
Let us consider the following function $
    f(z)=u|{x}|^{\frac{p(1-z)}{2}+\frac{pz}{q}}$ where $x=u|{x}|.$
By the Proposition \eqref{Interpolation}, there is a function $g=(g_n)_n$ continuous on the strip $\{z\in\mathbb{C}:0\leq Re(z)\leq 1\}$ and analytic in the interior such that 
$g(\theta)=y$ and 
$$\sup_{t\in\mathbb{R}}max\{\lVert g(it)\rVert_{L_2(\mathcal{M},\ell_1)},\lVert g(1+it)\rVert_{L_{q'}(\mathcal{M},\ell_1)}\}<1.$$
Let us define the analytic function on the interior of the strip \begin{align*}
    F(z)
&=e^{\delta(z^2-\theta^2)}\sum_{n}\tau[M^{(z-1)m+zb+i\gamma}_n(f(z))g_n(z)]
\end{align*}
for some $\delta>0$ which will be chosen later.
For any $t\in\mathbb{R}$, consider the following 
\begin{align*}
    F(it)&=e^{\delta(-t^2-\theta^2)}\sum_n\tau[M^{(z-1)m+zb+i\gamma}_n(f(z))g_n(z)]\\&\le e^{\delta(-t^2-\theta^2)}\Big\lVert M^{-m+i(tm+tb+\gamma)}_nf(it)\Big\rVert_{L_{2}(\mathcal{M},l_{\infty})}\lVert g_{n}(it)\rVert_{L_2(\mathcal{M},l_1)}\\&\lesssim_{m,\gamma}e^{\delta(-t^2-\theta^2)}\Big( e^{C(tm+tb+\gamma)^2}\Big)\lVert f(it) \rVert_2
    \\&\lesssim_{m,\gamma}\Big(e^{(-\delta+c_{m,b,\gamma})t^2-\delta\theta^2}\Big)\lVert f(it) \rVert_2
\end{align*}
Similarly, by Lemma \eqref{p-estimates-for-s-n} we have 
$|F(1+it)|\leq C_{\alpha,q,\gamma}e^{\delta}e^{-\delta(t^2+\theta^2)+c'_{m,b,\gamma}t^2}.$ Now choosing $\delta>\max\{c_{m,b,\gamma},c'_{m,b,\gamma }\}$, we have
\begin{align*}
    \sup_{t\in\mathbb{R}}max\{|{F(it)}|,|{F(1+it)}|\}< \infty.
\end{align*}
Therefore by the three lines lemma, we have
$|{F(\theta)}|\le C_{p,\beta,b,\gamma,N}.$
This yields
\begin{align*}
    |\sum_{n}\tau\Big( M^{\alpha}_n(x)y_n\Big)|<C_{\alpha,p} .
\end{align*}
This completes the proof of the theorem.
\end{proof}
\subsection{Contractively regular operators on classical $L_p$-spaces with finite peripheral spectrum} \label{section5}
  In this subsection we consider positive contractions on classical $L_p$-spaces. We show that due to spectral restrictions one can easily deduce maximal ergodic inequalities from the results of the positive contractive Ritt operators on classical $L_p$-spaces. We begin with the following definitions to state our results.
\begin{definition}($\RE$ Operator) Let $E=\{\xi_1,\xi_2,...,\xi_N\}$ be a subset of $\mathbb{T}$ . Let $T\in B(X)$, we say that $T$ is $\RE$ operator if $ \sigma(T)\subseteq\overline{\mathbb{D}}$ and there is a constant $c>0$ such that 
 \begin{equation*}
     \lVert R(z,T) \rVert \leq c\prod^N_{j=1}\lvert\xi_j-z\rvert^{-1}, \qquad z\in \mathbb{C},1<\lvert z \rvert<2.
 \end{equation*}
 \end{definition}
When $E=\{1\}$ the operator $T$ is called a Ritt operator \cite{Le-Merdy-Xu}. We refer to \cite{oualid-Ritt-E-operato}, \cite{bouabdillah2024squarefunctionsassociatedritte},\cite{parasar-samya} and \cite{MondalPalaiRay2026} on more about $\RE$ operators.
The following result provides an equivalent characterization of
$\RE$ operators.
\begin{theorem}[\cite{oualid-Ritt-E-operato}, Theorem 2.10 ]\label{equivalent-definition-Ritt-E}
An operator $T\in{B}(X)$ is a $\operatorname{Ritt}_{E}$ operator
if and only if the following two conditions hold:
\begin{enumerate}[\rm(i)]
    \item $T$ is power bounded; that is, there exists a constant
    $c_{0}\geq 1$ such that
    \[
        \lVert T^{n}\rVert\leq c_{0},
        \qquad n\geq 0.
    \]

    \item There exists a constant $c_{1}>0$ such that
    \[
        \left\|
            T^{n-1}\prod_{j=1}^{N}(\xi_{j}I-T)
        \right\|
        \leq \frac{c_{1}}{n},
        \qquad n\geq 1.
    \]
\end{enumerate}
\end{theorem}
\begin{definition}(Regular operator)
Let $1\leq p\leq \infty$ and
$T:L_{p}(\Omega)\longrightarrow L_{p}(\Omega)$
be a bounded operator. We say that $T$ is regular if there exists a constant $C\geq 0$ such that
\[
\left\|\sup_{k\geq 1}|T(x_k)|\right\|_{p}
\leq
C\left\|\sup_{k\geq 1}|x_k|\right\|_{p}
\]
for every finite sequence $(x_k)_{k\geq 1}$ in $L_{p}(\Omega)$.
 We say $T$ is contractively regular if the above inequality holds with $C=1$.

\end{definition}
    By repeating the proof of Theorem \eqref{strong-maximal-ergodic-theorem-for-l2} in the commutative setting, together with the square function estimate \cite[Corollary 6.4]{bouabdillah2024squarefunctionsassociatedritte} and \cite[Theorem 4.7]{oualid-Ritt-E-operato} one obtains the following theorem of maximal ergodic inequality for contractively regular operators on $L_p(\Omega)$, $1<p<\infty$  with finite peripheral spectrum. Indeed it is an extension  of the maximal ergodic theorem due to Le Merdy and Xu \cite{Le-Merdy-Xu} for contractively regular analytic operators and, consequently, strengthen Stein’s classical maximal inequality \cite{Stein-ergodic-theorem}.
 \begin{theorem}\label{maximal-inequality-classical-L-p}
     Let $1< p < \infty$. Assume that $T:L_p(\Omega)\rightarrow L_p(\Omega)$, is a contractively regular $\RE$ operator. Then we have the following
     \begin{align*}
         \left\lVert \sup_{n\geq1}\abs{T^nf}\right\rVert_p\lesssim \lVert f\rVert_p \ \text{for all} \ f\in L_p(\Omega).
     \end{align*}
     
 \end{theorem}
\begin{lemma}\label{T-power-M-Ritt}
   Let $M\in\mathbb N$ and $T$ be a $\RE$ operator such that $E=\{\xi_1,\dots,\xi_N\}$ and  $F=E\cap\sigma(T)\subseteq\{1,\omega,\omega^2,\dots,\omega^{M-1}\}$ where $\omega^{M}=1$.Then $T^{M}$ is Ritt operator.
\end{lemma}
\begin{proof}
    Since $T$ is power bounded, then $T^N$ is also power bounded.
Consider the following,
\begin{align*}
    &\lVert nT^{M(n-1)}(T^{M}-I)\rVert\\&\leq\lVert n T^{n-1}(T^{M}-I)\rVert\lVert T^{(M-1)(n-1)}\rVert\\&=\lVert n T^{n-1}\prod_{\omega^j\in F}(T-\omega^jI)\prod_{\omega^j\neq F}(T-\omega^jI)\rVert\lVert T^{(M-1)(n-1)}\rVert
    \\&\lesssim\Big\lVert nT^{n-1}\prod_{\omega^j\in F}(T-\omega^jI)\prod_{\xi_j\neq F}(T-\xi_jI)\Big\rVert\Big\lVert \prod_{w^j\not\in F}(T-\omega^jI)\prod_{\xi_j\neq F}(T-\xi_jI)^{-1}\Big\rVert\lVert T^{(M-1)(n-1)}\rVert<\infty
\end{align*}
for all $n\geq1.$ Hence by Theorem \eqref{equivalent-definition-Ritt-E} the proof of the lemma is completed.
\end{proof}
\begin{theorem}
    Let $T$ be a positive contraction which is also a $\RE$ operator on $L_p(\Omega)$, $1<p<\infty$ and $Q_E(T)$ defined as in equation \eqref{Q-E}. Then $\Big\lVert\sup\limits_{n\geq 1}(n+1)^m\lvert T^nQ_E(T)^mf\rvert\Big\rVert_p\lesssim\lVert f\rVert_p$ for all $m\geq0$ and $f\in L_p(\Omega).$ 
\end{theorem}
\begin{proof}
   Since $L_p(\Omega)$, $1<p<\infty$ is complex Banach lattice and $T$ is a contractively positive $\RE$ operator on $L_p$, then $F=E\cap\sigma(T)\subseteq\{1,\omega,\omega^2,\dots,\omega^{M-1}\}$ such that $\omega^{M}=1$ for some $M\geq1$ by [\cite{peripheralspectrumofpositiveoperators},Theorem 7.1].
   From Lemma \eqref{T-power-M-Ritt}, we deduce that $T^{M}$ is a  $Ritt$ operator which is also contractive and positive.
   Therefore, from \cite[Theorem 4.1]{Le-Merdy-Xu} we have
   \begin{align*}
       \Big\lVert \sup_{n\geq1}(n+1)^m\lvert T^{Mn}(T^{M}-I)^mf\rvert\Big\rVert_{p}\lesssim\lVert f\rVert_p.
   \end{align*}
    Therefore, for $j=0,1,\dots,M-1$ we have
   \begin{align}\label{esitimates-T-to-power-M}
       &\Big\lVert \sup_{n\geq1}(n+1)^m\lvert (T^{M}-I)^mT^{Mn+j}f\rvert\Big\rVert_p\lesssim\lVert T^jf\rVert_p\lesssim\lVert f\rVert_p.
   \end{align}
   One can see that $T^{M}-I=\prod\limits_{{\omega}^j\in F,\ 0\leq j\leq M}(T-\omega^jI)R(T)=Q_E(T)R(T)$ where $ R(T)=\prod\limits_{{\omega}^j\notin F,\ 0\leq j\leq M}(T-\omega^jI)\prod\limits_{{\xi_j}\notin F}(T-\xi_jI)^{-1}.$ Since $\omega^j\notin E$,
    $0\le j\le M$ implies that $\omega^j\notin \sigma(T)$. Therefore $R(T)$ becomes invertible.
     Observe that 
     \begin{align*}
     &\Big\lVert\sup_{n\geq1}(n+1)^m \lvert T^nQ_E(T)^mf\rvert\Big\rVert_p\\&=
         \Big\lVert\sup_{n\geq1}(n+1)^m \lvert T^n(T^{M}-I)^mR(T)^{-m}f\rvert\Big\rVert_p\\&\leq\sum^{M-1}_{j=0}\Big\lVert \sup_{n\geq1}(nM+j+1)^m\lvert T^{Mn+j}(T^{M}-I)^mR(T)^{-m}f\rvert\Big\rVert_p
         \\&\lesssim \sum^{M-1}_{j=0}\sup_{n\geq1}\frac{(nM+j+1)^m}{(n+1)^m}\Big\lVert \sup_{n\geq1}(n+1)^m\lvert T^{Mn+j}(T^{M}-I)^mR(T)^{-m}f\rvert\Big\rVert_p\underset{\eqref{esitimates-T-to-power-M}}{\lesssim}\lVert f\rVert_p.
     \end{align*}
     This completes the proof of the theorem.
\end{proof}
\begin{example}\label{example-of-contractivly-regular-Ritt-E-on-nc-L-p}
Let $\mathcal{M}=M_2(\mathbb{C})$ and choose
$\xi=e^{2\pi i\theta},
\ \text{where}\ \theta\ \text{is an irrational number}.$
Consider the unitary operator
\[
u=
\begin{pmatrix}
1 & 0\\
0 & \xi
\end{pmatrix}
\in\mathcal{M},
\]
and define $T:\mathcal{M}\longrightarrow\mathcal{M},\
T(x)=uxu^{*}.$
 Then $T$ is a completely positive complete isometry on $L_p(\calM)$ for all $1\leq p\leq\infty.$ 
Note that $T$ becomes an operator satisfying \eqref{thosefourconditions} whose peripheral spectrum $\{1,\xi,\overline{\xi}\}$ which is not cyclic.
\end{example}

\begin{remark}
    Note that because of the above example Theorem \eqref{strong-ergodic-theorem-for-ritt-E} applies to a large class of positive contractions on noncommutative $L_p$-spaces. It is easy to construct contractively regular $\RE$ operators with peripheral spectrum noncyclic as well. Moreover, one can show  that irrational rotation does not satisfy the maximal inequality in Theorem \eqref{maximal-inequality-classical-L-p}(see \cite{Stein-ergodic-theorem-book}).
\end{remark}
\section{Spectral Obstructions to Variational Inequalities}\label{kiarboli}

We begin with an extrapolation result for
$Ritt_E$-operators in the framework of interpolation of Banach spaces.
In particular, this extends the corresponding extrapolation theory for
classical Ritt operators to the broader class of $Ritt_E$-operators \cite{Fackler-Stephan, Blunk}. We refer to \cite{Fackler-Stephan} and references therein for the relevant preliminaries used in the following proposition. The proof of the following proposition is motivated from \cite{Lemerdybook}
\begin{proposition}\label{Blunck-RE-interpolation}
Let $X_i$, $i=0,1$ be Banach spaces such that $(X_0,X_1)$ forms an interpolation couple and let $X_{\theta}$ be a complex interpolation
 space of exponent $\theta\in(0,1)$ with respect to $(X_0,X_1)$. Let
 $T_{0}\in{B}(X_0)$ and $T_{1}\in{B}(X_1)$ be consistent
operators, and let $T_{\theta}\in{B}(X_{\theta})$ be the corresponding
 interpolated operator. Assume that $T_0$ is a power bounded operator and that
 $T_{1}$ is a $\RE$ operator. Then $T_{\theta}$ is a $\RE$ operator.
 \end{proposition}
\begin{proof} Since $T_{0}:X_0\rightarrow X_0$ is power bounded and $T_1:X_1\rightarrow X_1$ is $\RE$ as well as power bounded operator. By complex interpolation we have $T_{\theta}:X_{\theta}\rightarrow X_{\theta}$ is also power bounded operator. To prove $T_{\theta}$ is $\RE$, for all $n\geq N$, it is sufficient to prove 
\begin{align*}
    \lVert (n+1)T^nQ_{E}(T)\rVert_{X_{\theta}\to X_{\theta}}&=\lVert (n+1) T^{n-N}T^{N}Q_{E}(T)\rVert_{X_\theta\rightarrow X_{\theta}}\\&=\lVert (n+1)S^{-1}_{n}T^N\rVert_{X_{\theta}\rightarrow X_{\theta}}\\&=\lVert M^{-1}_nT^N\rVert_{X_{\theta}\rightarrow X_{\theta}}<\infty
\end{align*}
where $S^{-1}_n$ and $M^{-1}_{n}$ are defined in equation \eqref{all-required-definition-Here}. Since $T:X_{0}\rightarrow X_{0}$ is a power bounded operator. Then we have the following 
\begin{align}\label{equation-of-M-for-exterpolation-theorem}
    \lVert M^{\beta+i\gamma}_{n}\rVert_{X_0\rightarrow X_0}&=\lVert(n+1)^{-\beta-i\gamma}S^{\beta+i\gamma}_{n}\rVert_{X_{0}}\\&=(n+1)^{-\beta}\left\lVert\sum^n_{k=0}A^{\beta+i\gamma-1}_{n-k}T^k\right\rVert_{X_0\rightarrow X_0}\notag\\&\lesssim (n+1)^{-\beta}\sum^n_{k=0}\lvert A^{\beta-1+i\gamma}_{n-k}\rvert\underset{\eqref{final-estimates-of-A-k-lambda}}{\lesssim} e^{2\pi N\gamma^2}\notag.
\end{align}
As $T\in B(X_1)$ be a $\RE$ operator, then we have 
\begin{align*}
    \lVert S^{-m-1}_k \rVert_{X_1\rightarrow X_1}\underset{\eqref{equivalent-definition-Ritt-E}}{\lesssim}\frac{1}{(k+1)^{m+1}}.
\end{align*}
we consider $n\in\mathbb{N}$ and we set 
$\tilde{n}=[\frac{n}{2}]$. For convenience we assume that $n\geq 2N(m+1)$ so that $\tilde{n}\geq N(m+1).$
Let consider the following
\begin{align*}
    \sum^n_{k=0}A^{i\gamma}_{n-k}S^{-m-1}_k= V_1(n,\gamma)+V_2(n,\gamma)
\end{align*}
where $V_1(n,\gamma):=\sum\limits^{\tilde{n}}_{k=0}A^{i\gamma}_{n-k}S^{-m-1}_{k}$ and $V_2(n,\gamma):=\sum\limits^{n}_{k=\tilde{n}+1}A^{i\gamma}_{n-k}S^{-m-1}_{k}.$
Note that \begin{align*}
    (n+1)^m\lVert V_2(n,\gamma)\rVert_{X_1\rightarrow X_1}&\leq (n+1)^m\sum^n_{k=\tilde{n}+1}\lvert A^{i\gamma}_{n-k}\rvert\lVert S^{-m-1}_k\rVert_{X_1\rightarrow X_1}\\&\lesssim (n+1)^m\sum^n_{k=\tilde{n}+1}e^{2\pi N\gamma^2}(k+1)^{-(m+1)}\\&\lesssim e^{2\pi N\gamma^2}(n+1)^m(n-\tilde{n})(\tilde{n}+2)^{-(m+1)}\lesssim e^{2\pi N\gamma^2}\lesssim e^{3\pi N\gamma^2}.
\end{align*}
From equation \eqref{equation-last-abel-summation}, we have
\begin{align*}
V_1(n,\gamma)=\sum^{m}_{l=0}\sum^{N-1}_{j=0}\sum^N_{r=0}\mu_{j,r}A^{i\gamma-l}_{n-\tilde{n}-r}S^{-m+l}_{\tilde{n}-j}+\sum^{\tilde{n}}_{k=0}A^{i\gamma-m-1}_{n-k}S^{0}_k.
\end{align*}  Similarly as in Lemma \eqref{firsttimerealiseddiofferentialop}, we have
\begin{align*}
    (n+1)^m\lVert V_1(n,\gamma)\rVert_{X_1\rightarrow X_1}\lesssim e^{2\pi N\gamma^2}\lesssim e^{3\pi N\gamma^2}.
\end{align*}
Therefore we obtain that 
\begin{align}\label{X-1-norm-estimate-of-M-n}
    \lVert M^{-m+i\gamma}_n\rVert_{X_1\rightarrow X_1}\leq (n+1)^m \lVert V_1(n,\gamma)\rVert_{X_1\rightarrow X_1}+ (n+1)^m \lVert V_2(n,\gamma)\rVert_{X_1 \rightarrow X_1}\lesssim e^{3\pi N \gamma^2}.
\end{align}
Let $m\geq 1$ be an integer such that $\theta m>1$. Then consider $\beta=\frac{(\theta m-1)}{(1-\theta)}.$ This is a positive real number chosen so that 
$-1=(1-\theta)\beta-\theta m.$ Let $S=\{z:0\leq\Re(z)\leq 1\}.$ Consider the following analytic function 
\begin{align*}
    \Psi(z)=e^{\delta(z^2-\theta^2)}M^{-mz+\beta(1-z)}_{n}(T)
\end{align*}
where $\delta$ is a fix real number such that $\delta\geq 3\pi N(\beta+m)^2.$
Observe that 
\begin{align}
    \lVert \psi(it)\rVert_{X_0\rightarrow X_0}&=\lVert e^{\delta(-t^2-\theta^2)}M^{\beta-it(m+\beta)}_{n}(T)\rVert_{X_0\rightarrow X_0}\\&\underset{\eqref{equation-of-M-for-exterpolation-theorem}}{\lesssim} e^{-\delta\theta^2}e^{\left(-\delta+2\pi N(\beta+m)^2\right)t^2}\notag. 
\end{align}
At the end we have
\begin{align*}
    \lVert \Psi(1+it)\rVert_{X_1\rightarrow X_1}=\lVert e^{\delta(1-t^2-\theta^2)}M^{-m-i(\beta+m)t}_{n}(T)\rVert_{X_1\rightarrow X_1}\underset{\eqref{X-1-norm-estimate-of-M-n}}{\lesssim} e^{\delta}e^{\left(3\pi N(\beta+m)^2-\delta\right)t^2}\lesssim e^{\delta}.
\end{align*}
By complex interpolation theorem we have 
\begin{align*}
    \lVert \Psi(\theta)\rVert_{X_{\theta}\rightarrow X_{\theta}}<\infty.
\end{align*}

This completes the proof of the proposition.
\end{proof}
\begin{lemma}\label{lemma-for-ritt-E-form-numerical-range}
    Let $1<p<\infty$ and $T:\mathcal{M} \rightarrow \mathcal{M}$ be an operator such that it satisfies the  conditions (i), (ii), (iii) and (iv) of \eqref{thosefourconditions}. Then $T:L_p(\mathcal{M})\rightarrow L_p(\mathcal{M})$ is a $\RE$ operator. 
\end{lemma}
\begin{proof}
    Note that it follows from Lemma \eqref{lemma-for-T-to-use-square-functio-estimates} that $T:L_2(\calM)\to L_2(\calM)$ is a $Ritt_E$ operator. The proof is a direct consequence of  Proposition \eqref{Blunck-RE-interpolation}.
\end{proof}
The following theorem shows that the variational inequality for a
general Ritt$_E$ operator does not hold when \(E\) contains a point
different from \(1\).
Let $(a_n)_{n\geq 0}$ be a sequence of complex number and $1\le q<\infty$. The strong-$q$-variation norm is defined as 
\begin{equation*}
    \lVert (a_n)_{n\geq0}\rVert_{v_q}:=\sup\{(\lvert a_0\rvert^q+\sum_{k\geq1}\lvert a_{n_k}-a_{n_{k-1}}\rvert^q)^{\frac{1}{q}}\}
\end{equation*}
where the supremum has been taken over all increasing sequences $(n_k)_{k\geq0}$ of integers such that $n_0=0.$ It is well-known that the set $v_q$ of all sequences with finite strong-$q$-variation is a Banach space with the norm $\| .\|_{v^q}$
\begin{proposition}\label{counterexmale}
Let $1<p<\infty$, $2<q<\infty$, and let
$E=\{\xi_1,\ldots,\xi_N\}\subset \mathbb{T}$,
with $\xi_1=1$,
be a finite set.
Then for any $m\geq0$ the variational inequality
\begin{equation}\label{variationinequality}
\left\|
\bigl(
n^m T^n Q_E(T)^m x
\bigr)_{n\geq 1}
\right\|_{L_p(v^q)}
\lesssim
\|x\|_p
\end{equation}
holds for every contractively regular Ritt$_E$ operator $T$ if and only if
 $E=\{1\}$.
\end{proposition}
\begin{proof}
    The `if' direction follows from \cite[Theorem 4.4]{le-Merdy-Xu-q-variational-inequality}. We now prove the converse. Suppose that $E\neq\{1\}$. Take
$\xi\in E$, such that $\xi\neq 1$ and define  $T:\ell_p\rightarrow \ell_p$ by
$Te_j=(1-\frac{1}{2^j})\xi e_j,\ j\geq1$ where $(e_j)_{j\geq1}$ denotes the canonical basis of $\ell_p$.
Arguing as in \cite[Theorem 4.36]{Lemerdybook}, we see that $T$ is a $\RE$ operator. 
Let us set
$
a_n^{(j)}
:=
n^m((1-\frac{1}{2^j})\xi)^nQ_E((1-\frac{1}{2^j})\xi)^m.
$ 
Note that  $\abs{Q_{E}(\left(1-\frac{1}{2^j}\right)\xi)}\geq c\frac{1}{2^j}$ for some  $c>0$.
Now consider $n$ satisfying $2^j\leq n\leq2^{j+1}.$ Also observe that $\lim\limits_{j\rightarrow\infty}(1-\frac{1}{2^j})^{2^j}=e^{-1}$. Therefore  for sufficiently large $j$ we have the estimate $n^m\abs{Q_{E}((1-\frac{1}{2^j})\xi)}^m(1-\frac{1}{2^j})^n\geq (1-\frac{1}{2^j})^{2^{j+1}}c^m\geq c_1^m$ for some positive constant $c_1$.
Note that,
\[
|a_{n+1}^{(j)}-a_n^{(j)}|
=
n^m|Q_{E}(\left(1-\frac{1}{2^j}\right)\xi)|^m \left(1-\frac{1}{2^j}\right)^n
\left|
\left(1-\frac{1}{2^j}\right)\xi
\left(1+\frac1n\right)^m
-1
\right|.
\]
One can see that,
$\left(1-\frac{1}{2^j}\right)\xi
\left(1+\frac1n\right)^m
\longrightarrow \xi$
as $j\to\infty$. Since $\xi\neq1$, it follows from the previous observations that there exists $c_2>0$ such
that, for all sufficiently large $j$,
$|a_{n+1}^{(j)}-a_n^{(j)}|
\geq c_2$
 for every integer $n$ that
satisfies $2^j
\leq n
\leq 2^{j+1}$.
Hence for $q\geq2$, we have
\begin{align}\label{a2jbound}
\left\|
(a_n^{(j)})_{n\geq1}
\right\|_{v^q}
&\geq
\left(
\sum_{n=2^j}
^{2^{j+1}-1}
|a_{n+1}^{(j)}-a_n^{(j)}|^q
\right)^{1/q}\gtrsim
(2^j)^{1/q}\to\infty.\ \text{as}\ j\to\infty.
\end{align}
On the other hand, putting $x=e_j$ in \eqref{variationinequality} we obtain
$\left\|
\bigl(
n^m T^n Q_E(T)^m e_j
\bigr)_{n\geq 1}
\right\|_{L_p(v^q)}=\left\|
(a_n^{(j)})_{n\geq1}
\right\|_{v^q}\lesssim 1$ which clearly contradicts \eqref{a2jbound}.
\end{proof}
\begin{remark}\label{filter-variation-inequality}
Let $1<p<\infty$ and $q>2.$ Let $T:L_p(\Omega)\to L_p(\Omega)$ be a contractively regular $Ritt_E$ operator with $E=\{\xi_1,\dots,\xi_N\}.$ Then the variational estimate may fail due
to the oscillation of the sequence $(\xi^n_i)_{n\geq 1}$, whenever $\xi_i\neq 1$. However, we can prove a variant of variational inequality proved by Le Merdy and Xu \cite[Theorem 4.4]{le-Merdy-Xu-q-variational-inequality}. Let us denote
\[
\ell_j(z)=\prod_{k\neq j}\frac{z-\xi_k}{\xi_j-\xi_k}.
\]
Fix $j\in\{1,\ldots,N\}$ and set $S_j=\overline{\xi_j}T,\ L_j=\ell_j(T),\ \text{and}\ y_j=L_jx$ where $x\in L_p(\Omega).$ Moreover,
\[
(S_j-I)L_j
=
\overline{\xi_j}(T-\xi_j I)
\prod_{k\neq j}
\frac{T-\xi_kI}{\xi_j-\xi_k}
=
d_jQ_{E}(T),
\]
where
$d_j=
\frac{\overline{\xi_j}}
{\prod_{k\neq j}(\xi_j-\xi_k)}$.
Hence 
$S^n_j(S_j-I)y_j
=
d_j\overline{\xi_j}^{\,n}T^nQ_{E}(T)x.
$ By the square function estimate for contractively regular
$\operatorname{Ritt}_E$ operators  (see \cite{oualid-Ritt-E-operato},\cite{bouabdillah2024squarefunctionsassociatedritte}),
\[
\left\|
\left(
\sum_{n\geq 0}
(n+1)|S^n
_j(S_j-I)y_j|^2
\right)^{1/2}
\right\|_p
\lesssim
\|x\|_p.
\]
Therefore, as $S_j$ is contractively regular, proceeding similarly to the proof of \cite[Theorem 4.4]{le-Merdy-Xu-q-variational-inequality}, we obtain $\left\|
\left(S^n_jy_j\right)_{n\geq 1}
\right\|_{L_p(v^q)}
\lesssim
\|x\|_p.$
 Consequently, $\left\|
\left(
\overline{\xi_j}^{\,n}T^n\ell_j(T)x
\right)_{n\geq1}
\right\|_{L_p(v^q)}
\lesssim
\|x\|_p.$
\end{remark}

\section{Noncommutative weak-type inequalities for convolution powers}\label{sectionczo}
\subsection{Noncommutative Calder\'on-Zygmund decomposition} Let $\mathcal N:=\ell_\infty(\mathbb Z)\overline{\otimes}\mathcal M$ with trace $\varphi:=d\mu\otimes \tau,$ where $d\mu$ is the counting measure on $\mathbb Z.$ Note that $L_1(\calN)$ can be identified isometrically with the Bochner space $\ell_1(\mathbb Z;L_1(\calM)).$  Let us describe the natural dyadic filtration on \(\mathbb{Z}\). For \( k \ge 0 \), define the dyadic intervals
\[
I_{k,m} = \{ m2^k, m2^k + 1, \dots, (m+1)2^k - 1 \}, \quad m \in \mathbb{Z}.
\]
Then $\{I_{k,m}:m\in\mathbb Z\}$ is a partition of $\mathbb Z$ for all $k\geq 0.$ Let $N\in\mathbb N.$ Define $\mathcal D_k:=\{I_{N-k,m}:m\in\mathbb Z\}$ for all $1\leq k\leq N.$ Let us define
$\mathcal{F}_k$ to be the $\sigma$-algebra generated by $ \mathcal D_k .$ Then the sequence \( \mathcal{F}_{k}\) for all \({1\leq k\leq N} \) is an increasing sequence of $\sigma$-algebras, where $\mu$ is the counting measure. Let $\mathcal{E}_k$ be the corresponding conditional expectation operator for all $1\leq k\leq N.$ Then $\mathcal{N}_k:=\ell_\infty(\mathbb{Z},\mathcal{F}_k,d\mu|_{\mathcal{F}_k})\overline{\otimes}\mathcal{M}$ is an increasing sequence of von Neumann algebras and $\mathbb{E}_k:=\mathcal{E}_k\otimes id_{\mathcal M}$ will be the corresponding conditional expectation operators, for $1\leq k\leq N.$ 
Let $f \in L_1(\mathcal{N}) \cap \mathcal{N}$ be a positive and finitely supported function, and set $
f_n = \mathbb{E}_{n} f,$ where $1\leq n\leq N.$  Then $(f_n)_{1\leq n\leq N}$ is a noncommutative martingale.
Let us recall the following construction due to Cuculescu \cite{Cuculescu}, which is the starting point of noncommutative Calder\'{o}n--Zygmund decompositions introduced in \cite{Cadilhac-Conde-Alonso-Parcet-2022}.
\begin{lemma}
Let \( f \in L_1(\mathcal{N})_+ \) and \( \lambda > 0 \). Then there exists a non-increasing sequence of projections $\{q_j = q_j(f,\lambda)\}_{0\leq j\leq N}$
with \( q_0 = \mathbf{1}_{\mathcal{N}} \) such that the following properties hold:
\begin{itemize}
    \item[(i)] $q_j \in \mathcal{N}_j,$  for all $0\leq j \leq N.$
    \item[(ii)] $q_j\mathbb{E}_j(f)q_j \leq \lambda q_j$, for all $1\leq j \leq N.$
    \item[(iii)] $q_j$ commutes with $q_{j-1}\mathbb{E}_j(f)q_{j-1},$ for all $1\leq j\leq N.$
    \item[(iv)] If $q = \bigwedge_{0\leq j \leq N} q_j$, then $q f q \leq \lambda q\ \text{and}\ \varphi(\mathbf{1}_{\mathcal{N}}-q) \leq \frac{\|f\|_1}{\lambda}.$
\end{itemize}
\end{lemma}
Now given $f \in L_1(\mathcal{N})_{+}$ with $f$ nonzero only on finitely many points and $\lambda>0$. Choose $N\in\mathbb N$ large enough such that $2^{1-N}\sum_{n\in\mathbb Z}\|f(n)\|_{\mathcal M}\leq \lambda.$ Consider Cuculescu's projections $q_j = q_j(f,\lambda)$ and define
$p_j = q_{j-1} - q_j \in \mathcal{N}_j$ for $1\leq j \leq N,$
which yields the identity \[\sum_{1\leq j \leq N} p_j = \mathbf{1}_{\mathcal{N}}-q.\]
Note that $p_j$ and $q_j$ further can be described as
\[
p_j = \sum_{Q \in \mathcal{D}_j} p_Q,
\quad \text{where } p_Q = \mathbf{1}_Q\otimes \pi_Q, \ \pi_Q\in\mathcal{P}(\mathcal M)
\]
and
\[
q_j = \sum_{Q \in \mathcal{D}_j} q_Q,
\quad \text{where } q_Q = \mathbf{1}_Q\otimes \xi_Q,\ \xi_{Q}\in\mathcal{P}(\mathcal M).
\]
We have that $p_jf_jq_j=0$ for all $1\leq j\leq N.$
The following noncommutative Calder\'on-Zygmund decomposition is from \cite{Cadilhac-Conde-Alonso-Parcet-2022}. \[
f=g+b_d+b_{\mathrm{off}},\] where
\begin{align}\label{ncCZO}
&g=qfq+\sum_{j \geq 1}p_jf_jp_j, \\\notag
&b_d=\sum_{j \geq 1}b_{d,j},\ \text{where},\ b_{d,j}:=p_j(f-f_j)p_j, \\\notag
&b_{\mathrm{off}}=\sum_{j\geq 1}b_{\mathrm{off},j}\ \text{where},\
b_{\mathrm{off},j}:= p_j (f - f_j) q_j + q_j (f - f_j) p_j.
\end{align}
Moreover, $b_{\mathrm{off},j}=p_jfq_j+q_jfp_j.$
\begin{lemma}\cite{Cadilhac-Conde-Alonso-Parcet-2022}\label{Cadilhac-Conde-Alonso-Parcet-2022}The regular decomposition \eqref{ncCZO} satisfies:
\begin{itemize}
    \item[(i)] \( \|g\|_1 \leq \|f\|_1 \) and \( \|g\|_\infty \leq 2\, \lambda \),
    \item[(i)] $\sum_{j \geq 1} \|b_{d,j}\|_1 \leq 2 \|f\|_1,$ {and}
    $\mathbb{E}_j(b_{d,j}) = \mathbb{E}_j(b_{\mathrm{off},j}) = 0$ for all $j\geq 1.$
\end{itemize}
\end{lemma}
 For a dyadic interval $Q=\{a,\ldots,a+L-1\}$, define $5Q:=\{a-2L,\ldots,a+3L-1\}$ and its discrete centre by $c_Q:=a+\left\lfloor\frac{L-1}{2}\right\rfloor.$
Then $|5Q|=5|Q|$. We write
$\mathbf1_{5Q}\pi_Q$ for the projection
$\mathbf1_{5Q}\otimes\pi_Q\in\mathcal N$.Define
\[
\zeta := \mathbf{1}_{\mathcal N} - \bigvee_{1\leq j \leq N} \ \bigvee_{Q \in \mathcal{D}_j} \mathbf{1}_{5Q} \, \pi_Q.
\]
\begin{lemma}\label{lem5.3}
 The projection \( \zeta \) satisfies
\[
\varphi(\mathbf{1}_{\mathcal N} - \zeta) \,\lesssim\, \frac{\|f\|_{L^1(\calN)}}{\lambda},
\]
and we note $\zeta(x)\, p_j(y) = p_j(y)\, \zeta(x) = 0.$
whenever \( y \in Q \in \mathcal{D}_j \) and \( x \in 5Q \).
\end{lemma}
\begin{theorem}\label{thm5.4}
Let $(\mu_n)$ be a sequence of probability measures on $\mathbb{Z}$, and for finitely supported function $f : \mathbb{Z} \to \mathcal{M}$ define the convolution operators
\[
(\text{Conv}_{\mu_n}f)(x): = (\mu_n * f)(x), \quad x \in \mathbb{Z}.
\]
Let us assume that for each $n \geq 1$,
\[
|\mu_n(x+y) - \mu_n(x)| \leq C\frac{|y|}{|x|^{2}}, \quad \text{for } x,y \in \mathbb{Z}, \; 0 < 2|y| \leq |x|.
\]
Then the family $(\text{Conv}_{\mu_n})_{n\geq 1}$ is of weak-type $(1,1)$; i.e., there exists a constant $C' > 0$ such that for any $\lambda > 0$, and $f\in \ell_1(\mathbb{Z}; L_1(\mathcal M))$ there exist a projection $e\in\mathcal{P}(\ell_\infty(\mathbb{Z})\overline{\otimes}\mathcal M)$ such that
\[
\|e(\text{Conv}_{\mu_n}f)e\|_\infty\leq \lambda,\quad \text{and}\quad \varphi(e^{\perp})\leq  \frac{C'}{\lambda} \|f\|_{\ell_1(\mathbb{Z}; L_1(\mathcal M))}.
\]
\end{theorem}
\begin{proof}
It is enough to prove the weak type inequality for $f\in\calN\cap L_1(\calN)_{+}$ such that $\|f\|_\infty$ has compact support. Let $f\in L_1(\calN)_{+}$ as assumed in the noncommutative Calder\'on-Zygmund decomposition. Note that each $\text{Conv}_{\mu_n}$ is a positive unital operator and hence $\text{Conv}_{\mu_n}g\leq 2\lambda$ for all $n\geq 1.$ Hence clearly we need to focus on the bad parts. 

To estimate the ${b_d}$-term, we first rewrite the $j$-th diagonal term as a sum of its
restrictions to dyadic cubes in $\mathcal{D}_j$
\[
{b_d}
= \sum_{j \geq 1} \sum_{Q \in \mathcal{D}_j}
\mathbf{1}_Q \otimes \pi_Q (f - f_Q) \pi_Q
=: \sum_{j \geq 1} \sum_{Q \in \mathcal{D}_j} {b_{d,Q}},
\]
where
\[
f_Q = \frac{1}{|Q|} \sum_{x\in Q} f(x).
\]

Observe that this means ${b_{d,Q}} = \mathbf{1}_Q \, b_{d,j}$ for $Q \in \mathcal{D}_j$.
Assume that we have fixed a cube $Q \in \mathcal{D}_j$. Since $\operatorname{supp}_{\mathbb{Z}}(b_{d,Q}) \subseteq Q$, if $x \notin 5Q$ and $c_Q\in Q$ is the center, then using the kernel representation and the fact that
$\mathbb{E}_j(b_{d,Q})= 0,$
we get
\begin{align*}
\zeta(x)\, \text{Conv}_{\mu_n}(b_{d,Q})(x)\, \zeta(x)
&= \zeta(x)\Big( \sum_{y\in\mathbb{Z}} \mu_n(x-y)\, b_{d,Q}(y) \Big)\zeta(x) \\
&= \zeta(x)\Big( \sum_{y\in Q} \big[\mu_n(x-y) -\mu_n(x-c_Q)\big] b_{d,Q}(y) \Big)\zeta(x)\\
&\leq C\zeta(x)\big(\sum_{y\in Q}\frac{|y-c_Q|}{|x-c_Q|^{2}}|b_{d,Q}(y)|\big)\zeta(x)\\
&\leq C\zeta(x)B_Q(x)\zeta(x),
\end{align*}
where $B_Q(x):=\mathbf{1}_{(5Q)^c}(x)\sum_{y\in Q}\frac{|y-c_Q|}{|x-c_Q|^{2}}|b_{d,Q}(y)|.$ By a similar computation we also get \[-\zeta(x)\, B_Q(x)\, \zeta(x)C\leq  \zeta(x)\text{Conv}_{\mu_n}(b_{d,Q})(x)\zeta(x)\] for $x\notin 5Q.$
On the other hand, if $x \in 5Q$, then by Lemma~\eqref{lem5.3},
\[
\zeta(x)\, \text{Conv}_{\mu_n}(b_{d,Q})(x)\, \zeta(x) = 0.
\] Therefore, we obtain 
\begin{equation}\label{eq:bdQ-estimate}
-\zeta\, B_Q(x)\zeta\, C\leq \zeta\, \mathrm{Conv}_{\mu_n}(b_{d,Q})\, \zeta
\leq
C\, \zeta B_Q \zeta.
\end{equation}

Summing over $j \geq 1$ and $Q \in \mathcal{D}_j$, we get
\begin{equation}\label{eq:bd-estimate}
-\zeta\big(\sum_{j \geq 1} \sum_{Q \in \mathcal{D}_j}
B_Q\big)\zeta\, C\leq\zeta\, \mathrm{Conv}_{\mu_n}(b_d)\, \zeta
\leq
C\, \zeta\big(\sum_{j \geq 1} \sum_{Q \in \mathcal{D}_j}
B_Q\big)\zeta
=: F_d.
\end{equation}
Note that by  Lemma \eqref{Cadilhac-Conde-Alonso-Parcet-2022}
\begin{equation}\label{eq:F1-L1}
\|F_d\|_1
\lesssim
\sum_{j \geq 1}\sum_{Q\in\mathcal{D}_j} \sum_{x\in(5Q)^c}\sum_{y\in Q}\frac{|y-c_Q|}{|x-c_Q|^{2}}\|b_{d,Q}(y)\|_1
\leq
\|f\|_1.
\end{equation}
In the  above inequality we have  used that for $y\in Q,$ $|y-c_Q|\leq |Q|$ and $\sum_{x\in(5Q)^c}\frac{1}{|x-c_Q|^{2}}\lesssim\frac{1}{|Q|}.$
Now let $e_d := \mathbf{1}_{[0,\lambda]}(F_d).$
By Chebyshev's inequality, we have
\begin{equation}\label{eq:chebyshev}
\varphi(e_d^{\perp})
\leq
\frac{\|F_d\|_1}{\lambda}.
\end{equation}

On the other hand, $e_d F_d e_d \leq \lambda e_d,$
which implies $
(e_d \wedge \zeta)\, F_d\, (e_d \wedge \zeta)
\leq
\lambda\, (e_d \wedge \zeta).$ Combining with \eqref{eq:bd-estimate}, we obtain
\begin{equation}\label{eq:final-estimate}
(e_d \wedge \zeta)\,
\big( \zeta\, \mathrm{Conv}_{\mu_n}(b_d)\, \zeta \big)\,
(e_d \wedge \zeta)
\leq
C\, (e_d \wedge \zeta)\, F_d\, (e_d \wedge \zeta)
\leq
C\, \lambda\, (e_d \wedge \zeta).
\end{equation}
Similarly, $(e_d \wedge \zeta)\,
\big( \zeta\, \mathrm{Conv}_{\mu_n}(b_d)\, \zeta \big)\,
(e_d \wedge \zeta)\geq -C\lambda\,e_d \wedge \zeta. $ On the other hand we observe
\begin{equation}\label{eq:trace-bound}
\varphi\big( (e_d \wedge \zeta)^{\perp} \big)
\leq
\varphi(e_d^{\perp}) + \varphi(\zeta^{\perp})
\leq
C \frac{\|f\|_1}{\lambda}.
\end{equation}

Now we deal with the off-diagonal estimate. Recall the $j$-th term of $b_{\mathrm{off}}$:
\begin{equation}\label{eq:boff-j}
b_{\mathrm{off},j}=p_jfq_j+q_jfp_j
=:b^{a}_{\mathrm{off},j}+b^{b}_{\mathrm{off},j}.
\end{equation}
For each $Q\in\mathcal D_j$, define
$b^{a}_{\mathrm{off},Q}:=\mathbf1_Q b^{a}_{\mathrm{off},j}
=\mathbf1_Q\pi_Qf\xi_Q$,
$b^{b}_{\mathrm{off},Q}:=\mathbf1_Q b^{b}_{\mathrm{off},j}
=\mathbf1_Q\xi_Qf\pi_Q$, and
$b_{\mathrm{off},Q}:=b^{a}_{\mathrm{off},Q}+b^{b}_{\mathrm{off},Q}$.
Then both terms are supported on $Q$ and
\[
\sum_{y\in\mathbb Z}b^{a}_{\mathrm{off},Q}(y)
=\sum_{y\in\mathbb Z}b^{b}_{\mathrm{off},Q}(y)=0.
\]
If $Q=\{a,\ldots,b\}\in\mathcal D_j$, and define
$P_{Q,r}:=\sum_{y=a}^{r}f(y)$,
$A_{Q,r}:=\pi_QP_{Q,r}\xi_Q$, and
$H_{Q,r}:=A_{Q,r}+A_{Q,r}^{*}$ for all $a\leq r\leq b$.
Notice that $H_{Q,r}=\sum_{y=a}^{r}b_{\mathrm{off},Q}(y)$ and $H_{Q,b}=0$ by Lemma \eqref{Cadilhac-Conde-Alonso-Parcet-2022}.
Therefore, by Abel's summation formula, we obtain
\begin{equation}\label{eq:boff-Abel}
\text{Conv}_{\mu_n}(b_{\mathrm{off},Q})(x)
=\sum_{r=a}^{b-1}
\big[\mu_n(x-r)-\mu_n(x-r-1)\big]H_{Q,r}
\end{equation}
for all $x\in\mathbb Z$.
Notice that if $x\notin5Q$ and $r\in Q$, then
$|x-r|\geq2|Q|\geq2$ and
$|x-a|\leq|x-r|+|Q|\leq\frac32|x-r|$.
Hence the kernel assumption gives
\[
|\mu_n(x-r)-\mu_n(x-r-1)|
\lesssim\frac{1}{|x-r|^2}
\lesssim\frac{1}{|x-a|^2},
\]
uniformly in $n$ and $r$.
To this end, define $w_Q(x):=|x-a|^{-2}$ for $x\notin5Q$ and $w_Q(x):=0$ otherwise, and set
\[
\widetilde B_Q(x):=w_Q(x)\sum_{r=a}^{b-1}|H_{Q,r}|,
\] where we define $\widetilde B_Q=0$ if $|Q|=1$.

Since each $H_{Q,r}$ is self-adjoint, we have
$-|H_{Q,r}|\leq H_{Q,r}\leq|H_{Q,r}|$.
The coefficients in \eqref{eq:boff-Abel} are real. Thus the previous estimate,
together with Lemma~\ref{lem5.3} for $x\in5Q$, gives
\begin{equation}\label{eq:boffQ-majorant}
-C\zeta\widetilde B_Q\zeta
\leq\zeta\,\text{Conv}_{\mu_n}(b_{\mathrm{off},Q})\,\zeta
\leq C\zeta\widetilde B_Q\zeta
\end{equation}
for all $n\geq1$.

We now estimate $\|\widetilde B_Q\|_1$. Set
$m_Q:=\sum_{y\in Q}\tau(\pi_Qf(y)\pi_Q)$.
First note that $0\leq P_{Q,r}\leq P_{Q,b}=|Q|f_Q$.
By Cuculescu's estimate, $\xi_Qf_Q\xi_Q\leq\lambda\xi_Q$, yielding
\[
0\leq\xi_QP_{Q,r}\xi_Q\leq\lambda |Q|\xi_Q,
\qquad
\|\xi_QP_{Q,r}\xi_Q\|_\infty\leq\lambda |Q|.
\]
Write $A_{Q,r}=(\pi_QP_{Q,r}^{1/2})(P_{Q,r}^{1/2}\xi_Q)$.
Then by H\"older's inequality we obtain
\begin{align*}
\|A_{Q,r}\|_2
&\leq\|\pi_QP_{Q,r}^{1/2}\|_2
       \|P_{Q,r}^{1/2}\xi_Q\|_\infty\\
&=\tau(\pi_QP_{Q,r}\pi_Q)^{1/2}
       \|\xi_QP_{Q,r}\xi_Q\|_\infty^{1/2}
\leq m_Q^{1/2}\sqrt{\lambda |Q|}.
\end{align*}
Note that $A_{Q,r}=\pi_QA_{Q,r}$, and hence the support of
$|A_{Q,r}^{*}|$ is contained in $\pi_Q$.
Therefore, by Cauchy--Schwarz, we obtain
\begin{align*}
\|A_{Q,r}\|_1
&=\tau(|A_{Q,r}^{*}|)
=\tau(\pi_Q|A_{Q,r}^{*}|)\\
&\leq\tau(\pi_Q)^{1/2}\|A_{Q,r}\|_2
\leq\sqrt{\lambda |Q|\tau(\pi_Q)m_Q}.
\end{align*}
Consequently, $\|H_{Q,r}\|_1\leq2\|A_{Q,r}\|_1
\leq2\sqrt{\lambda |Q|\tau(\pi_Q)m_Q}$.
Also, since $|x-a|\geq2|Q|$ whenever $x\notin5Q$, we have
\[
\sum_{x\in\mathbb Z}w_Q(x)
\lesssim\frac{1}{|Q|}.
\]
Using the positivity of $\widetilde B_Q$, we now obtain
\begin{align*}
\|\widetilde B_Q\|_1
&=\sum_{x\in\mathbb Z}\tau(\widetilde B_Q(x))
=\left(\sum_{x\in\mathbb Z}w_Q(x)\right)
  \sum_{r=a}^{b-1}\|H_{Q,r}\|_1\\
&\lesssim\sqrt{\lambda |Q|\tau(\pi_Q)m_Q}.
\end{align*}
Now summing over $j\geq1$ and $Q\in\mathcal D_j$, and applying
Cauchy--Schwarz once more, we get
\begin{align*}
\sum_{j\geq1}\sum_{Q\in\mathcal D_j}\|\widetilde B_Q\|_1
&\lesssim
\left(\lambda\sum_{j\geq1}\sum_{Q\in\mathcal D_j}
|Q|\tau(\pi_Q)\right)^{1/2}
\left(\sum_{j\geq1}\sum_{Q\in\mathcal D_j}m_Q\right)^{1/2}.
\end{align*}
However,
\[
\sum_{j\geq1}\sum_{Q\in\mathcal D_j}|Q|\tau(\pi_Q)
=\sum_{j\geq1}\varphi(p_j)
=\varphi(\mathbf1_{\mathcal N}-q)
\leq\frac{\|f\|_1}{\lambda},
\]
and, since the $p_j$ are mutually orthogonal,
\[
\sum_{j\geq1}\sum_{Q\in\mathcal D_j}m_Q
=\sum_{j\geq1}\varphi(p_jfp_j)
=\varphi\big(f^{1/2}(\mathbf1_{\mathcal N}-q)f^{1/2}\big)
\leq\varphi(f)=\|f\|_1.
\]
This implies that
$\sum_{j\geq1}\sum_{Q\in\mathcal D_j}\|\widetilde B_Q\|_1
\lesssim\|f\|_1$.
Define
\[
F_{\mathrm{off}}
:=\zeta\left(\sum_{j\geq1}\sum_{Q\in\mathcal D_j}
\widetilde B_Q\right)\zeta.
\]
Then $F_{\mathrm{off}}\geq0$ and
$\|F_{\mathrm{off}}\|_1\lesssim\|f\|_1$.
Moreover, summing \eqref{eq:boffQ-majorant}, we have
\begin{equation}\label{eq:boff-majorant}
-CF_{\mathrm{off}}
\leq\zeta\,\text{Conv}_{\mu_n}(b_{\mathrm{off}})\,\zeta
\leq CF_{\mathrm{off}}
\end{equation}
for all $n\geq1$.

Now let $e_{\mathrm{off}}:=\mathbf1_{[0,\lambda]}(F_{\mathrm{off}})$
and set $e_{\mathrm{final}}:=e_d\wedge\zeta\wedge e_{\mathrm{off}}$, where
$e_d=\mathbf1_{[0,\lambda]}(F_d)$ is the projection chosen for the diagonal part.
By Chebyshev's inequality and \eqref{eq:trace-bound}, we get
\[
\varphi(e_{\mathrm{final}}^{\perp})
\leq\varphi\big((e_d\wedge\zeta)^{\perp}\big)
+\varphi(e_{\mathrm{off}}^{\perp})
\lesssim\frac{\|f\|_1}{\lambda}.
\]
Since $e_{\mathrm{final}}\leq e_{\mathrm{off}}$ and $e_{\mathrm{final}}\leq\zeta$, we have
$e_{\mathrm{final}}F_{\mathrm{off}}e_{\mathrm{final}}\leq\lambda e_{\mathrm{final}}$.
Thus \eqref{eq:boff-majorant} yields
\[
-C\lambda e_{\mathrm{final}}
\leq e_{\mathrm{final}}\,\text{Conv}_{\mu_n}(b_{\mathrm{off}})\,e_{\mathrm{final}}
\leq C\lambda e_{\mathrm{final}},
\qquad n\geq1.
\]
Combining this with the diagonal estimate and
$\|\text{Conv}_{\mu_n}g\|_\infty\leq2\lambda$, we obtain
\[
\sup_{n\geq1}\|e_{\mathrm{final}}\,\text{Conv}_{\mu_n}(f)\,e_{\mathrm{final}}\|_\infty
\lesssim\lambda,
\qquad
\varphi(e_{\mathrm{final}}^{\perp})\lesssim\frac{\|f\|_1}{\lambda}.
\]
This completes the proof.
\end{proof}

\section{Individual Ergodic Theorem}\label{individualergodicthm}
To establish noncommutative individual ergodic theorems, we will need the notion of a.u. and b.a.u. convergence which were introduced in \cite{Lance,jajte85nclimitbook}.

\begin{definition}Let $(x_n)_{n\geq 1}\subseteq L_0(\mathcal M)$ be a sequence and $x\in L_0(\mathcal M).$ We say that the sequence $(x_n)_{n\geq 1}$ converges to $x$ {almost uniformly} (in short a.u.) if for any $\varepsilon>0$ there exists a projection $e\in\mathcal M$ such that \[\tau(e^{\perp})<\varepsilon \quad \text{ and } \quad \lim_{n\to\infty}\|(x_n-x)e\|_\infty=0.\]
	We say that $(x_n)_{n\geq 1}$ converges to $x$ {bilaterally almost uniformly} (in short b.a.u.) if for any $\varepsilon>0$ there exists a projection $e\in\mathcal M$ such that \[\tau(e^{\perp})<\varepsilon \quad \text{ and } \quad \lim_{n\to\infty}\|e(x_n-x)e\|_\infty=0.\]
\end{definition}
It follows from Egorov's theorem that in the case of classical probability spaces, the above definitions  are equivalent to the usual notion of almost everywhere convergence.

Let $1\leq p\leq\infty.$ We denote $L_p(\mathcal M;\ell_\infty^c)$ to be the space of all sequences $x=(x_n)_{n\geq 1}\subseteq L_p(\mathcal M)$ admitting a factorization of the form $x_n=y_na $ for all $n\geq 1,$ where $a\in L_p(\mathcal M)$ and $(y_n)_{n\geq 1}\subseteq L_\infty(\mathcal M)$ with $\sup\limits_{n\geq 1}\|y_n\|_\infty<\infty.$ We define  \[\|x\|_{L_p(\mathcal M;\ell_\infty^c)} =\inf\big\{\|a\|_p\sup_{n\geq 1}\|y_n\|_{\infty}\big\},\] where the infimum has been taken over all possible factorizations. We denote by $L_p(\mathcal M; c_0)$  and $L_p(\mathcal M; c_0^c)$ to be the closure of all finite sequences in $L_p(\mathcal M;\ell_\infty)$ and $L_p(\mathcal M;\ell_\infty^c)$ respectively. We refer to \cite{musat03interpolationbmolp} and \cite{junge-maximal-theorem-type-in-nc-L-p} for more information on these spaces.
We mention the following proposition gives a very useful for criterion for checking b.a.u. and a.u. convergence of sequences in noncommutative $L_p$-spaces.
\begin{proposition}[\cite{junge-maximal-theorem-type-in-nc-L-p}]\label{ct1}
	\emph{(i)} Let $1\leq p<\infty$ and   $(x_n)_{n\geq 1}\in L_p(\mathcal M,c_0).$ Then $x_n\to 0$ b.a.u. as $n\to \infty.$
	
	\emph{(ii)}  Let $2\leq p<\infty$ and $(x_n)_{n\geq 1}\in L_p(\mathcal M,c_0^c).$ Then $x_n\to 0$ a.u. as $n\to\infty.$ 
\end{proposition}
We can prove the following corollary in a similar way as in  \cite[Corollary 3.5]{Bekjan2008} using \eqref{strong-ergodic-theorem-for-ritt-E} and Kadison-Schwarz inequality.
\begin{corollary}\label{coloum-maximal-inequality}
Let $2<p\leq\infty$ and $T:\mathcal{M}\rightarrow \mathcal{M}$ be $2$-positive operator such that it satisfies the conditions in \eqref{thosefourconditions}. Then 
\[
\left\|\bigl(T^n(x)\bigr)_{n\geq 0}\right\|_
{L_p(\mathcal{M};\ell_\infty^c)}
\leq C_{p}\|x\|_p,
\qquad \text{for all}\ x\in L_p(\mathcal{M}).
\]
\end{corollary}
  \begin{lemma}[\cite{MondalPalaiRay2026}, Lemma 5.7]
   Let $T$ be a power bounded operator on a reflexive Banach space $X$ and $E=\{\xi_1,\dots,\xi_N\}\subseteq\mathbb T.$ Then we have the following decomposition
    \begin{equation}
        X=\bigoplus^N_{j=1}\text{\text{Ker}}(T-\xi_jI)\bigoplus\overline{Ran}\left(\prod\limits^N_{j=1}(T-\xi_jI)\right).\label{Mean- ergodic-theorem-1}
    \end{equation}  
\end{lemma}
Let $T:\calM\to\calM$ be such that it satisfies \eqref{thosefourconditions}. Let $1<p<\infty$ and $\xi_i\in E.$ By the mean ergodic theorem \cite{Ulrich}, $
M_n(\overline{\xi_i}T)x$
converges in $L_p(\mathcal M)$ for every $x\in L_p(\mathcal M)$. Let us define
\[
F_i x:=
\lim_{n\to\infty}
\frac{1}{n+1}
\sum_{k=0}^{n}
\overline{\xi_i}^{\,k}T^k(x).
\]
Then $F_i$ is a projection onto
$\operatorname{Ran}(F_i)
=
\text{Ker}(T-{\xi_i} I).$ Thus, for every $x\in L_p(\mathcal M)$, by equation \eqref{Mean- ergodic-theorem-1}
there exists a unique $x_0\in
\overline{\operatorname{Ran}Q_E(T)}$ such that
\begin{equation}\label{eq:x-decomposition}
x=\sum_{j=1}^{N}F_j(x)+x_0.
\end{equation}
Since
$F_j(x)\in\text{Ker}(T-\xi_jI)$,
we have
$TF_j(x)=\xi_jF_j(x)$, and consequently $T^nF_j(x)=\xi_j^nF_j(x),\ n\geq0.$ Hence it follows from equations \eqref{eq:x-decomposition} that
\begin{equation}\label{eq:peripheral-decomposition}
T^n(x)-\sum_{j=1}^{N}\xi_j^nF_j(x)
=
T^nx_0, \ \text{for all} \ n\geq 0.
\end{equation}
\begin{proposition}\label{T-n-converges-to-F-x} Let $1 < p < \infty$. Then $T^n(x)-\sum_{j=1}^{N}\xi_j^nF_j(x)$ converges to $0$ in $L_p(\mathcal{M})$ for any $x \in L_p(\mathcal{M})$.
\end{proposition}
\begin{proof} Let $x \in L_p(\mathcal{M})$. Then, by Lemma \eqref{lemma-for-ritt-E-form-numerical-range} $T$ is a $Ritt_E$ operator. Therefore 
\begin{align}\label{eq:QE-limit-L2}
\lim_{n \to \infty} \|T^{n}Q_{E}(T)(x)\|_p = 0.
\end{align}
Let $x_0\in\overline{\operatorname{Ran}Q_E(T)}.$
Given $\varepsilon>0$, there exists $y\in L_p(\mathcal M)$
such that $\|x_0-Q_E(T)y\|_p<\varepsilon.$

Since $T$ is a contraction on $L_p(\mathcal M)$, we have
\[
\begin{aligned}
\|T^nx_0\|_p
&\leq
\|T^n(x_0-Q_E(T)y)\|_p
+
\|T^nQ_E(T)y\|_p\\
&\leq
\varepsilon+
\|T^nQ_E(T)y\|_p.
\end{aligned}
\]
Using equation \eqref{eq:QE-limit-L2}, we obtain
$\limsup_{n\to\infty}\|T^nx_0\|_p
\leq\varepsilon.$ Since $\varepsilon>0$ is arbitrary, we have
$\lim_{n\to\infty}\|T^nx_0\|_p=0.$
Therefore, by equation \eqref{eq:peripheral-decomposition},
$\lim_{n\to\infty}
\left\|
T^n(x)-\sum_{j=1}^{N}\xi_j^nF_j(x)
\right\|_p
=0.$
This completes the proof of the lemma.
\end{proof}

 \begin{theorem}
     Let $T:\calM\to\calM$ be such that $T$ satisfies \eqref{thosefourconditions} with $E=\{\xi_1,\dots,\xi_N\}.$ Let $\xi_1=1.$ Then for all $x\in \text{Ker}(T-I) \oplus\overline{Ran}\left(\prod\limits^N_{j=1}(T-\xi_jI)\right),$ we have the following
     \begin{itemize}
 \item[(i)] If $1 < p <\infty$, then $T^n(x)$ converges b.a.u.\ to $F_1(x)$.
 \item[(ii)] If $2 < p < \infty$, and $T$ is $2$-positive, then $T^n(x)$ converges a.u.\ to $F_1(x)$.
 \end{itemize}
 \end{theorem}
 \begin{proof}
     Let $x\in L_p(\mathcal{M})$ such that $x=x_1+x_2$, where $x_1\in\text{Ker}(T-I)$ and $x_2\in\overline{Ran}\left(\prod\limits^N_{j=1}(T-\xi_jI)\right).$ It is immediate that $F_1(x_2)=0.$ Fix $k\in\mathbb{N}.$ One can see that $T^n(T^k(x)-F_1(x))=T^{n+k}(x)-F_1(x)$ for all $n\in\mathbb{N}.$
     Observe that $T^k(x)-F_1(x)=T^k(x_1+x_2)-F_1(x_1+x_2)=x_1+T^k(x_2)-x_1-F_1(x_2)=T^k(x_2).$ From equation \eqref{eq:QE-limit-L2} and Theorem \eqref{strong-ergodic-theorem-for-ritt-E}, we have $\lim_{k \rightarrow \infty}\lVert T^n(T^k(x)-F_1(x))\rVert_{L_p(\mathcal{M},\ell_{\infty})}=0.$ Therefore $(T^n(x)-F_1(x))_{n\geq 0}\in L_p (\mathcal{M},c_0).$
     Arguing similarly as above and using Corollary \eqref{coloum-maximal-inequality}, we have
     $(T^n(x)-F_1(x))_{n\geq 0}\in L_p (\mathcal{M},c^c_0).$ At the end we conclude our theorem using Proposition \eqref{ct1}.
 \end{proof}

\textbf{Acknowledgement:}
The second named author thanks the DST-INSPIRE Faculty Fellowship DST/INSPIRE/04/2020/001132,
Prime Minister Early Career Research Grant \\ ANRF/ECRG/2024/000699/PMS and
ANRF/ARGM/2025/000895/MTR. The first-named author gratefully acknowledges the hospitality of the The Institute of Mathematical Sciences, Chennai where a substantial part of this work was carried out. The authors thank Hiranmoy Ghorui and Suman Mondal for their careful reading of the manuscript and suggestions.

\textbf{Declaration of generative AI and AI-assisted technologies in the manuscript preparation process:}

\textbf{Statement:} During the preparation of this manuscript, the authors used ChatGPT-5.6 Sol and ChatGPT-6 solely to assist with English-language editing, manuscript organization and presentation, LaTeX formatting, and preliminary checks of mathematical consistency.
 The main ideas, arguments, and overall strategy of the paper were developed by the authors. All AI-assisted material was independently verified and revised where necessary, and the authors take full responsibility for the mathematical content and conclusions of the article.

\bibliographystyle{amsplain}
\nocite{*}
\bibliography{main}

@article {oualid-Ritt-E-operato,
    AUTHOR = {Bouabdillah, Oualid and Le Merdy, Christian},
     TITLE = {Polygonal functional calculus for operators with finite
              peripheral spectrum},
   JOURNAL = {Israel J. Math.},
  FJOURNAL = {Israel Journal of Mathematics},
    VOLUME = {263},
      YEAR = {2024},
    NUMBER = {2},
     PAGES = {517--551},
}

@article{MondalPalaiRay2026,
  author    = {Mondal, Suman and Palai, Subhajit and Ray, Samya Kumar},
  title     = {{$H^\infty$ Functional Calculus for a Commuting Pair of $\text{Ritt}_{\text{E}}$ Operators}},
  journal   = {Integral Equations and Operator Theory},
  year      = {2026},
  volume     = {98},
  number     = {3},
  pages      = {24}
}

@article {Blunk,
    AUTHOR = {Blunck, S\"onke},
     TITLE = {Analyticity and discrete maximal regularity on {$L_p$}-spaces},
   JOURNAL = {J. Funct. Anal.},
  FJOURNAL = {Journal of Functional Analysis},
    VOLUME = {183},
      YEAR = {2001},
    NUMBER = {1},
     PAGES = {211--230},
     
}

@article {Le-Merdy-Xu,
    AUTHOR = {Le Merdy, Christian and Xu, Quanhua},
     TITLE = {Maximal theorems and square functions for analytic operators
              on {$L^p$}-spaces},
   JOURNAL = {J. Lond. Math. Soc. (2)},
  FJOURNAL = {Journal of the London Mathematical Society. Second Series},
    VOLUME = {86},
      YEAR = {2012},
    NUMBER = {2},
     PAGES = {343--365},
    
}

@article {junge-Xu,
    AUTHOR = {Junge, Marius and Xu, Quanhua},
     TITLE = {Noncommutative maximal ergodic theorems},
   JOURNAL = {J. Amer. Math. Soc.},
  FJOURNAL = {Journal of the American Mathematical Society},
    VOLUME = {20},
      YEAR = {2007},
    NUMBER = {2},
     PAGES = {385--439},
      ISSN = {0894-0347,1088-6834},
   
}

@article {Non-commutative-vector-valued-L-p-space-pisier,
    AUTHOR = {Pisier, Gilles},
     TITLE = {Non-commutative vector valued {$L_p$}-spaces and completely
              {$p$}-summing maps},
   JOURNAL = {Ast\'erisque},
  FJOURNAL = {Ast\'erisque},
    NUMBER = {247},
      YEAR = {1998},
     PAGES = {vi+131},
  
}

@article {Doob-inequality-for-non-commutative-martingales-Junge-Marius,
    AUTHOR = {Junge, Marius},
     TITLE = {Doob's inequality for non-commutative martingales},
   JOURNAL = {J. Reine Angew. Math.},
  FJOURNAL = {Journal f\"ur die Reine und Angewandte Mathematik. [Crelle's
              Journal]},
    VOLUME = {549},
      YEAR = {2002},
     PAGES = {149--190},
 
}

@article {bouabdillah2024squarefunctionsassociatedritte,
    AUTHOR = {Bouabdillah, Oualid},
     TITLE = {Square functions associated with {${\rm Ritt}_E$} operators},
   JOURNAL = {Indag. Math. (N.S.)},
  FJOURNAL = {Koninklijke Nederlandse Akademie van Wetenschappen.
              Indagationes Mathematicae. New Series},
    VOLUME = {36},
      YEAR = {2025},
    NUMBER = {5},
     PAGES = {1417--1452},
     
}

@article {le-Merdy-Xu-q-variational-inequality,
    AUTHOR = {Le Merdy, Christian and Xu, Quanhua},
     TITLE = {Strong {$q$}-variation inequalities for analytic semigroups},
   JOURNAL = {Ann. Inst. Fourier (Grenoble)},
  FJOURNAL = {Universit\'e{} de Grenoble. Annales de l'Institut Fourier},
    VOLUME = {62},
      YEAR = {2012},
    NUMBER = {6},
     PAGES = {2069--2097},
    
}

@article {Hong-Tao-Ma,
    AUTHOR = {Hong, Guixiang and Ma, Tao},
     TITLE = {Vector valued {$q$}-variation for ergodic averages and
              analytic semigroups},
   JOURNAL = {J. Math. Anal. Appl.},
  FJOURNAL = {Journal of Mathematical Analysis and Applications},
    VOLUME = {437},
      YEAR = {2016},
    NUMBER = {2},
     PAGES = {1084--1100},

}

@article {Lance,
    AUTHOR = {Lance, E. Christopher},
     TITLE = {Ergodic theorems for convex sets and operator algebras},
   JOURNAL = {Invent. Math.},
  FJOURNAL = {Inventiones Mathematicae},
    VOLUME = {37},
      YEAR = {1976},
    NUMBER = {3},
     PAGES = {201--214},

}

@article{Flajolet1990SingularityAO,
  title={Singularity Analysis of Generating Functions},
  author={Philippe Flajolet and Andrew M. Odlyzko},
  journal={SIAM J. Discret. Math.},
  year={1990},
  volume={3},
  pages={216-240},
}

@article {Hong2024,
    AUTHOR = {Hong, Guixiang and Liu, Wei and Xu, Bang},
     TITLE = {Quantitative mean ergodic inequalities: power bounded
              operators acting on one single noncommutative {$L_p$} space},
   JOURNAL = {J. Funct. Anal.},
  FJOURNAL = {Journal of Functional Analysis},
    VOLUME = {286},
      YEAR = {2024},
    NUMBER = {1},
     PAGES = {Paper No. 110190, 37},
 
}

@article {Hong2023,
    AUTHOR = {Hong, Guixiang and Ray, Samya Kumar and Wang, Simeng},
     TITLE = {Maximal ergodic inequalities for some positive operators on
              noncommutative {$L_p$}-spaces},
   JOURNAL = {J. Lond. Math. Soc. (2)},
  FJOURNAL = {Journal of the London Mathematical Society. Second Series},
    VOLUME = {108},
      YEAR = {2023},
    NUMBER = {1},
     PAGES = {362--408},
  
}

@article {Hong2021,
    AUTHOR = {Hong, Guixiang and Liao, Ben and Wang, Simeng},
     TITLE = {Noncommutative maximal ergodic inequalities associated with
              doubling conditions},
   JOURNAL = {Duke Math. J.},
  FJOURNAL = {Duke Mathematical Journal},
    VOLUME = {170},
      YEAR = {2021},
    NUMBER = {2},
     PAGES = {205--246},
   
}

@article {Hong2020,
    AUTHOR = {Hong, Guixiang and Liu, Wei and Ma, Tao},
     TITLE = {Vector-valued {$q$}-variational inequalities for averaging
              operators and the {H}ilbert transform},
   JOURNAL = {Arch. Math. (Basel)},
  FJOURNAL = {Archiv der Mathematik},
    VOLUME = {115},
      YEAR = {2020},
    NUMBER = {4},
     PAGES = {423--433},
 
}

@article {Hong20201,
    AUTHOR = {Hong, Guixiang},
     TITLE = {Non-commutative ergodic averages of balls and spheres over
              {E}uclidean spaces},
   JOURNAL = {Ergodic Theory Dynam. Systems},
  FJOURNAL = {Ergodic Theory and Dynamical Systems},
    VOLUME = {40},
      YEAR = {2020},
    NUMBER = {2},
     PAGES = {418--436},

}

@article {Hong2017,
    AUTHOR = {Hong, Guixiang and Ma, Tao},
     TITLE = {Vector valued {$q$}-variation for differential operators and semigroups {I}},
   JOURNAL = {Math. Z.},
  FJOURNAL = {Mathematische Zeitschrift},
    VOLUME = {286},
      YEAR = {2017},
    NUMBER = {1-2},
     PAGES = {89--120},
 
}

@article {Hong2016,
    AUTHOR = {Hong, Guixiang and Ma, Tao},
     TITLE = {Vector valued {$q$}-variation for ergodic averages and
              analytic semigroups},
   JOURNAL = {J. Math. Anal. Appl.},
  FJOURNAL = {Journal of Mathematical Analysis and Applications},
    VOLUME = {437},
      YEAR = {2016},
    NUMBER = {2},
     PAGES = {1084--1100},

}

@article {Xu2025,
    AUTHOR = {Xu, Quanhua},
     TITLE = {Holomorphic functional calculus and vector-valued
              {L}ittlewood-{P}aley-{S}tein theory for semigroups},
   JOURNAL = {J. Eur. Math. Soc. (JEMS)},
  FJOURNAL = {Journal of the European Mathematical Society (JEMS)},
    VOLUME = {27},
      YEAR = {2025},
    NUMBER = {8},
     PAGES = {3191--3248},

}

@article {JungeLeMerdyXu06,
    AUTHOR = {Junge, Marius and Le Merdy, Christian and Xu, Quanhua},
     TITLE = {{$H^\infty$} functional calculus and square functions on
              noncommutative {$L^p$}-spaces},
   JOURNAL = {Ast\'erisque},
  FJOURNAL = {Ast\'erisque},
    NUMBER = {305},
      YEAR = {2006},
     PAGES = {vi+138},

}

@article {Xu2015,
    AUTHOR = {Xu, Quanhua},
     TITLE = {{$H^\infty$} functional calculus and maximal inequalities for
              semigroups of contractions on vector-valued {$L_p$}-spaces},
   JOURNAL = {Int. Math. Res. Not. IMRN},
  FJOURNAL = {International Mathematics Research Notices. IMRN},
      YEAR = {2015},
    NUMBER = {14},
     PAGES = {5715--5732},
 
}

@article {Xu2020,
    AUTHOR = {Xu, Quanhua},
     TITLE = {Vector-valued {L}ittlewood-{P}aley-{S}tein theory for
              semigroups {II}},
   JOURNAL = {Int. Math. Res. Not. IMRN},
  FJOURNAL = {International Mathematics Research Notices. IMRN},
      YEAR = {2020},
    NUMBER = {21},
     PAGES = {7769--7791},
 
}

@article {Bekjan2008,
    AUTHOR = {Bekjan, Turdebek N.},
     TITLE = {Noncommutative maximal ergodic theorems for positive
              contractions},
   JOURNAL = {J. Funct. Anal.},
  FJOURNAL = {Journal of Functional Analysis},
    VOLUME = {254},
      YEAR = {2008},
    NUMBER = {9},
     PAGES = {2401--2418},

}

@article {HongWangWang,
    AUTHOR = {Hong, Guixiang and Wang, Simeng and Wang, Xumin},
     TITLE = {Pointwise convergence of noncommutative {F}ourier series},
   JOURNAL = {Mem. Amer. Math. Soc.},
  FJOURNAL = {Memoirs of the American Mathematical Society},
    VOLUME = {302},
      YEAR = {2024},
    NUMBER = {1520},
     PAGES = {v+86},
 
}

@article{hong2025noncommutative,
  author    = {Hong, Guixiang and Lai, Xudong and Ray, Samya Kumar and Xu, Bang},
  title     = {Noncommutative Maximal Strong {$L_p$}-Estimates of {Calder\'on--Zygmund Operators}},
  journal   = {Israel Journal of Mathematics},
  year      = {2025},
  doi        = {10.1007/s11856-025-2883-2},
  pages   ={}
}

@book {I.G.Macdonald,
    AUTHOR = {Macdonald, I. G.},
     TITLE = {Symmetric functions and {H}all polynomials},
    SERIES = {Oxford Mathematical Monographs},
   EDITION = {Second},
      NOTE = {With contributions by A. Zelevinsky,
              Oxford Science Publications},
 PUBLISHER = {The Clarendon Press, Oxford University Press, New York},
      YEAR = {1995},
     PAGES = {x+475},
    
}

@article {Akcogluergodicthm,
    AUTHOR = {Akcoglu, M. A.},
     TITLE = {A pointwise ergodic theorem in {$L\sb{p}$}-spaces},
   JOURNAL = {Canadian J. Math.},
  FJOURNAL = {Canadian Journal of Mathematics. Journal Canadien de
              Math\'ematiques},
    VOLUME = {27},
      YEAR = {1975},
    NUMBER = {5},
     PAGES = {1075--1082},

}

@article {ZeqianXuZhi,
    AUTHOR = {Chen, Zeqian and Xu, Quanhua and Yin, Zhi},
     TITLE = {Harmonic analysis on quantum tori},
   JOURNAL = {Comm. Math. Phys.},
  FJOURNAL = {Communications in Mathematical Physics},
    VOLUME = {322},
      YEAR = {2013},
    NUMBER = {3},
     PAGES = {755--805},

}

@article {Bourgain,
    AUTHOR = {Bourgain, Jean},
     TITLE = {On pointwise ergodic theorems for arithmetic sets},
   JOURNAL = {C. R. Acad. Sci. Paris S\'er. I Math.},
  FJOURNAL = {Comptes Rendus des S\'eances de l'Acad\'emie des Sciences.
              S\'erie I. Math\'ematique},
    VOLUME = {305},
      YEAR = {1987},
    NUMBER = {10},
     PAGES = {397--402},
  
}

@incollection {Bellow-Calderon,
    AUTHOR = {Bellow, Alexandra and Calder\'on, Alberto P.},
     TITLE = {A weak-type inequality for convolution products},
 BOOKTITLE = {Harmonic analysis and partial differential equations
              ({C}hicago, {IL}, 1996)},
    SERIES = {Chicago Lectures in Math.},
     PAGES = {41--48},
 PUBLISHER = {Univ. Chicago Press, Chicago, IL},
      YEAR = {1999},
  
}

@article{Stein-ergodic-theorem,
 ISSN = {00278424, 10916490},
 URL = {http://www.jstor.org/stable/71067},
 author = {E. M. Stein},
 journal = {Proceedings of the National Academy of Sciences of the United States of America},
 number = {12},
 pages = {1894--1897},
 publisher = {National Academy of Sciences},
 title = {On the Maximal Ergodic Theorem},
 urldate = {2026-05-13},
 volume = {47},
 year = {1961}
}

@article {Wiener-ergodic-theorem,
    AUTHOR = {Wiener, Norbert},
     TITLE = {The ergodic theorem},
   JOURNAL = {Duke Math. J.},
  FJOURNAL = {Duke Mathematical Journal},
    VOLUME = {5},
      YEAR = {1939},
    NUMBER = {1},
     PAGES = {1--18},
    
}

@article{birkhoff1931proof,
  title={Proof of the ergodic theorem},
  author={Birkhoff, George D},
  journal={Proceedings of the National Academy of Sciences},
  volume={17},
  number={12},
  pages={656--660},
  year={1931},
  publisher={National Academy of Sciences}
}

@article{neumann1932proof,
  title={Proof of the quasi-ergodic hypothesis},
  author={Neumann, J. v},
  journal={Proceedings of the National Academy of Sciences},
  volume={18},
  number={1},
  pages={70--82},
  year={1932},
  publisher={National Academy of Sciences}
}

@article {DunfordSchwartz,
    AUTHOR = {Dunford, Nelson and Schwartz, J. T.},
     TITLE = {Convergence almost everywhere of operator averages},
   JOURNAL = {J. Rational Mech. Anal.},
  FJOURNAL = {Journal of Rational Mechanics and Analysis},
    VOLUME = {5},
      YEAR = {1956},
     PAGES = {129--178},
     
}

@article{Delyon-Bernard,
     author = {Delyon, Bernard and Delyon, Fran\c{c}ois},
     title = {Generalization of von {Neumann's} spectral sets and integral representation of operators},
     journal = {Bulletin de la Soci\'et\'e Math\'ematique de France},
     pages = {25--41},
     year = {1999},
     publisher = {Soci\'et\'e math\'ematique de France},
     volume = {127},
     number = {1},
     
}

@article {Bang-Xu,
    AUTHOR = {Hong, Guixiang and Lai, Xudong and Xu, Bang},
     TITLE = {Maximal singular integral operators acting on noncommutative
              {$L_p$}-spaces},
   JOURNAL = {Math. Ann.},
  FJOURNAL = {Mathematische Annalen},
    VOLUME = {386},
      YEAR = {2023},
    NUMBER = {1-2},
     PAGES = {375--414},
     
}

@incollection {PisierXuNclp,
    AUTHOR = {Pisier, Gilles and Xu, Quanhua},
     TITLE = {Non-commutative {$L^p$}-spaces},
 BOOKTITLE = {Handbook of the geometry of {B}anach spaces, {V}ol.\ 2},
     PAGES = {1459--1517},
 PUBLISHER = {North-Holland, Amsterdam},
      YEAR = {2003},
    
}

@article {pisierXuMartinglevariation,
    AUTHOR = {Pisier, Gilles and Xu, Quan Hua},
     TITLE = {The strong {$p$}-variation of martingales and orthogonal
              series},
   JOURNAL = {Probab. Theory Related Fields},
  FJOURNAL = {Probability Theory and Related Fields},
    VOLUME = {77},
      YEAR = {1988},
    NUMBER = {4},
     PAGES = {497--514},
   
}

@article{Yeadon-Noncommutative-ergodic-theorem,
author = {Yeadon, F. J.},
title = {Ergodic Theorems for Semifinite Von Neumann Algebras. {I}},
journal = {Journal of the London Mathematical Society},
volume = {s2-16},
number = {2},
pages = {326-332},
year = {1977}
}

@article {Noncommutative-ergodic-theorem-kummmerer,
    AUTHOR = {K\"ummerer, Burkhard},
     TITLE = {A non-commutative individual ergodic theorem},
   JOURNAL = {Invent. Math.},
  FJOURNAL = {Inventiones Mathematicae},
    VOLUME = {46},
      YEAR = {1978},
    NUMBER = {2},
     PAGES = {139--145},
   
}

@article {Conze,
    AUTHOR = {Conze, J.-P. and Dang-Ngoc, N.},
     TITLE = {Ergodic theorems for noncommutative dynamical systems},
   JOURNAL = {Invent. Math.},
  FJOURNAL = {Inventiones Mathematicae},
    VOLUME = {46},
      YEAR = {1978},
    NUMBER = {1},
     PAGES = {1--15},
 
}

@article {Cuculescu,
    AUTHOR = {Cuculescu, I.},
     TITLE = {Martingales on von {N}eumann algebras},
   JOURNAL = {J. Multivariate Anal.},
  FJOURNAL = {Journal of Multivariate Analysis},
    VOLUME = {1},
      YEAR = {1971},
    NUMBER = {1},
     PAGES = {17--27},

}

@article {Cadilhac-Conde-Alonso-Parcet-2022,
    AUTHOR = {Cadilhac, L\'eonard and Conde-Alonso, Jos\'e{} M. and Parcet,
              Javier},
     TITLE = {Spectral multipliers in group algebras and noncommutative
              {C}alder\'on-{Z}ygmund theory},
   JOURNAL = {J. Math. Pures Appl. (9)},
  FJOURNAL = {Journal de Math\'ematiques Pures et Appliqu\'ees. Neuvi\`eme
              S\'erie},
    VOLUME = {163},
      YEAR = {2022},
     PAGES = {450--472},
  
}

@article {CunnyRttWeaktype,
    AUTHOR = {Cuny, Christophe},
     TITLE = {On the {R}itt property and weak type maximal inequalities for
              convolution powers on {$\ell^1(\mathbb{Z})$}},
   JOURNAL = {Studia Math.},
  FJOURNAL = {Studia Mathematica},
    VOLUME = {235},
      YEAR = {2016},
    NUMBER = {1},
     PAGES = {47--85},
   
}

@article {varitaionalarthimeticsetBourgain,
    AUTHOR = {Bourgain, Jean},
     TITLE = {Pointwise ergodic theorems for arithmetic sets},
      NOTE = {With an appendix by the author, Harry Furstenberg, Yitzhak
              Katznelson and Donald S. Ornstein},
   JOURNAL = {Inst. Hautes \'Etudes Sci. Publ. Math.},
  FJOURNAL = {Institut des Hautes \'Etudes Scientifiques. Publications
              Math\'ematiques},
    NUMBER = {69},
      YEAR = {1989},
     PAGES = {5--45},

}

@article {JonesSeegerWrightvariation,
    AUTHOR = {Jones, Roger L. and Seeger, Andreas and Wright, James},
     TITLE = {Strong variational and jump inequalities in harmonic analysis},
   JOURNAL = {Trans. Amer. Math. Soc.},
  FJOURNAL = {Transactions of the American Mathematical Society},
    VOLUME = {360},
      YEAR = {2008},
    NUMBER = {12},
     PAGES = {6711--6742},

}

@article {TaoTorreaXuvariation,
    AUTHOR = {Ma, Tao and Torrea, Jos\'e{} Luis and Xu, Quanhua},
     TITLE = {Weighted variation inequalities for differential operators and
              singular integrals},
   JOURNAL = {J. Funct. Anal.},
  FJOURNAL = {Journal of Functional Analysis},
    VOLUME = {268},
      YEAR = {2015},
    NUMBER = {2},
     PAGES = {376--416},

}

@article {MirekSteinTrojanvariation,
    AUTHOR = {Mirek, Mariusz and Stein, Elias M. and Trojan, Bartosz},
     TITLE = {{$\ell^p(\Bbb Z^d) $}-estimates for discrete operators of
              {R}adon type: variational estimates},
   JOURNAL = {Invent. Math.},
  FJOURNAL = {Inventiones Mathematicae},
    VOLUME = {209},
      YEAR = {2017},
    NUMBER = {3},
     PAGES = {665--748},

}

@article {OberlinSeegerTaoThieleWrightVariation,
    AUTHOR = {Oberlin, Richard and Seeger, Andreas and Tao, Terence and
              Thiele, Christoph and Wright, James},
     TITLE = {A variation norm {C}arleson theorem},
   JOURNAL = {J. Eur. Math. Soc. (JEMS)},
  FJOURNAL = {Journal of the European Mathematical Society (JEMS)},
    VOLUME = {14},
      YEAR = {2012},
    NUMBER = {2},
     PAGES = {421--464},

}

@article{ray2025noncommutative,
  title={Noncommutative ergodic theorems for action of semisimple Lie groups},
  author={Ray, Samya Kumar and Guixiang Hong},
  journal={arXiv preprint arXiv:2508.07444},
  year={2025},
}

@article {Operatorvaluedhardyspacemei,
    AUTHOR = {Mei, Tao},
     TITLE = {Operator valued {H}ardy spaces},
   JOURNAL = {Mem. Amer. Math. Soc.},
  FJOURNAL = {Memoirs of the American Mathematical Society},
    VOLUME = {188},
      YEAR = {2007},
    NUMBER = {881},
     PAGES = {vi+64},

}

@article {peripheralspectrumofpositiveoperators,
    AUTHOR = {Gl\"uck, Jochen},
     TITLE = {On the peripheral spectrum of positive operators},
   JOURNAL = {Positivity},
  FJOURNAL = {Positivity. An International Mathematics Journal Devoted to
              Theory and Applications of Positivity},
    VOLUME = {20},
      YEAR = {2016},
    NUMBER = {2},
     PAGES = {307--336},

}

@book {Ulrich,
    AUTHOR = {Krengel, Ulrich},
     TITLE = {Ergodic theorems},
    SERIES = {De Gruyter Studies in Mathematics},
    VOLUME = {6},
      NOTE = {With a supplement by Antoine Brunel},
 PUBLISHER = {Walter de Gruyter \& Co., Berlin},
      YEAR = {1985},
     PAGES = {viii+357},
      ISBN = {3-11-008478-3},
   MRCLASS = {28-02 (28Dxx 47A35)},
  MRNUMBER = {797411},
MRREVIEWER = {E.\ Flytzanis},
       DOI = {10.1515/9783110844641},
       URL = {https://doi.org/10.1515/9783110844641},}

@book{GrahamKnuthPatashnik1989,
  author    = {Ronald L. Graham and Donald E. Knuth and Oren Patashnik},
  title     = {Concrete Mathematics: A Foundation for Computer Science},
  publisher = {Addison-Wesley},
  address   = {Reading, Massachusetts},
  year      = {1989}
}

@article {Fackler-Stephan,
    AUTHOR = {Fackler, Stephan},
     TITLE = {Regularity of semigroups via the asymptotic behaviour at zero},
   JOURNAL = {Semigroup Forum},
  FJOURNAL = {Semigroup Forum},
    VOLUME = {87},
      YEAR = {2013},
    NUMBER = {1},
     PAGES = {1--17},
      ISSN = {0037-1912,1432-2137},
   MRCLASS = {47D03 (47H09)},
  MRNUMBER = {3079770},
MRREVIEWER = {Filippo\ Dell'Oro},
       DOI = {10.1007/s00233-013-9466-y},
       URL = {https://doi.org/10.1007/s00233-013-9466-y},
}

@article {junge-maximal-theorem-type-in-nc-L-p,
    AUTHOR = {Defant, Andreas and Junge, Marius},
     TITLE = {Maximal theorems of {M}enchoff-{R}ademacher type in
              non-commutative {$L_q$}-spaces},
   JOURNAL = {J. Funct. Anal.},
  FJOURNAL = {Journal of Functional Analysis},
    VOLUME = {206},
      YEAR = {2004},
    NUMBER = {2},
     PAGES = {322--355},
      ISSN = {0022-1236,1096-0783},
   MRCLASS = {46L52 (42B35)},
  MRNUMBER = {2021850},
MRREVIEWER = {Narcisse\ Randrianantoanina},
       DOI = {10.1016/j.jfa.2002.07.001},
       URL = {https://doi.org/10.1016/j.jfa.2002.07.001},
}

@article{musat03interpolationbmolp,
  author    = {Musat, Magdalena},
  title     = {Interpolation between non-commutative {BMO} and non-commutative {$L_p$}-spaces},
  journal   = {Journal of Functional Analysis},
  volume    = {202},
  year      = {2003},
  number    = {1},
  pages     = {195--225}
}

@book{jajte85nclimitbook,
  author    = {Jajte, Ryszard},
  title     = {Strong Limit Theorems in Noncommutative Probability},
  series    = {Lecture Notes in Mathematics},
  volume    = {1110},
  publisher = {Springer-Verlag},
  address   = {Berlin},
  year      = {1985},
  pages     = {vi+152}
}

@book{Stein-ergodic-theorem-book,
 ISBN = {9780691080673},
 URL = {http://www.jstor.org/stable/j.ctt1bgzbcx},
 author = {Eias M. Stein},
 publisher = {Princeton University Press},
 title = {Topics in Harmonic Analysis Related to the Littlewood-Paley Theory. (AM-63)},
 urldate = {2026-08-29},
 year = {1970}
}

@article {parasar-samya,
    AUTHOR = {Mohanty, Parasar and Ray, Samya Kumar},
     TITLE = {On joint functional calculus for {R}itt operators},
   JOURNAL = {Integral Equations Operator Theory},
  FJOURNAL = {Integral Equations and Operator Theory},
    VOLUME = {91},
      YEAR = {2019},
    NUMBER = {2},
     PAGES = {Paper No. 14, 18},
      ISSN = {0378-620X,1420-8989},
   MRCLASS = {47A60 (47A20)},
  MRNUMBER = {3928691},
MRREVIEWER = {Andrei\ Halanay},
       DOI = {10.1007/s00020-019-2513-7},
       URL = {https://doi.org/10.1007/s00020-019-2513-7},
}

@misc{hong2026lampertioperatorsdilationtheory,
      title={Lamperti Operators, Dilation Theory, and Applications in Noncommutative Ergodic Theory}, 
      author={Guixiang Hong and Wei Liu and Samya Kumar Ray and Bang Xu},
      year={2026},
      eprint={2601.02174},
      archivePrefix={arXiv},
      primaryClass={math.OA},
      url={https://arxiv.org/abs/2601.02174}, 
}

@book{rudin,
  author    = {Rudin, Walter},
  title     = {Principles of Mathematical Analysis},
  edition   = {3},
  publisher = {McGraw-Hill},
  address   = {New York},
  year      = {1976}
}

@book {Lemerdybook,
    AUTHOR = {Le Merdy, Christian},
     TITLE = {Ritt operators and their functional calculus},
    SERIES = {Operator Theory: Advances and Applications},
    VOLUME = {312},
 PUBLISHER = {Birkh\"{a}user/Springer, Cham},
      YEAR = {2026},
     PAGES = {xi+428},
     }
\end{document}